\documentclass[11pt]{amsart}

\usepackage[all]{xy}
\usepackage{mathrsfs}
\usepackage{amsmath, amsthm, amssymb, amscd, amsgen,amsxtra,amsfonts, amsbsy}
\allowdisplaybreaks
\usepackage[all]{xy}
\usepackage{verbatim}

\usepackage{tikz}
\usepackage{pgfplots}
\usepackage{tikz-cd}
\usepackage{etex}
\usepackage{calligra,mathrsfs}
\usepackage{color}
\definecolor{shadecolor}{gray}{0.875}
\definecolor{dblue}{rgb}{0,0,.6}
\usepackage[colorlinks=true, linkcolor=dblue, citecolor=dblue, filecolor = dblue, menucolor = dblue, urlcolor = dblue]{hyperref}
\usepackage{mathtools}
\usepackage{extarrows}
\usepackage{ifthen}
\usepackage{bm}

\usepackage{graphicx}
\usepackage{caption,rotating}
\usepackage{color}
\usepackage{subcaption}

\usepackage{times}
\newcommand{\mathds}[1]{{\mathbb #1}}

\numberwithin{equation}{subsection}

\begin{document}
%
%
%
\theoremstyle{definition}
\newtheorem{Definition}{Definition}[section]
\newtheorem*{Definitionx}{Definition}
\newtheorem{Convention}{Definition}[section]
\newtheorem{Construction}[Definition]{Construction}
\newtheorem{Example}[Definition]{Example}
\newtheorem{Examples}[Definition]{Examples}
\newtheorem{Remark}[Definition]{Remark}
\newtheorem*{Remarkx}{Remark}
\newtheorem{Remarks}[Definition]{Remarks}
\newtheorem{Caution}[Definition]{Caution}
\newtheorem{Conjecture}[Definition]{Conjecture}
\newtheorem*{Conjecturex}{Conjecture}
\newtheorem{Question}[Definition]{Question}
\newtheorem*{Questionx}{Question}
\newtheorem*{Acknowledgements}{Acknowledgements}
\newtheorem*{Notation}{Notation}
\newtheorem*{Organization}{Organization}
\newtheorem*{Disclaimer}{Disclaimer}
\theoremstyle{plain}
\newtheorem{Theorem}[Definition]{Theorem}
\newtheorem*{Theoremx}{Theorem}

\newtheorem*{thm: regeneration}{Theorem \ref{thm: regeneration}}
\newtheorem*{thm: curves on complex k3}{Theorem \ref{thm: curves on complex k3}}
\newtheorem*{cor: curves on pos char k3}{Corollary \ref{cor: curves on pos char k3}}
\newtheorem*{def: regeneration}{Definition \ref{def: regeneration}}
\newtheorem*{cor: reduction}{Proposition \ref{cor: reduction}}
\newtheorem*{thm: nodal curves}{Theorem \ref{thm: nodal curves}}
\newtheorem*{thm: curves on complex k3 picard rank r}{Theorem \ref{thm: curves on complex k3 picard rank r}}
\newtheorem*{thmA}{Theorem A}
\newtheorem*{thmB}{Theorem B}
\newtheorem*{thmC}{Theorem C}
\newtheorem*{thmD}{Theorem D}

\newtheorem{Proposition}[Definition]{Proposition}
\newtheorem*{Propositionx}{Proposition}
\newtheorem{Lemma}[Definition]{Lemma}
\newtheorem{Corollary}[Definition]{Corollary}
\newtheorem*{Corollaryx}{Corollary}
\newtheorem{Fact}[Definition]{Fact}
\newtheorem{Facts}[Definition]{Facts}
\newtheoremstyle{voiditstyle}{3pt}{3pt}{\itshape}{\parindent}%
{\bfseries}{.}{ }{\thmnote{#3}}%
\theoremstyle{voiditstyle}
\newtheorem*{VoidItalic}{}
\newtheoremstyle{voidromstyle}{3pt}{3pt}{\rm}{\parindent}%
{\bfseries}{.}{ }{\thmnote{#3}}%
\theoremstyle{voidromstyle}
\newtheorem*{VoidRoman}{}

\newenvironment{specialproof}[1][\proofname]{\noindent\textit{#1.} }{\qed\medskip}
\newcommand{\blowup}{\rule[-3mm]{0mm}{0mm}}
\newcommand{\cal}{\mathcal}
\newcommand{\Aff}{{\mathds{A}}}
\newcommand{\BB}{{\mathds{B}}}
\newcommand{\CC}{{\mathds{C}}}
\newcommand{\EE}{{\mathds{E}}}
\newcommand{\FF}{{\mathds{F}}}
\newcommand{\GG}{{\mathds{G}}}
\newcommand{\HH}{{\mathds{H}}}
\newcommand{\NN}{{\mathds{N}}}
\newcommand{\ZZ}{{\mathds{Z}}}
\newcommand{\PP}{{\mathds{P}}}
\newcommand{\QQ}{{\mathds{Q}}}
\newcommand{\RR}{{\mathds{R}}}
\newcommand{\Liea}{{\mathfrak a}}
\newcommand{\Lieb}{{\mathfrak b}}
\newcommand{\Lieg}{{\mathfrak g}}
\newcommand{\Liem}{{\mathfrak m}}
\newcommand{\ideala}{{\mathfrak a}}
\newcommand{\idealb}{{\mathfrak b}}
\newcommand{\idealg}{{\mathfrak g}}
\newcommand{\idealm}{{\mathfrak m}}
\newcommand{\idealp}{{\mathfrak p}}
\newcommand{\idealq}{{\mathfrak q}}
\newcommand{\idealI}{{\cal I}}
\newcommand{\lin}{\sim}
\newcommand{\num}{\equiv}
\newcommand{\dual}{\ast}
\newcommand{\iso}{\cong}
\newcommand{\homeo}{\approx}
\newcommand{\mm}{{\mathfrak m}}
\newcommand{\pp}{{\mathfrak p}}
\newcommand{\qq}{{\mathfrak q}}
\newcommand{\rr}{{\mathfrak r}}
\newcommand{\pP}{{\mathfrak P}}
\newcommand{\qQ}{{\mathfrak Q}}
\newcommand{\rR}{{\mathfrak R}}
\newcommand{\OO}{{\cal O}}
\newcommand{\numero}{{n$^{\rm o}\:$}}
\newcommand{\mf}[1]{\mathfrak{#1}}
\newcommand{\mc}[1]{\mathcal{#1}}
\newcommand{\into}{{\hookrightarrow}}
\newcommand{\onto}{{\twoheadrightarrow}}
\newcommand{\Spec}{{\rm Spec}\:}
\newcommand{\BigSpec}{{\rm\bf Spec}\:}
\newcommand{\Spf}{{\rm Spf}\:}
\newcommand{\Proj}{{\rm Proj}\:}
\newcommand{\Pic}{{\rm Pic }}
\newcommand{\MW}{{\rm MW }}
\newcommand{\Br}{{\rm Br}}
\newcommand{\NS}{{\rm NS}}
\newcommand{\Sym}{{\mathfrak S}}
\newcommand{\Aut}{{\rm Aut}}
\newcommand{\Autp}{{\rm Aut}^p}
\newcommand{\ord}{{\rm ord}}
\newcommand{\coker}{{\rm coker}\,}
\newcommand{\chara}{{\rm char}}
\newcommand{\divisor}{{\rm div}}
\newcommand{\Def}{{\rm Def}}
\newcommand{\rank}{\mathop{\mathrm{rank}}\nolimits}
\newcommand{\Ext}{\mathop{\mathrm{Ext}}\nolimits}
\newcommand{\EXT}{\mathop{\mathscr{E}{\kern -2pt {xt}}}\nolimits}
\newcommand{\Hom}{\mathop{\mathrm{Hom}}\nolimits}
\newcommand{\HOM}{\mathop{\mathscr{H}{\kern -3pt {om}}}\nolimits}
\newcommand{\calA}{\mathscr{A}}
\newcommand{\calC}{\mathscr{C}}
\newcommand{\calH}{\mathscr{H}}
\newcommand{\calL}{\mathscr{L}}
\newcommand{\calM}{\mathscr{M}}
\newcommand{\calN}{\mathscr{N}}
\newcommand{\calX}{\mathscr{X}}
\newcommand{\calK}{\mathscr{K}}
\newcommand{\calD}{\mathscr{D}}
\newcommand{\calY}{\mathscr{Y}}
\newcommand{\piet}{{\pi_1^{\rm \acute{e}t}}}
\newcommand{\Het}[1]{{H_{\rm \acute{e}t}^{{#1}}}}
\newcommand{\Hfl}[1]{{H_{\rm fl}^{{#1}}}}
\newcommand{\Hcris}[1]{{H_{\rm cris}^{{#1}}}}
\newcommand{\HdR}[1]{{H_{\rm dR}^{{#1}}}}
\newcommand{\hdR}[1]{{h_{\rm dR}^{{#1}}}}
\newcommand{\loc}{{\rm loc}}
\newcommand{\et}{{\rm \acute{e}t}}
\newcommand{\defin}[1]{{\bf #1}}
\newcommand{\blue}{\textcolor{blue}}
\newcommand{\red}{\textcolor{red}}

\renewcommand{\HH}{{\rm{H}}}

\title{Curves on K3 surfaces}

\author{Xi Chen}
\address{632 Central Academic Building, University of Alberta, Edmonton, Alberta T6G 2G1, Canada}
\email{xichen@math.ualberta.ca}

\author{Frank Gounelas}
\address{Georg-August-Universit\"at G\"ottingen, Fakult\"at f\"ur Mathematik und Informatik, Bunsenstr. 3-5, 37073 G\"ottingen, Germany}
\email{gounelas@mathematik.uni-goettingen.de}

\author{Christian Liedtke}
\address{TU M\"unchen, Zentrum Mathematik - M11, Boltzmannstr. 3, 85748 Garching bei M\"unchen, Germany}
\email{liedtke@ma.tum.de}

\date{October 16, 2022}
\makeatletter
\@namedef{subjclassname@2020}{%
  \textup{2020} Mathematics Subject Classification}
\makeatother

\begin{abstract}
We complete the remaining cases of the conjecture predicting existence of infinitely many rational curves on K3 surfaces
in characteristic zero, prove almost all cases in positive characteristic and improve the proofs of the previously known
cases. To achieve this, we introduce two new techniques in the deformation theory of curves on K3 surfaces.
\begin{enumerate}
    \item \textit{Regeneration}, a process opposite to specialisation, which preserves the geometric genus and does not
    require the class of the curve to extend.
    \item The \textit{marked point trick}, which allows a controlled degeneration of rational curves to integral ones in
    certain situations.
\end{enumerate}
Combining the two proves existence of integral curves of unbounded degree of any geometric genus $g$ for any
projective K3 surface in characteristic zero.
\end{abstract}

\maketitle

\setcounter{tocdepth}{1}
\tableofcontents

\section{Introduction}

There is extensive literature on the existence and deformations of curves on projective K3 surfaces. In characteristic zero,
existence of rational curves was settled in Mori--Mukai \cite{morimukai} following an argument originally due to
Bogomolov and Mumford. It has since been expected that the following stronger statement should be true.
\begin{Conjecture}
\label{conj: inf many rational curves}
    Let $X$ be a projective K3 surface over an algebraically closed field. Then $X$ contains infinitely many rational
    curves.
\end{Conjecture}
This conjecture has been established in many important cases: for very general K3 surfaces in characteristic zero
\cite{chen, C-L}, for K3 surfaces with odd Picard rank \cite{liliedtke} (building on ideas of \cite{bht}), for elliptic
K3 surfaces and K3 surfaces with infinite automorphism groups \cite{btdensity, Tayou}, and for some special K3 surfaces
\cite{charles}. We refer to \cite{benoist, huybrechts} for a survey and further details. 

The main result of this paper is the following (see Theorem \ref{thm: curves on complex k3} and Corollary \ref{cor:char
p cor} for more precise statements, in particular in characteristic $2,3$ where almost all cases are also dealt with).

\begin{thmA}
Let $X$ be a projective K3 surface over an algebraically closed field $k$ with $\chara(k)=p\neq2,3$ and let $g\geq0$
be an integer. Assume that either $p=0$ or that $g\leq1$. Then there exists a sequence of integral curves $C_n\subset X$
of geometric genus $g$ such that for any ample line bundle $H$ on $X$  
\[
\lim_{n\to\infty} \deg_H C_n = \lim_{n\to\infty} H C_n = \infty.
\]
\end{thmA}

If $g=0$ then this establishes Conjecture \ref{conj: inf many rational curves} in almost all remaining cases.
To prove this we introduce two new techniques in the deformation
theory of curves in families of K3 surfaces, summarised in the next two subsections. The first leads to a more concise
proof of the odd rank case of Li--Liedtke. Then the aforementioned work of Bogomolov--Tschinkel combined with lattice
classification results of Vinberg and Nikulin reduce the theorem to two K3 lattices of rank four and infinitely many of
rank two, which we deal with separately. 

Both techniques rely on the existence of integral rational curves in high enough multiples of any ample class on a
generic K3 surface of Picard rank two, a result obtained in Appendix \ref{appendix} (see Section \ref{subsect:generic} for a
summary of the main results), which strengthens previous work of the first named author \cite{chen}. 

\subsection{The regeneration technique}

The first technique we introduce in this article provides a new method for constructing curves of prescribed geometric
genus in families of K3 surfaces. Let $\calX\to S$ be a family of K3 surfaces over the spectrum of a discrete valuation ring (DVR)
with algebraically closed residue field $k$, and denote by $\calX_0$, $\calX_\eta$, $\overline{\calX_\eta}$
the special fibre, the generic fibre and the geometric generic fibre, respectively.

\begin{Definitionx}
A {\em regeneration} of an integral curve $C_0\subset\calX_0$ is a geometrically integral curve $D_\eta\subset
\calX_\eta$ such that
\begin{enumerate}
\item the geometric genus of $C_0$ and $D_\eta$ is the same, and
\item the specialisation of $D_\eta$ to $\calX_0$ contains $C_0$.
\end{enumerate}
\end{Definitionx}

Thus, a regeneration of a curve can be viewed as the ``inverse'' of a specialisation with fixed geometric genus. It is
easy to see that the specialisation of $D_\eta$ is equal to $C_0$ and a (possibly empty) union of rational curves. The
main point of a regeneration compared to a deformation is that we do {\em not} require the class of the line bundle
$\OO_{\calX_0}(C_0)$ to extend from $\calX_0$ to $\overline{\calX_\eta}$. Therefore, regenerations are more flexible than
deformations. The price we have to pay for this flexibility is that we must attach rational curves. The
main result on regenerations that we prove in this article is the following.

\begin{thmB}
    Let $\calX\to S$ be as above so that $\calX_0$ is not uniruled, and let $C_0\subset{\calX_0}$ be an integral curve
    that deforms in the expected dimension in $\calX_0$, e.g., $\chara(k)=0$ or $C_0$ is at worst nodal. Then the
    curve $C_0$ admits at least one regeneration on $\overline{\calX_\eta}$.
\end{thmB}

Again, we stress that we do {\em not} require that the class of the curve $C_0$ in $\Pic({\calX}_0)$ extends to
$\Pic(\overline{\calX_\eta})$. On the other hand, if it does extend, then it is more or less well-known that one can use
moduli spaces of stable maps to deform $C_0$ to $\overline{\calX_\eta}$ while preserving the geometric genus.
We refer to Section \ref{subsec: stable maps} and Lemma \ref{lem: assertion two} for further details.

The condition of deforming in the expected dimension holds for every integral curve $C_0$ if the characteristic of the
residue field $k$ is zero, or if the singularities of $C_0$ are mild, e.g., any nodal curve has this property even in a
unirational K3 surface (see Proposition \ref{prop:defexpdim}). In Section \ref{sec: applications}, we show by
example that this assumption cannot be dropped. If $C_0$ is a rigid rational curve, then such a theorem was already
envisioned in \cite{bht} and building on their ideas, \cite{liliedtke} already proved some important special cases of
the above theorem.

There are two applications of regeneration in this paper. First, it is a key step in the proof of Theorem \ref{thm:
curves on complex k3} in the odd rank case (see Corollary \ref{cor:odd rank}). Second, it is key to the construction of rational
curves on K3 surfaces, namely towards Conjecture \ref{conj: inf many rational curves} - in fact, this was the motivation
of \cite{bht, liliedtke} where weaker versions of the regeneration theorem were developed. In particular, the
regeneration theorem provides a slightly cleaner and simpler proof of the main theorem of \cite{liliedtke}. As another
corollary of our regeneration technique we have the following specialisation result for Conjecture \ref{conj: inf many
rational curves}.

\begin{Corollaryx}
If $\calX_0$ is not uniruled and has infinitely many rational curves, then so does
$\overline{\calX_{\eta}}$.
\end{Corollaryx}

Note that Conjecture \ref{conj: inf many rational curves} in characteristic zero is known to follow from the conjecture
over $\overline{\QQ}$. The above reduction now shows that it suffices to prove that for a K3 surface $X$ over $\overline{\QQ}$,
one need only find one prime $\idealp$ of good reduction such that $X_{\overline{\idealp}}$ is not uniruled and has
infinitely many rational curves; we explain this and other related corollaries in Section \ref{sec: applications}. Our
proof of the conjecture in the remaining cases follows a different approach though, outlined in the next subsection.

\subsection{The marked point trick}

By a theorem of Li and the third named author \cite{liliedtke} (see also Theorem \ref{thm: ll}), a K3 surface of Picard
rank one has infinitely many rational curves.  Moreover, by work of Bogomolov--Tschinkel (see Section \ref{sec:
existence}) and Vinberg, the remaining cases of Conjecture \ref{conj: inf many rational curves} are K3 surfaces of Picard rank 2 and
4 that are neither elliptic nor have infinitely many automorphisms. The difficulty is that infinitely many rational
curves on a very general K3 surface could all specialise to finitely many on one of Picard rank 2.

To address this problem, we develop a technique of controlled degeneration, which is a consequence of what we call the
\textit{marked point trick} in Proposition \ref{prop: marked point trick g=0}. A simplified description of this is the
following. Assume that we are given a diagram
\[
\begin{tikzcd}
    \mathcal{C}\ar{r}{F}\ar{d}[left]{g} & \calX\ar{d}{f} \\
    V\ar{r}{\pi} & U
\end{tikzcd}
\]
where $U,V$ are smooth projective surfaces, $f$ is a smooth proper family of surfaces, $g$ is a generically smooth family of stable
genus 0 curves mapping to fibres of $f$ with $V\to U$ generically finite, and assume that for some point $v\in
\pi^{-1}(u)$ the curve $F(\mathcal{C}_v)\subset\calX_u$ is reducible. We mark now four points on fibres of
$g$ and pull back the bounday $\partial\overline{\mathcal{M}}_{0,4}$ under the induced moduli map $V\to
\overline{\mathcal{M}}_{0,4}\cong \PP^1$. This will be a divisor in $V$ which is nef, implying that it does not get
contracted when mapped to $U$. Hence the locus in $U$ of surfaces for which the image of the corresponding fibre of $g$
is reducible is in fact divisorial. One can then use basic properties of Noether--Lefschetz loci and of the moduli space
of K3 surfaces to prove the following theorem, which we state for simplicity over the complex numbers.

\begin{thmC}
    Let $X$ be a complex projective K3 surface, with $\Pic(X) = \Lambda$ for a lattice $\Lambda$ of rank $r\ge 2$ and
    let $L\in \Lambda$ be a primitive ample divisor on $X$. Let $M\subseteq M_{\Lambda,\CC}$ be a component of the moduli space of
    $\Lambda$-polarised K3 surfaces containing the point representing $X$. Suppose that for a general $Y\in M$,
    there is an integral rational curve in $|L|$ on $Y$. Then the same holds for $X$, i.e., there is an integral
    rational curve in $|L|$ on $X$.
\end{thmC}

The above theorem, along with the existence of integral rational curves in all positive enough ample classes on a
general K3 surface of Picard rank 2 and 4 achieved in Appendix \ref{appendix} and \cite{generic} concludes the remaining cases of Conjecture
\ref{conj: inf many rational curves}. In Section \ref{sec: pos char} we show how the marked point trick can be extended
to mixed characteristic, giving the conjecture in arbitrary characteristic.

\subsection{Leitfaden}
This article is organised as follows:

In Section \ref{sec:deformations}, we recall and extend results concerning the Kontsevich
moduli space of stable curves on K3 surfaces, and also of the moduli space of lattice polarised K3 surfaces itself.

In Section \ref{sec: existence}, we recall known facts about existence of rational curves on K3 surfaces, and extend them
to genus 1. We also summarise some of the results of Appendix \ref{appendix} and \cite{generic} which will be used later on.

In Section \ref{sec: regeneration}, we establish Theorem B, proving regeneration of curves on K3
surfaces. The idea of the proof is similar to the use of {\em rigidifiers} as introduced by Li and the third named author in
\cite{liliedtke}. Here, we need more general existence results for such curves proved in Appendix \ref{appendix}.

In Section \ref{sec: applications}, we discuss applications of the regeneration technique, reproving the main theorem of
\cite{liliedtke}. Furthermore we extend regeneration to Enriques surfaces and discuss it for other surfaces of Kodaira
dimension zero.

In Sections \ref{sec: main theorem 1}, \ref{sec: main theorem 2}, we reduce the main theorem to genus $\leq1$ and Picard rank
$2,4$ and prove the marked point trick. Then we prove Theorem A in characteristic zero, which is achieved by combining
regeneration, the marked point trick and the results of Appendix \ref{appendix} and \cite{generic}.

In the final Section \ref{sec: pos char} we explain how to extend the results of the previous two sections to positive characteristic.

\subsection{Notation}
A \textit{K3 surface $X$} over a field $k$ will be a geometrically integral, smooth, proper and
separated scheme of relative dimension $2$ over $k$, so that $\omega_X\cong\OO_X$ and $\HH^1(X,\OO_X)=0$.
Let $S$ be a connected base scheme. Then a morphism
\[
   \begin{tikzcd}f:\calX \ar{r} &S\end{tikzcd}
\]
is a \textit{smooth family of surfaces} if $f$ is a smooth and proper morphism of algebraic spaces of relative dimension
two whose geometric fibres are irreducible. In particular a \textit{family of K3 surfaces} is a family of surfaces where every
fibre is a K3 surface as above.
A property that holds for a \textit{general point} in a set will mean that it holds for all points of a Zariski-open subset,
whereas a \textit{very general point} will be one in the complement of countably many Zariski-closed subsets.
We will be a little sloppy on the distinction between a divisor $D\in Z^1(X)$ and its divisor class $[D]$ in $\Pic(X)$. For
example, when we say $D\in \Lambda$ or $D\not\in \Lambda$ for a divisor $D$, we mean $[D]\in \Lambda$ or $[D]\not\in
\Lambda$.

\begin{Acknowledgements}
We thank O. Benoist, D. Huybrechts, K. Ito, M. Kemeny, G. Martin, J. C. Ottem and S. Tayou for discussions and comments
and in particular A. Knutsen for remarks and corrections. We are very grateful to the anonymous referees whose comments
and corrections greatly improved the paper. The first named author is partially supported by the NSERC
Discovery Grant 262265.  The second and third named authors are supported by the ERC Consolidator Grant 681838
``K3CRYSTAL''.
\end{Acknowledgements}

\section{Curves, K3 surfaces and moduli}
\label{sec:deformations}

In this section, we study deformations of K3 surfaces and of curves in families of K3 surfaces: first, via
elementary deformation theory and second, via moduli spaces of stable maps. Almost nothing in this section is
new, but we gather facts and definitions and extend them to positive characteristic, something not always easily found
in the literature.

\subsection{Moduli of lattice-polarised K3 surfaces}
\label{subsec:mod of k3s}

First, we recall polarised moduli spaces of K3 surfaces. If $X$ is a K3 surface over a field, then a \emph{polarisation}
(resp., a \emph{quasi-polarisation}) on $X$ is an ample (resp., big and nef) line bundle on $X$. A
(quasi)-polarisation is called \emph{primitive} if it is primitive in the Picard group, i.e., it cannot be written as
a multiple of another line bundle. If $(X,H)$ is a (quasi-)polarised K3 surface, then the self-intersection
$H^2$ is a non-negative and even integer, which is called the {\em degree}. For every integer $d>0$, there exists a
moduli stack $\rm{\mathbf{M}}_{2d}$ of degree $2d$ primitively polarised K3 surfaces. In fact, this stack is
Deligne--Mumford of finite type over $\ZZ$ and it is also separated \cite[Theorem 4.3.3]{rizov}. We refer to \cite[\S
5]{huybrechts}, \cite{rizov} for details.

More generally, assume that $\Lambda$ is a lattice of rank $r$ and signature $(1,r-1)$ which
can be primitively embedded in the K3 lattice $E_8(-1)^{\oplus2}\oplus U^{\oplus3}$. Denote by $\Delta_\Lambda$ the
discriminant (i.e., the determinant of the intersection matrix) of $\Lambda$. Then we refer to \cite{dolgachev,
beauville, achter} for the construction of the moduli stack $\rm{\mathbf{M}}_\Lambda$ of $\Lambda$-polarised K3
surfaces (which is again separated, Deligne--Mumford, and of finite type over $\ZZ$) and for conditions for its
nonemptiness.

Regarding smoothness of the stack $\rm{\mathbf{M}}_\Lambda$, we note that over the complex numbers the above stacks are
smooth \cite[Proposition 2.1]{dolgachev}, \cite[Proposition 2.6]{beauville}. 

\begin{Proposition}(\cite[Proposition 3.3]{achter}, \cite[\S 4]{lieblichmaulik})
\label{prop: lattice polarised stack smooth}
    The stack $\rm{\mathbf{M}}_\Lambda$ is separated and Deligne--Mumford. Over $\ZZ[1/\Delta_\Lambda]$, it is smooth
    and of relative dimension $20-r$. Moreover, if $(X,\Lambda)$ is a $\Lambda$-polarised K3 surface over an
    algebraically closed field $k$ of positive characteristic $p$ so that $X$ is not supersingular, then $(X,\Lambda)$
    is a smooth point of $\rm{\mathbf{M}}_\Lambda$.
\end{Proposition}
\begin{proof}
    The first and second claims go back to \cite{rizov}, see, e.g., \cite[Proposition 3.3]{achter}. What remains to be
    proven is that non-supersingular K3 surfaces give smooth points in moduli, which is pointed out in \cite[\S
    4]{lieblichmaulik}, and relies on results of \cite[Remark 1.9]{ogus} for the ordinary case and \cite[Proposition
    10.3]{katsuravandergeer} for the finite height case.
\end{proof}

These moduli stacks have quasi-projective coarse moduli spaces, but in all applications in this paper we will be working
in some (usually smooth) atlas of the stack which by slight abuse of notation we shall denote by $M_{2d}$ and $M_\Lambda$
(or $M_{2d,\ZZ}$ etc.). For example, the complex moduli space $M_{\Lambda,\CC}$ is smooth of dimension $20-r$, and if
$r:=\rank(\Lambda)$ is small enough (see \cite[Proposition 5.6]{dolgachev} for particulars), then it is irreducible.

In the remainder of this subsection we will outline a construction, used repeatedly in proofs in this paper, giving the
algebraic deformation space of a K3 along with a big and nef divisor.

\begin{Construction}\label{construction:defspace}
Let $X$ be a K3 surface over an algebraically closed field $k$, and let $L$ be a line bundle on $X$. If $
\chara(k)=0$, then we set $A:=k$ and if $\chara(k)=p>0$, then we set $A:=W(k)$, the ring of Witt vectors of $k$.
Consider now $\operatorname{Def}(X/A)$ the smooth formal versal deformation space (see \cite[Theorem 8.5.19]{fga},
\cite[Proposition 9.5.2]{huybrechts}) which is flat and of relative dimension 20 over $\operatorname{Spf}A$. From a
theorem of Deligne \cite[Th\'eor\`eme 1.6]{deligne} (see also \cite[Theorem 9.5.4]{huybrechts}) the locus
$\operatorname{Def}(X/A,L) \subset \operatorname{Def}(X/A)$ where $L$ deforms is a formal Cartier divisor that is flat
over $A$. As pointed out in \cite[Remarque 3.6]{benoist}, this divisor need not be smooth or irreducible if $L$
is not a primitive polarisation or if the special fibre is supersingular. 

If $H$ is big and nef, then the formal family over $\widehat{T}(H):=\operatorname{Def}(X/A,H)$ is algebraisable after passing to a finite
extension from Grothendieck's Existence Theorem and a theorem of Deligne \cite[Corollaire 1.8]{deligne} (see also \cite[Corollary
9.5.5]{huybrechts}), i.e., we obtain a family $\calX\to T(H)$ of K3 surfaces whose generic fibre is in characteristic
zero regardless of the characteristic of $k$, and which has a line bundle $\calH$ on $\calX$ such
that $\calH_0=H$. Note that since we used a big and nef line bundle to algebraise, the algebraisation may only exist in
the category of algebraic spaces.
\end{Construction}

\subsection{Lifting curves}

First, assume that $S$ is the spectrum of a complete DVR and
assume that there exists a relatively ample line bundle $\calH$ on $f:\calX\to S$.
We denote by $\calH_\eta$ and $\calH_0$ the restriction of $\calH$ to the generic fibre $\calX_\eta$ and the special
fibre $\calX_0$, respectively.

Assume that the higher direct images satisfy $R^if_\ast\calH=0$ for $i\geq1$, which is the case if
$h^i(\calX_0,\calH_0)=0$ for $i\geq1$, which is the case, for example, if $\calH_0$ is ``sufficiently ample'' by Serre
vanishing. Then, $f_\ast\calH$ is a locally free $\OO_S$-module, and thus, a free $\OO_S$-module of rank $h^0(\calX_0,\calH_0)$. In particular, the
maps
\[
\begin{tikzcd}
    \HH^0(\calX_\eta,\calH_\eta) &\ar{l}[swap]{\imath} \HH^0(\calX,\calH)
    \ar{r}{\pi} & \HH^0(\calX_0,\calH_0)
\end{tikzcd}
\]
have the property that $\pi$ is a surjection and that $\imath$ is an injection that becomes an isomorphism after tensoring
with $k(\eta)$. It is in this sense that we can lift curves on $\calX_0$ that lie in the linear system $|\calH_0|$ to
$\calX_\eta$. In this case, the lifted curve lies in the linear system $|\calH_\eta|$, but in general, other than
giving an upper bound, one cannot control the geometric genus of the lift.

\begin{Remark}
 If $f$ is a family of K3 surfaces and $\calH_0$ is big and nef, then the vanishing assumption $R^if_\ast\calH=0$ for
 $i\geq1$ is fulfilled, see, for example, \cite[Proposition 2.3.1]{huybrechts}.
\end{Remark}

\subsection{Deforming curves via moduli spaces of stable maps}
\label{subsec: stable maps}

In the previous subsection, we lifted curves from the special to the general fibre. However, this required some
cohomology vanishing. Moreover, if the curve on $\calX_0$ was singular, we could not control the geometric genus of the
deformation. To control this as well, we use moduli spaces of stable maps as introduced by Kontsevich \cite{kontsevich}
and we refer to \cite{abramovichvistoli} for various algebraic constructions, as well as to \cite{bht, hk, liliedtke}
for discussions that are already adapted to moduli spaces of stable maps on families of K3 surfaces. We begin by noting
that the \textit{arithmetic and geometric genus} of a geometrically irreducible curve $C$ over a possibly non-perfect
field $k$ will be defined as the arithmetic and geometric genus over the algebraic closure $\overline{k}$ of $k$, i.e.,
$h^1(\overline{C},\OO_{\overline{C}})$ and $h^1(\overline{C}',\OO_{\overline{C}'})$ respectively, where
$\overline{C}=C\otimes_k\overline{k}$ and $\overline{C}'$ denotes its normalisation. For reducible curves the
genus in each case is defined as the sum of the genera of all the irreducible components.

A \textit{rational curve} will be a geometrically integral curve over a field whose geometric genus is zero. We call a rational
curve in a variety \textit{rigid} if it does not deform in a one-dimensional family.
In particular, we use the word rigid in the most liberal manner, i.e., we allow infinitesimal deformations.

We mention that rational curves on K3 surfaces in characteristic zero are rigid.  Non-rigid rational curves exist on K3
surfaces in positive characteristic, but then, the K3 surface in question is uniruled.  We note that uniruled K3
surfaces indeed exist and the first examples, which are even unirational, were given in \cite{ShiodaExample}.  Moreover,
by \cite[\S2]{ShiodaExample}, the Picard rank of a uniruled K3 surface over an algebraically closed field is equal to
its second Betti number, which is equal to $22$, i.e., the K3 surface is \textit{supersingular (in the sense of
Shioda)}. Since the Tate conjecture for K3 surfaces is by now fully established, Shioda's notion of supersingular
coincides with the others (height of the formal Brauer group \cite{artinsupersingular} or in terms of the Newton polygon
on cohomology), see \cite{totaro}. We note that supersingular K3 surfaces form 9-dimensional families
\cite{artinsupersingular} and are thus very special, even in positive characteristic.  We also note that even on
uniruled K3 surfaces in positive characteristic not all rational curves move, see Section \ref{sec: existence}.

Let now $f:\calX\to S$ be a smooth projective family of varieties over a base scheme $S$, and let $\calH$ a fixed ample line
bundle on $\calX$. There is a Deligne-Mumford stack $\overline{\rm{\mathbf{M}}}_g(\calX/S, d)$ parametrising stable genus $g$
maps from curves $h:D\to\calX_t$ into fibres of $f$ so that $h_*[D].\calH_t=d\in\mathbb{Z}$ (see, e.g., \cite[Theorem
13.3.5]{olsson}, the results ultimately going back to \cite{kontsevich, fp}). Following \cite[Theorem
50-51]{ak}, its coarse moduli space $\overline{\mathcal{M}}_g(\calX/S, d)$ is a projective
scheme.
From \cite[\S 8.2]{abramovichvistoli}, we have a closed and open substack, with projective coarse moduli space denoted by
$\overline{\mathcal{M}}_g(\calX/S, \beta)$, parametrising morphisms where $h_*[D]=\beta$ for some cohomology class
$\beta\in R^{2n-2}f_*\mathbb{Z_\ell}$ (where $\ell$ is a prime invertible on $S$, which will always exist in this article)
whose cup product with $\calH_t$ equals $d$.

If the fibres of $f$ are smooth projective K3 surfaces, then rational, algebraic and numerical equivalence all agree so we
have a coarse moduli space $\overline{\mathcal{M}}_g(\calX/S,\calH)$ parametrising maps so that $\OO_{\calX_t}(h(D))
\cong \calH_t$.

Denote by $(X,L)=(\calX_t, \calH_t)$ a fibre. We briefly recall the deformation theory of $[h:D\to
X]\in\overline{\mathcal{M}}_g(X,L)$ where $h$ is a closed embedding and $D$ is smooth and irreducible. The deformation
theory of the image $D\subset X$ is determined in the Hilbert scheme by the normal sheaf $\calN_{D/X}$, with the Zariski
tangent space to the Hilbert scheme at the point $D$ being given by $\HH^0(D, \calN_{D/X})$ and obstructions lying in
$\HH^1(D, \calN_{D/X})$. In other words locally around $[D]$, the Hilbert scheme $\operatorname{Hilb}(X)$ can be defined
by up to $h^1(\calN_{D/X})$ equations in a smooth $h^0(\calN_{D/X})$-dimensional space (\cite[I.2.17]{kollar} for the
similar relative statement). Similarly, if we deform the morphism $h$ while keeping the source $D$ and target $X$ fixed,
then the deformations in $\Hom(D,X)$ are controlled by the sheaf $h^*T_X$ (\cite[Theorem II.1.7]{kollar}).

Assume now that $h:D\to X$ is a stable morphism that satisfies the following two conditions.
\begin{enumerate}
    \item Every irreducible component of $D$ is smooth.
    \item The morphism $h$ is unramified.
\end{enumerate}

To explain the second condition a bit, it implies that from the standard cotangent sequence
\[
\begin{tikzcd}
    h^*\Omega^1_X\ar{r}& \Omega^1_D\ar{r}& \Omega_h \ar{r} &0
\end{tikzcd}
\]
we have a surjection $h^*\Omega^1_X\to\Omega^1_D\to0$. More geometrically, this is the case for example if $h(D)$ is at worst nodal and
all components $D_i$ of $D$ map, under $h$, birationally onto their image and all the $h(D_i)$ are distinct.
From \cite[\S 3.4.2-3.4.3]{sernesi}, \cite[\S 2.1]{ghs}, \cite[\S 4]{bht}, the deformation theory of such an $h$ in
$\overline{\mathcal{M}}_g(X, L)$ is quite well behaved. In particular the complex governing the
deformations is quasi-isomorphic to a line bundle $\calN_h$ defined by the sequence
\begin{equation}\label{def:normalsheaf}
\begin{tikzcd}
    0\ar{r} &T_D\ar{r}& h^*T_X\ar{r}& \calN_h\ar{r}& 0.
\end{tikzcd}
\end{equation}

\begin{Definition}
Let $X$ be a K3 surface over an algebraically closed field. If $h:D\to X$ is an element of $\overline{\mathcal{M}}_g(X,
L)$, then we say that it \textit{deforms in the expected dimension} if every irreducible component $M\subset
\overline{\mathcal{M}}_g(X, L)$ with $[h]\in M$ satisfies $\dim M \leq g$. We say that an integral curve $C\subset X$
deforms in the expected dimension if the composition of the embedding with the normalisation morphism $h:\widehat{C}\to
X$ does so.
\end{Definition}

For example, rational curves on a K3 surface in characteristic zero are rigid and thus, deform in the expected dimension.
We recall now that stable unramified morphisms
deform in the expected dimension, a fact which has appeared in various forms in the literature \cite[Proposition
2.3]{hk}, \cite[Lemma 2.3.4]{kemenythesis}, and later on we will extend this to allow worse singularities.

\begin{Lemma}\label{lem:unramified}
    Let $h:D\to X$ be an unramified morphism from a connected, nodal, (arithmetic) genus $g$ curve to a smooth
    projective K3 surface over an algebraically closed field $k$ and assume that every irreducible component of $D$ is
    smooth. Then \[ h^0(D, \calN_h) \leq g. \]
\end{Lemma}
\begin{proof}
    If $D$ is irreducible, then taking determinants in the short exact sequence of vector bundles
    \[\begin{tikzcd} 0\ar{r}& T_D\ar{r}& h^*T_X \ar{r}& \calN_h\ar{r}&0, \end{tikzcd}\]
    we obtain $\calN_h\cong\Omega^1_D$ and so $\rm{H}^0(D, \calN_h)\cong\rm{H}^0(D,\Omega^1_D)\cong k^g$. Note
    now that we can always choose a labelling of the irreducible components of $D$ so that $D=\cup_{i=1}^n D_i$ and
    for all $j\leq n$ we have that $\cup_{i=1}^j D_i$ is connected: this is an assumption in \cite[Lemma 2.3.4]{kemenythesis}
    but here is a simple proof that it always exists. By induction, assume that we have labelled the first $k$ components so
    that for all $j\leq k$ we have that $\cup_{i=1}^j D_i$ is connected. To conclude, from the set of remaining
    components of $D$ choose any one that meets one of the $D_1,\ldots,D_k$---one such must exist otherwise $D$ would be
    disconnected.

    Recall that \cite[Lemma 2.6]{ghs} (or \cite[Proposition 2.3.1]{kemenythesis}) gives that for any connected sub-curve $C\subset D$
    with $\{p_1,\ldots,p_r\}=C\cap\overline{D\setminus C}$ and $h_C:C\to X$ the restriction of $h$ we have
    $\calN_{h_C}(\sum_{i=1}^r p_i)=\calN_h|_C$. In particular, let $T=\cup_{i=1}^{n-1}D_i$, let $\{p_1,\ldots,p_r\} = D_n\cap
    \overline{D\setminus D_n}$, let $I_{D_n}$ be the ideal sheaf of $D_n\subset D$, and let us consider the standard short exact
    sequence
    \[ \begin{tikzcd} 0 \ar{r}& \calN_h\otimes I_{D_n} \ar{r}& \calN_h \ar{r}& \calN_h|_{D_n} \ar{r}& 0.\end{tikzcd} \]
    Note now that $I_{D_n}\cong \OO_T(-\sum_{i=1}^r p_i)$ so $\calN_h\otimes I_{D_n}\cong \calN_h|_T(-\sum p_i)\cong \calN_{h_T}$
    and similarly $\calN_h|_{D_n}=\calN_{h_{D_n}}(\sum_{i=1}^r p_i)$ which gives
    \[ \begin{tikzcd} 0 \ar{r}& \calN_{h_T} \ar{r}& \calN_h \ar{r}& \calN_{h_{D_n}}\left(\sum p_i\right) \ar{r}& 0. \end{tikzcd}\]
    From $\calN_{h_{D_n}}\cong\Omega^1_{D_n}$ it follows that since $h^1(D_n, \Omega^1_{D_n}(\sum p_i))=0$ by Serre duality,
    we can compute by Riemann-Roch that $h^0(D_n,\Omega^1_{D_n}(\sum p_i)) = g(D_n)+r-1$ and so since the
    inductive hypothesis gives that $h^0(T, \calN_{h_T})\leq g(T)$, we obtain
    \begin{align*}
        h^0(D, \calN_h) &\leq h^0(D_n, \Omega^1_{D_n}\left(\sum p_i\right)) + h^0(\calN_{h_T}) \\
                    &\leq g(D_n)+r-1+g(T) \\
                    &= g(D).\qedhere
    \end{align*}
\end{proof}

\begin{Theorem}\label{thm:expdimdef}
    Let $X$ be a K3 surface over an algebraically closed field $k$, and let $L$ be a line bundle on $X$.
    Assume that there exists a stable (arithmetic) genus $g$ curve $h:D\to X$ with $\OO_X(h(D))\cong L$ so that $h$ is
    unramified and every irreducible component of $D$ is smooth. Then $h$ deforms in the expected dimension.
\end{Theorem}
\begin{proof}
    This follows immediately from Lemma \ref{lem:unramified} and the fact that the tangent space to the deformation space
    of $h$ has dimension $h^0(\calN_h)$.
\end{proof}

For the purposes of this paper we can relax the above conditions slightly, namely to those of \cite[\S 4]{bht}.
\begin{Definition}
Let $h:D\to X$ be a morphism from a nodal connected curve. We say that $h$ is \textit{generically an embedding} if there exists a dense open subset in $D$ on which $h$ is an embedding.
\end{Definition}

In particular no irreducible components are being contracted via $h$. For a stable morphism $h:D\to X$ which is generically an embedding, it is proved in \cite[Lemma 11]{bht} that the deformation theory of $[h]$ as an element of $\overline{\mathcal{M}}_g(X, L)$ is controlled by a coherent sheaf $\calN_h$ on $D$ that is generically of rank one.
In the case where $D$ is smooth and irreducible, its dual agrees with the kernel of the morphism $h^*\Omega^1_X\to\Omega^1_D$ on the complement of the ramification locus of $h$. This sheaf is locally free if and only if $h$ is unramified. Denote by $\calK_h$ its torsion subsheaf fitting in the exact sequence \[\begin{tikzcd} 0\ar{r}& \calK_h\ar{r}& \calN_h\ar{r}& \calN_h'\ar{r}&0\end{tikzcd}\] so that $\calN_h'$ is locally free.
\begin{Remark}
As pointed out in \cite[p539]{bht}, the substack
\[\overline{\mathbf{M}}_g^0(\calX/S, \calH) \subset \overline{\mathbf{M}}_g(\calX/S, \calH)\]
parametrising stable maps which are generically an embedding on every irreducible component is actually a scheme. This
follows from the fact that all stabilisers are trivial since such stable maps have no automorphisms and the fact that the stack admits a morphism to the Hilbert scheme parametrising the image curves, and the fibres of this morphism
are finite since there are only finitely many connected curves through which the normalisation of the image factors. One
concludes by noting that such a stack is a scheme (\cite[Lemma 2.3.9]{ah} and
\cite[\href{https://stacks.math.columbia.edu/tag/0417}{Tag 0417}]{stacks}).
\end{Remark}
We now include a proof and extensions of the fact, well known to experts, that irreducible curves on a K3 surface over the complex numbers deform in the expected dimension.
\begin{Proposition}\label{prop:defexpdim}
Let $C\subset X$ be an integral curve in a K3 surface over an algebraically closed field $k$ and let $h:D\to X$ be its normalisation. Then $h$ deforms in the expected dimension if any of the following conditions holds:
\begin{enumerate}
    \item $\chara(k)=0$,
    \item $h(D)\subset X$ is nodal,
    \item the geometric genus $g(D)\le 1$ and $X$ is not uniruled.
\end{enumerate}
\end{Proposition}
\begin{proof}
The second claim follows from Theorem \ref{thm:expdimdef}. For the first, if $C$ is rational, then it must be rigid, so we
are done. In higher genus the key observation, due to Arbarello--Cornalba \cite[Lemma 1.4]{arbcorn} (see also \cite[\S
XXI.9]{acgh2}, \cite[Lemma 3.41]{harrismorrison} for expanded accounts), is that even though global sections of
$\calN_h$ parametrise first order deformations of $h$, sections coming from $\calK_h$ do not contribute to the
deformations of the image of $h$. For our purposes, the result is summarised succinctly in \cite[Lemma
2.2]{dedieusernesi}. An application of this in our context (which follows from \cite[Lemma 2.3, Theorem
2.5]{dedieusernesi}) is the fact that assuming the genus $g$ of $D$ is at least one then, even though $h$ may be highly
ramified as $C$ is highly singular, the general member $[f]$ of an irreducible component $M\subset
\overline{\mathcal{M}}_g(X, \OO_X(C))$ containing $[h]$ will be unramified. From Theorem \ref{thm:expdimdef} this
implies that $\dim M\leq g$ as required.

If $X$ is not uniruled, then the third claim is clear when $g(D) = 0$.  Now let us assume that $g(D) = 1$ and that $D$ deforms in
dimension $\ge 2$. So there exists a two-parameter family of genus 1 curves on $X$. Fixing a general point $q\in X$,
there is a one-parameter family of genus 1 curves on $X$ passing through $q$. If this family has varying moduli, then it
degenerates to a union of rational curves on $X$, one of them passing through $q$, which implies that $X$ is uniruled
since $q$ is a general point. So this family of genus 1 curves must be isotrivial. Suppose that $D$ is a member of this
family and $h(p) = q$ for $p\in D$. Then the map $h: (D,p)\to X$ deforms in dimension $1$ with $h(p) = q$. Therefore,
there exists a dominant rational map $f: Y = D\times B \dashrightarrow X$ with $p\in D$ such that $f(P) = q$ for $P =
\pi_D^{-1}(p)$, where $B$ is a smooth projective curve and $\pi_D: Y=D\times B\to D$ is the projection to $D$. Clearly,
$f$ cannot be regular along $P$. Let $\phi: \widehat{Y}\to Y$ be a birational morphism resolving the indeterminacy of
$f$, i.e., such that $f\circ \phi:\widehat{Y}\to X$ is regular. Then there must be a rational curve $E\subset
\widehat{Y}$ such that $\phi_* E = 0$, $E \subset \phi^{-1}(P)$, $(f\circ \phi)_* E \ne \emptyset$ and $q\in F =
f\circ\phi(E)$. So $F$ is a rational curve on $X$ passing through $q$. Again this implies that $X$ is uniruled.
\end{proof}

\begin{Remark}
\label{rem: char p AC}
Note that the extra condition in positive characteristic is necessary as there are cuspidal rational curves on
supersingular K3 surfaces that move. It is likely one can improve the above result (and hence the main regeneration Theorem
\ref{thm: regeneration}) in positive characteristic to allow singularities which are in a sense bounded by the
characteristic, as for example is done for rational curves in \cite{iil}, but we do not pursue this here.
\end{Remark}

We have thus obtained upper bounds on the deformation spaces of stable morphisms into K3 surfaces, and now we aim to
achieve lower bounds using a classical idea of Ran, Bloch and Voisin on the semi-regularity map and the twistor line.
Our proof is very similar to that of \cite[\S 2.2]{btdensity}, \cite[Theorem 18]{bht} and extends \cite[Proposition
2.1]{hk} to positive characteristic.

\begin{Theorem}\label{theo:deffam}
    Let $\calX\to S$ be a smooth projective family of K3 surfaces over an irreducible base scheme $S$ and let $\calH$ be a
    line bundle on $\calX$. Then every irreducible component $M \subset \overline{\mathcal{M}}_g(\calX/S, \calH)$
    satisfies \[\dim M \geq g+\dim S. \]
\end{Theorem}
\begin{proof}
Consider an irreducible component $M\subset\overline{\mathcal{M}}_g(\calX/S, \calH)$ through a point
$h:D\to\calX_0$ for $0\in S$. The usual lower bound on the dimension of such a component (e.g., from
\cite[Theorem I.2.15]{kollar}) is given by
\begin{equation}\label{modulilowerbound}
\dim M \geq \chi(D, (h^*\Omega^1_{\calX_0}\to \Omega^1_D)^*)+\dim S = g-1+\dim S
\end{equation}
(see \cite[p541]{bht} or \cite[Proposition 2.1]{hk}). The aim now is to improve this bound by one, which for example in
the case that $S$ is defined over the complex numbers and $\calH$ is fibre-wise ample, is achieved by deforming $\calX_0$ in
the direction where it becomes non-algebraic, namely that of the twistor line.

More generally, let $K=\overline{k(0)}$ be the algebraic closure of the field of definition of
$(X,L)=(\calX_0,\calH_0)$. Like in Construction \ref{construction:defspace}, if $S$ is defined over $\mathbb{Q}$, then
let $A=K$, whereas in mixed or positive characteristic, let $A=W(K)$ be the ring of Witt vectors.
Consider now the equations defining the image of $\Spf \widehat{\OO}_{S,0}$ in $\operatorname{Def}(X/A, L)$
as equations in $\operatorname{Def}(X/A)$ and like in Construction \ref{construction:defspace}, after passing to a
finite extension, we obtain a family $\widehat{f}: \widehat{\calX}\to \widehat{U}$ with $\dim\widehat{U}=\dim S+1$ so that $\calH$ only
deforms in directions along $S$, i.e.,\ $\widehat{f}$ will only be proper and non-algebraisable, as the
line bundle $\calH$ is obstructed and so cannot deform in tangent directions transverse to those of $S$.

One can now form the limit of the corresponding moduli spaces (for the proper fibres one uses \cite[\S
8.4]{abramovichvistoli}) to obtain the formal scheme $\overline{\mathcal{M}}_g(\widehat{\calX}/\widehat{U},
\calH)$ over $\widehat{U}$. The same computation as above in Equation (\ref{modulilowerbound}) shows that the dimension of an
irreducible component $\widehat{M}$ through $h$ will now be $g+\dim S$.  On the other hand, since
$\OO(h(D))\cong\calH$, $\calH$ cannot deform to the general fibre of $\widehat{f}$ and neither can $h(D)$. In
other words, the moduli spaces $\overline{\mathcal{M}}_g(\widehat{\calX}/\widehat{U}, \calH)$ and
$\overline{\mathcal{M}}_g(\calX_U/U, \calH)$ agree near $h$ which gives the result.
\end{proof}

\begin{Corollary}\label{cor:def}
    Let $f:\calX\to S$ be a smooth projective family of K3 surfaces over an irreducible base scheme $S$, and
    let $\calH$ be a line bundle on $\calX$. Assume that there exists a stable morphism $h:D\to\calX_0$ to some fibre
    $\calX_0=f^{-1}(0)$ so that $\OO_{\calX_0}(h(D))\cong
    \calH_0$ and $h$ deforms in the expected dimension.
    Then any irreducible component $[h]\in M\subset\overline{\mathcal{M}}_g(\calX/S,\calH)$ has dimension
    exactly $\dim S+g$ and surjects onto $S$.
\end{Corollary}
\begin{proof}
    Since from Theorem \ref{theo:deffam} every irreducible component $[h]\in M$ has dimension at least $g+\dim S$ then,
    for $\pi:\overline{\mathcal{M}}_g(\calX/S, \calH) \to S$, if we had $\dim \pi(M)<\dim S$ there would have to be a
    $t\in S$ so that $M_t=M\cap\pi^{-1}(t)$ has $\dim M_t>g$, but $\pi^{-1}(t) =
    \overline{\mathcal{M}}_g(\calX_t,\calH_t)$, contradicting the assumption. If $\dim M>\dim S+g$, then similarly some
    fibre would have too many deformations giving the same contradiction. One concludes by noting that as $M$ is
    projective, $\pi$ is proper, hence as $M$ dominates $S$, it must surject onto it.
\end{proof}

\section{Existence of curves}
\label{sec: existence}

In this section we gather some facts about existence of rational curves that will be needed in later sections. We by no
means give a complete summary of the various problems in the area, but rather reference surveys in the literature. In Section
\ref{subsec:ellcurves}, we extend some of the results below to the case of curves of geometric genus 1, as this will be
needed in the proof of Theorem \ref{thm: curves on complex k3}.

\subsection{Rational curves}
\label{subsec:ratcurves}

The fact that every projective K3 surface over the complex numbers contains at least one rational curve was first shown
by Mori and Mukai \cite{morimukai}, who attribute this to Bogomolov and Mumford. We will need this result in the
following form.

\begin{Theorem}[Bogomolov--Mumford $+\varepsilon$]
\label{thm: bogomolov mumford}
    Let $X$ be a K3 surface over an algebraically closed field and let $\calL$ be a non-trivial and effective line bundle.
    Then, there exists a divisor in $|\calL|$ that is a sum of rational curves.
\end{Theorem}
\begin{proof}
Over the complex numbers, this is shown in \cite[Proposition 2.5]{btdensity}. For an extension to positive
characteristic, see \cite[Proposition 17]{bht} or \cite[Theorem 2.1]{liliedtke}.
\end{proof}

\subsection{Nodal rational curves on generic K3s}
\label{subsect:generic}

In \cite{chen}, the first named author proved existence of integral nodal rational curves in all multiples of the polarisation
of a generic K3 surface of Picard rank one. In a recent work of ours \cite{generic}, we extended this result to K3
surfaces of Picard rank two, namely the following theorem. 

\begin{Theorem}[\cite{generic}]
 \label{thm: nodal curves}
Let $\Lambda$ be a lattice of rank two with intersection matrix
\begin{equation}\label{K3lattice-2m}
\begin{bmatrix}
2d & a\\
a & 2b
\end{bmatrix}_{2\times2}
\end{equation}
for some $a, b\in \ZZ$ and $d\in \ZZ^+$ satisfying $4bd-a^2<0$. Let $M_{\Lambda,k}$ be the moduli space of
$\Lambda$-polarised K3 surfaces over an algebraically closed field $k$ of characteristic 0, and let $L\in \Lambda$ be such that $L$ is big and nef on a general
$X\in M_{\Lambda,k}$. Then there exists an open and dense subset $U\subseteq M_{\Lambda,k}$ (with respect to the Zariski
topology), depending on $L$, such that on every K3 surface $X\in U$, the complete linear series $|L|$ contains an
integral nodal rational curve if one of the following holds:
\begin{enumerate}
\item[A1.] $\det(\Lambda)$ is even;
\item[A2.] $L = L_1 + L_2 + L_3$ for some $L_i\in \Lambda$ satisfying that
$L L_i > 0$ and $L_i^2 > 0$ for $i=1,2,3$;
\item[A3.] $L = L_1 + L_2$ for some $L_i\in \Lambda$ satisfying that
$L L_i > 0$ for $i=1,2$, $L_1^2 > 0$, $L_2^2 = -2$, $L_1\not\in 2\Lambda$,
$L_1 - L_2\not\in n\Lambda$ for all $n\in\ZZ$ and $n\ge 2$,
and $L_1^2 + 2L_1L_2\ge 18$.
\end{enumerate}
\end{Theorem}
In Theorem \ref{thm: appendix} in the appendix of the current paper, we give a shorter and more direct
argument that avoids degenerations to log K3 surfaces, but which produces rational curves with unramified normalisation
(i.e.,\ weaker than nodal) in all classes as above. For the purposes of this paper (namely in the Regeneration Theorem
\ref{thm: regeneration} and Theorem \ref{thm: curves on complex k3}), Theorem \ref{thm: appendix}
suffices as one does not need nodal curves. 

The existence of curves as in Theorem \ref{thm: appendix} will be used in the Regeneration Theorem
\ref{thm: regeneration} as follows. Fixing $L, C\in \Lambda$ such that $L$ is ample on a general K3 surface $X\in
M_{\Lambda,\CC}$, we want to show that there are nodal rational curves in $|NL - C|$ for $N$ sufficiently large. When
$\det(\Lambda)$ is even, this is true as long as $NL-C$ is big and nef. When $\det(\Lambda)$ is odd, this is true if
$(N-2) L - C$ is big and nef since we can write
\[ NL - C = L + L + ( (N-2) L - C) \]
which satisfies the condition A2 in the theorem. In either case, we can find a rational curve in $|NL-C|$ by
Theorem \ref{thm: appendix} as long as we take $N$ sufficiently large.

\begin{Remark}
    Note that condition A2 in Theorems \ref{thm: appendix} and \ref{thm: nodal curves} is necessary in the odd determinant
    case as the following example shows. Let $X$ be a Jacobian elliptic K3 surface with lattice
    \[\begin{pmatrix}-2&1\\1&0\end{pmatrix}\] considered by Bryan--Leung, generated by $C$ and $F$. Then $C+nF$ is ample
    for $n\geq3$ but every curve in the linear system $|C+nF|$ is reducible. In other words there is no integral
    rational curve in $|C+nF|$. On the other hand, there are rational curves in $|3(C+nF)|$.
\end{Remark}

\begin{Remark}
    We give a short argument to explain why Theorems \ref{thm: appendix} and \ref{thm: nodal curves} hold more generally
    over an arbitrary algebraically closed field $k$ of characteristic zero, even though they are stated over the
    complex numbers. Let $X$ be a general K3 surface over $k$ such that $\Pic(X)=\Lambda$ and $L\in\Lambda$ satisfies
    one of the conditions of the theorems. If $k$ is a non-trivial extension of $\CC$, then by spreading $X$ out over a
    positive dimensional irreducible scheme $B$ whose function field is $k$, we may find, over the general closed point
    $b\in B$ (whose residue field is necessarily $\CC$), a rational curve in $|L|$ with unramified normalisation (or
    nodal). We conclude by noting that from Corollary \ref{cor:def}, this curve deforms to the generic fibre $X$ and
    must also have unramified normalisation (resp., nodal) as this is an open condition. If $k\subset\CC$, then we have
    a rational curve in $|L|$ on $X_\CC=X\times_k\CC$, which is necessarily defined over an extension $\CC/K/k$.
    Spreading out to a family with generic fibre $X_K$, all closed fibres are isomorphic to $X$, so we may specialise
    the rational curve in any direction and obtain one in $|L|$ on $X$ itself.
\end{Remark}

The following technical statement is similar to the above and involves, by work of Vinberg, the only two rank four
lattices for which the corresponding K3 with such Picard lattice is not elliptic nor has infinite automorphisms.
Existence of infinitely many integral rational (resp., geometric genus 1) curves on a generic such K3 will allow us to
conclude it for all such K3s in Section \ref{sec: main theorem 1}.

\begin{Theorem}[\cite{generic}] 
\label{thm: curves on complex k3 picard rank 4}
Let $\Lambda$ be one of the following lattices of rank 4:
\begin{equation*}
\begin{bmatrix}
2 & -1 & -1 & -1 \\
-1 & -2 &  0 &  0 \\
-1 &  0 & -2 &  0 \\
-1 &  0 &  0 & -2 \\
\end{bmatrix},\hspace{10pt}
\begin{bmatrix}
12 & -2 &  0 & 0 \\
-2 & -2 & -1 & 0 \\
0 & -1 & -2 & -1 \\
0 &  0 & -1 & -2 \\
\end{bmatrix}.
\end{equation*}
Then for a general K3 surface $X$ with $\Pic(X) = \Lambda$, there is an integral rational (resp., geometric genus 1) curve in $|L|$ if
$L$ is a big and nef divisor $L$ on $X$ with the property that
\begin{equation*}
\begin{aligned}
L = L_1 + L_2 + L_3 \text{ for some }& L_i\in \Lambda \text{ satisfying that}\\
    &LL_i > 0 \text{ and } L_i^2 > 0 \text{ for } i=1,2,3.
\end{aligned}
\end{equation*}
\end{Theorem}

\subsection{Minimal nef divisors}
\label{subsec:minimal nef}

Next, we introduce a useful technique that allows us to decompose divisors into ``minimal nef'' ones,
which is somewhat reminiscent of the Enriques Reducibility Lemma for Enriques surfaces \cite[Chapter 2.3]{cdl}, or of
the Zariski decomposition.

\begin{Definition}
We call a Cartier divisor $L$ on a variety $X$ a {\em minimal nef} divisor if $L$ is nef and there does not exist a nef
divisor $M$ on $X$ such that $M\not\equiv 0, M\not\equiv L$ and $L-M$ is effective, where ``$\equiv$'' denotes numerical
equivalence.
\end{Definition}

We have the following simple observation:
\begin{Lemma}
\label{lem: minimal nef K3}
Let $X$ be a projective K3 surface over an algebraically closed field.
Then every effective divisor $D$ on $X$ can be written as
\begin{equation}\label{eq: minimal nef decomposition}
D \,=\, \sum_{i=1}^m L_i + C \,=\, L + C,
\end{equation}
where $L_i$ are minimal nef divisors and $C$ is an effective divisor whose support has negative definite self-intersection matrix.
\end{Lemma}

\begin{proof}
We consider the set
\[
\begin{aligned}
\Pi &= \Big\{(L,C): D = L+C,\ C \text{ effective and}\\
&\hspace{24pt}
L = \sum_{i=1}^m L_i \text{ for some } L_i \text{ minimal nef }
\Big\} \subset Z^1(X)\times Z^1(X).
\end{aligned}
\]
Since $D$ is effective, $\Pi \ne \emptyset$. Let us choose $(L,C)\in \Pi$ that minimises $\deg_A C$ for a fixed ample
divisor $A$. We claim that the support of $C$ has negative definite self-intersection matrix.  Recall that on a K3
surface, if a union of curves has negative definite self-intersection matrix, then it is a union of $(-2)$-curves with
simple normal crossings whose dual graph is a ``forest'', i.e., a disjoint union of trees.

Otherwise, suppose that the self-intersection matrix of $\mathop{\text{supp}}(C)$ is not negative definite. Then there
exists an effective divisor $B\ne 0$ such that $B\le C$ and $B^2 \ge 0$. By Riemann-Roch, there is a non-trivial moving
part of $|B|$ for such $B$. Therefore,
\[
B = P + N
\]
for some $P\not\equiv 0$ nef and $N$ effective. From the definition of minimal nefness, we can write $P = Q + M$ for
some $Q\not\equiv 0$ minimal nef and $M$ effective. Then
\[
D = L + C = L + B + (C-B) = L + P + N + (C-B) = (L+Q) + (M+N +C-B)
\]
and hence $(L+Q, M+N+C-B)\in \Pi$. Since $Q\not\equiv 0$ is nef,
\[
\deg_A (M+N+C-B) = A(M+N+C-B) = A(C - Q) < AC = \deg_A C,
\]
which contradicts our choice of $(L,C)$.
\end{proof}

\begin{Definition}
We call \eqref{eq: minimal nef decomposition} a {\em minimal nef decomposition} of $D$.
\end{Definition}

One of the key properties of minimally nef divisors on a K3 surface is the following proposition.

\begin{Proposition}\label{prop: integral curves in min nef divisor}
    Let $X$ be a non-uniruled K3 surface over an algebraically closed field and let $D$ be a minimally nef divisor.
    Then there exists an integral geometric genus 1 curve $C$ in $|D|$.
\end{Proposition}
\begin{proof}
If $D^2 = 0$, then a general member $C\in |D|$ is a smooth elliptic curve since $X$ is not uniruled (cf.\
\cite[Proposition 2.3.10]{huybrechts}). 

Suppose now that $D^2 > 0$. Then from the discussion in Construction \ref{construction:defspace}, there exists a smooth
proper family $\calX\to B=\Spec S$ over a DVR $S$ such that $\calX_0 = X$, whose geometric generic fibre
$\overline{\calX_\eta}$ is a general projective K3 surface of Picard rank $1$ in characteristic zero, and $D$ extends to a divisor $\calD$ on
$\calX$ such that $\calD_0 = D$. We choose a section $P\subset\calX$ of $\calX/B$ after a finite base change such that
$p = P_0$ is a general point on $X$. Since $\overline{\calX_\eta}$ is general, there is a nodal rational
curve in $|\calD_\eta|$ from the main theorem of \cite{chen} and so (see \cite[Proposition 13.2.1]{huybrechts}) there
exists a one-parameter family of genus 1 curves in $|\calD_\eta|$ since every node of the rational curve can be deformed
independently. After a base change there exists a flat proper family $\mathscr{C}\subset \calX$ of curves over $B$ such
that $\mathscr{C}_\eta \in |\calD_\eta|$ is a genus 1 curve passing through $P_\eta$. Let $\Gamma$ be the irreducible
component of $C = \mathscr{C}_0$ passing through $p=P_0$. Since $X$ is not uniruled and $p$ is a general point on $X$,
the geometric genus of $\Gamma$ is 1. Proposition \ref{prop:defexpdim} implies that $\Gamma$
moves in the expected dimension and so is nef. Finally, as $\Gamma\subset C\in |D|$, $\Gamma$ is nef and $D$ is
minimally nef, we must have $\Gamma = C$ implying that $C$ is a genus 1 curve in $|D|$.
\end{proof}

The following result will be crucial in the extension of Theorem \ref{thm: curves on complex k3} to higher genus and can
be considered as a genus 1 analogue of Theorem \ref{thm: bogomolov mumford}. 

\begin{Theorem}\label{thm: elliptic curve non-primitive class}
Let $X$ be a non-uniruled K3 surface over an algebraically closed field and let $\Lambda \subsetneq \Pic(X)$ be a
primitive sublattice of $\Pic(X)$. Suppose that $D^2 \ne 0$ for all non-zero $D\in \Lambda$ and that there is an ample divisor
$A\in \Lambda$ satisfying
\begin{equation}\label{eq: elliptic curve non-primitive class 03}
A R \,\ge\, 4 A^2 + \sqrt{17(A^2 - 2)A^2}
\end{equation}
for all $(-2)$-curves $R\subset X$ and $R\not\in\Lambda$. Then there exists an integral geometric genus 1 curve
$C\subset X$ whose divisor class does not lie in $\Lambda$.
\end{Theorem}

The proof is technical and proceeds until the end of this subsection. First, we show that existence of nodal rational curves
on a general K3 $X$ of Picard rank two implies the existence of an integral geometric genus 1 curve on $X$. Then we
degenerate this curve to an arbitrary K3 satisfying the assumptions of the theorem, while preserving
integrality. This final step in the proof---making use of the discriminant (i.e., singular) locus in a family of stable
curves along with a minimal nef decomposition of Lemma \ref{lem: minimal nef K3}---should be thought of as a precursor to
the marked point trick developed later in Section \ref{sec: main theorem 2}.

We will define now a total order ``$\prec$'' on the divisors of $X$ which will be used in the proof of the above
theorem. Fixing a basis of $\Pic_\QQ(X)$ consisting of $r = \rho(X)$ ample divisors $A=A_1, A_2,$ $\ldots, A_r$, we say
$F \prec G$ or $G\succ F$ if

\begin{itemize}
    \item either $F A_i = G A_i$ for all $i=1,2,\ldots,r$, or
    \item $FA_1 = GA_1$, $FA_2 = GA_2$, \ldots, $FA_{i-1} = GA_{i-1}$ and
    $FA_i < GA_i$ for some $1\le i\le r$.
\end{itemize}
It is easy to check the following:
\begin{itemize}
    \item $F\prec G$ and $F\succ G$ if and only if $F\equiv G$.
    \item If $F\prec G$ and $G\prec H$, then $F\prec H$.
    \item If $F_1\prec F_2$ and $G_1 \prec G_2$, then $F_1 + G_1 \prec F_2 + G_2$.
    \item If $F\ge G$, i.e., $F-G$ is (pseudo)-effective, then $F\succ G$.
\end{itemize}

\begin{proof}[Proof of Theorem \ref{thm: elliptic curve non-primitive class}]
From Proposition \ref{prop: integral curves in min nef divisor} there exists a $C$ as in the conclusion of the theorem
if there is a minimal nef divisor $D\not\in \Lambda$ on $X$. Otherwise, suppose that $\Lambda$ contains all minimal nef
divisors on $X$. Then since $\Lambda \ne \Pic(X)$, there exists an effective divisor $D$ on $X$ such that $D\not\in
\Lambda$. By Lemma \ref{lem: minimal nef K3}, we can write $D = L + E$, where $L$ is a sum of minimal nef divisors
and $E$ is a sum of $(-2)$-curves. Since $D\not\in \Lambda$ and $L\in \Lambda$, $E\not\in \Lambda$. Therefore, there
is at least one $(-2)$-curve $R$ on $X$ such that $R\not\in \Lambda$.

Let us choose a nef divisor $D\not\equiv 0$ on $X$ and a $(-2)$-curve $R\subset X$ with $R\not\in \Lambda$ such that
both $D$ and $R$ are minimal under the order ``$\prec$'' defined above. That is, $D\prec D'$ for all nef divisors
$D'\not\equiv 0$ on $X$ and $R\prec R'$ for all $(-2)$-curves $R'\subset X$ with $R'\not\in \Lambda$.  Clearly, $D$ is
minimally nef and hence $D\in \Lambda$. Since we have assumed that $\Lambda$ does not contain isotropic classes, $D$ is
big and nef. We claim that
\begin{eqnarray}\label{eq: elliptic curve non-primitive class 00}
D + R \equiv \Gamma + E \text{ for } \Gamma \not\equiv 0 \text{ nef and } E\ne 0 \text{ effective}
\\\hfill\Rightarrow D\equiv \Gamma \text{ and } R = E.\nonumber
\end{eqnarray}
Let $\Gamma = \Gamma_1 + \Gamma_2$ and $E = E_1 + E_2$ be minimal nef decompositions of $\Gamma$ and $E$, respectively, where
$\Gamma_1$ and $E_1$ are sums of minimal nef divisors and $\Gamma_2$ and $E_2$ are sums of $(-2)$-curves.
Since $\Gamma_1+E_1\in \Lambda$ and $D+R\not\in \Lambda$, $\Gamma_2+E_2\not\in \Lambda$ and hence $\Gamma_2+E_2$ contains
a $(-2)$-curve $R'\not\in \Lambda$. Therefore,
\[
D \prec \Gamma_1 + E_1 \text{ and } R \prec \Gamma_2 + E_2
\]
by our choice of $D$ and $R$. Since $D + R \equiv \Gamma + E$, we necessarily have
\[
D \equiv \Gamma_1 + E_1 \text{ and } R \equiv \Gamma_2 + E_2.
\]
Since $D\prec \Gamma_1$, we clearly have $E_1\equiv 0$. Hence $E_2 \equiv E$. Since $R$ is a $(-2)$-curve, $R\equiv
\Gamma_2 + E_2 \equiv \Gamma_2 + E$ and $\Gamma_2$ and $E\ne 0$ are effective, we must have $\Gamma_2 = 0$ and $R = E$.
In conclusion, $D\equiv \Gamma_1 = \Gamma$ and $R = E$. This proves our claim \eqref{eq: elliptic curve non-primitive
class 00}. Next, we claim that
\begin{equation}\label{eq: elliptic curve non-primitive class 01}
D \not\equiv nF\text{ and } D - R \not\equiv nG\text{ for all } F,G\in Z^1(X),\ n\in \ZZ \text{ and }
n\ge 2.
\end{equation}
It is clear that $D\not\equiv nF$ for all divisors $F$.  Otherwise, if $D\equiv nF$ for some divisor $F$ and $n\ge 2$,
then $F$ is nef and $FA < DA$, contradicting $D\prec F$.

Suppose that $D - R\equiv nG$ for some divisor $G$, $n\in \ZZ$ and $n\ge 2$. Then
\[
(n-1) D + R \equiv n(D - G) = nP.
\]
Since $AP > 0$ and $P^2 \ge 0$, $P$ is effective. Let $P = P_1 + P_2$ be a minimal nef decomposition of $P$, where $P_1$
is a sum of minimal nef divisors and $P_2$ is a sum of $(-2)$-curves with negative definite intersection matrix. Since
$P^2 \ge 0$, $P_1\not\equiv 0$. Since $P_1\in \Lambda$ and $P\not\in \Lambda$, $P_2 \not\in \Lambda$ and hence $P_2$
contains a $(-2)$-curve $R'\not\in \Lambda$. Therefore, $D\prec P_1$ and $R\prec P_2$. Then
\[
(nP)A = n P_1A + n P_2A \ge n DA + n RA > ((n-1) D + R)A
\]
which is a contradiction, proving our claim \eqref{eq: elliptic curve non-primitive class 01}.
Next, let us prove that
\begin{equation}\label{eq: elliptic curve non-primitive class 02}
D R \ge 8.
\end{equation}
Let us consider the divisors $A, D, R$. By the Hodge Index Theorem, their intersection matrix
has signature $(1,2)$ if
they are linearly independent in $\Pic_\QQ(X)$ and $(1,1,1)$ otherwise. The same holds for the intersection matrix
\[
\begin{bmatrix}
(2A + (AR)R)^2 & (2A + (AR)R)(2D + (DR)R) & 0\\
(2A + (AR)R)(2D + (DR)R) & (2D + (DR)R)^2 & 0\\
0 & 0 & -2
\end{bmatrix}
\]
of $2A + (AR)R$, $2D + (DR)R$ and $R$.
Either way, this translates to
\[
(2AD + (AR)(DR))^2 \ge (2A^2 + (AR)^2)(2 D^2 + (DR)^2).
\]
Since $D$ is big and nef, $D^2\ge 2$. Since $D\prec A$, $AD \le A^2$. So we have
\[
\begin{aligned}
&\quad (2 A^2 + (AR)(DR))^2 \ge (2A^2 + (AR)^2)(4 + (DR)^2)
\\
&\Leftrightarrow a y^2 - 2a xy + (2x^2 + 4a - 2a^2) \le 0
\end{aligned}
\]
for $a = A^2$, $x=AR$ and $y=DR$. It is easy to check that \eqref{eq: elliptic curve non-primitive class 03}
implies that $y\ge 8$, i.e., \eqref{eq: elliptic curve non-primitive class 02}.
And since $D^2 \ge 2$, we have
\begin{equation}\label{eq: elliptic curve non-primitive class 04}
D+R \text{ big and nef, } D^2 > 0,\ (D+R)R > 0 \text{ and } D^2 + 2D R \ge 18.
\end{equation}

There exists a smooth proper family $\calX\to B=\Spec S$ over a DVR $S$ such that $\calX_0 = X$,
$\overline{\calX_\eta}=\calX\times_S \overline{K(S)}$ is a general projective K3 surface with Picard rank $2$ over the algebraic
closure of the function field $K(S)$ of $S$ and $D$ and $R$ extends to divisor $\calD$ and $\mathscr{R}$ on $\calX$ with
$\calD_0 = D$ and $\mathscr{R}_0 = R$. Let $L = D + R$ and $\calL = \calD+\mathscr{R}$. By \eqref{eq: elliptic curve
non-primitive class 01} and \eqref{eq: elliptic curve non-primitive class 04}, $\calL_\eta$ satisfies A3 in Theorem \ref{thm:
appendix}. Consequently, there is an integral rational curve with unramified normalisation and hence, by deforming a
partial normalisation of this curve, a one-parameter family of genus 1
curves in $|\calL_\eta|$ on $\calX_\eta$, after a base change. We fix a general section $P\subset \calX$ of
$\calX/B$ and choose a genus 1 curve in $|\calL_\eta|$ passing through $P_\eta$ on $\calX_\eta$. So we obtain a flat
proper family $\mathscr{C}\subset \calX$ of curves over $B$ such that $\mathscr{C}_\eta$ is a nodal genus 1 curve in
$|\calL_\eta|$ on $\calX_\eta$ and $P\subset \mathscr{C}$. Clearly, $C = \mathscr{C}_0$ is a union of irreducible
components of geometric genus $\le 1$ with at most one genus 1 component. Let us write
\[
C = \Gamma + E\in |L| = |D+R|,
\]
where $\Gamma$ is the component of $C$ passing through $p=P_0$. Since $p$ is a general point on $X$, $\Gamma$ must be a
genus 1 curve and hence nef. By \eqref{eq: elliptic curve non-primitive class 00}, either $C = \Gamma$ is an integral
genus 1 curve in $|L|$ or $C = \Gamma + R$, where $\Gamma$ is an integral genus 1 curve in $|D|$. If it is the former,
then we are done. Otherwise, suppose that $C = \Gamma + R$. The corresponding stable map onto $C$ is rigid as $X$ is not
uniruled and so $C$ cannot deform fixing a point. Here by ``rigid'', we mean that there are only finitely many stable maps
$f: \widehat{C}\to X$ of genus $1$ such that $f_* \widehat{C} = C$ and more precisely, if we consider the moduli space
of all stable maps of genus $1$ to $X$ whose image passes through $p$, then this moduli space has dimension $0$ at the
maps $f: \widehat{C}\to X$ with $f_* \widehat{C} = C$. We will now show that this is impossible. 

We begin by extending the family $\calX/B$ to a two-dimensional family, which we still denote by $\calX$, over $U = \Spec T$, where
\begin{itemize}
\item $T$ is a normal local ring of dimension $2$,
\item $B$ has codimension one in $U$,
\item $\calX$ is a family of K3 surfaces over $U$, 
\item $\calX\times_T \overline{K(T)}$ is a projective K3 surface with Picard rank $1$,
\item $\calL = \calD+\mathscr{R}$ extends to a big and nef divisor $\calH$ on $\calX$, whose restriction to $\calX_B = \calX\times_U B$ over
$B\subset U$ is $\calL$, and
\item $B$ is the locus of K3 surfaces, where the class $R$ deforms.
\end{itemize}
Note that $\calH$ is not necessarily ample on $\calX$ and hence $\calX$ is not necessarily projective over $U$.

We also extend $P$ to a section of $\calX/U$.  Meanwhile, from Corollary \ref{cor:def}, the family $\mathscr{C}$ spreads
to a family of curves over $U$ after a finite base change, which we still denote by $\mathscr{C}$, such that
$\mathscr{C}_\eta \subset \calX_\eta$ is an integral genus 1 curve passing through $P_\eta$. We use
$\mathscr{C}_{B,\eta}$ and $\calX_{B,\eta}$ to denote the generic points of $\mathscr{C}_B$ and $\calX_B$, respectively.
Note that $\mathscr{C}_{B,\eta}$ is integral.

Let us consider the moduli space of stable maps of genus $1$ to $\calX$, whose images intersect $P$ (see \cite{ak}). This
moduli space has an irreducible component $\mathcal M$, flat over $U$, whose generic point $\mathcal M_\eta$ represents
the normalisation $\nu: \mathscr{C}_\eta^\nu\to \mathscr{C}_\eta$ of $\mathscr{C}_\eta$. Since the stable map over
$C = \mathscr{C}_0$ is rigid by our analysis of $C$, we actually have $\mathcal M\cong U$. Hence there exists a flat
family $f: \widehat{\mathscr{C}}\to \calX$ of stable maps of genus $1$ over $U$ such that $f_* \widehat{\mathscr{C}} =
\mathscr{C}$.

Let $q$ be the node of $\widehat{C} = \widehat{\mathscr{C}}_0$. The following is well known.

\begin{Lemma}\label{lem:mpt0}
    Let $\mathscr{C}\to U$ be a generically smooth family of stable curves over an integral scheme $U$, and assume that
    there is a point $0\in U$ such that the fibre $\mathscr{C}_0$ is singular. Then there exists a codimension 1 locus
    $0\in W\subset U$ over which every fibre is singular.
\end{Lemma}
\begin{Remark}
    In Propositions \ref{prop: marked point trick g=0} and \ref{prop: marked point trick g=1} we will prove a
    generalisation of this, called the marked point trick, in the case where $U$ is a surface and the family is not
    stable.
\end{Remark}

That is to say, $q$ deforms in codimension one, i.e., there exists $W\subset U$ of codimension one such that
$0\in W$ and $\widehat{\mathscr{C}}_W = \widehat{\mathscr{C}}\times_U W$ is singular along a section $Q$ of
$\widehat{\mathscr{C}}_W/W$ with $q = Q_0$. In other words, $W$ is the discriminant locus of the family
$\widehat{\mathscr{C}}/U$.

To conclude the proof of Theorem \ref{thm: elliptic curve non-primitive class}, let now $\widehat{C} = \Sigma_1\cup
\Sigma_2$, where $\Sigma_1$ and $\Sigma_2$ are the two components of $\widehat{C}$
meeting at $q$ with $f_* \Sigma_1 = \Gamma$ and $f_* \Sigma_2 = R$. As fibres over $W$ are deformations of
$\widehat{C}$, which is furthermore reducible with 1 node, $\widehat{\mathscr{C}}_W$ also has two irreducible
components, both flat over $W$ and containing $\Sigma_1$ and $\Sigma_2$, respectively. Hence the class $R$ necessarily deforms in the
family $\calX_W$ over $W$ since one of these two flat components parametrises deformations of $f:\Sigma_2\to R$. As we
had chosen the embedding $B\subset U$ in such a way that $B$ was the Noether--Lefschetz locus of $\OO(R)$, it must be the
case that $W=B$. But by construction $\widehat{\mathscr{C}}_{B,\eta}$ is smooth since $\mathscr{C}_{B,\eta}$ was chosen
to be an integral genus 1 curve. This gives the required contradiction, proving therefore that $C$ must be integral.
\end{proof}

\subsection{Curves of geometric genus 0 and 1}
\label{subsec:ellcurves}

We now extend some of the results of the previous sections to prove more existence results for curves of geometric genus 0 and 1.
In particular, we will extend the known existence results of rational curves on elliptic K3 surfaces to genus 1. In the
following, recall that an isotrivial elliptic fibration on a K3 surface is one such that all smooth fibres are
isomorphic.

\begin{Theorem}(\cite{btrational}, \cite{Tayou} $+\varepsilon$)
\label{thm: bt tayou}
    Let $X$ be a nonuniruled K3 surface over an algebraically closed field $k$ of characteristic $p\geq0$, and
    let $g\in\{0,1\}$. Assume that any of the following holds:
    \begin{itemize}
        \item $g=0$ and either $X$ is nonisotrivially elliptic or $X$ is elliptic and $p\neq2,3$;
        \item $g=1$ and $X$ is elliptic;
        \item $X$ has infinite automorphism group. 
    \end{itemize}
    Then there exists a sequence of integral curves $C_i$ of geometric genus $g$ such that for any ample
    divisor $H$ on $X$,
     \[\lim_{i\to\infty} H.C_i=\infty.\]
\end{Theorem}

\begin{proof}
For $g=0$, if $X$ has infinite automorphisms or has a non-isotrivial elliptic fibration, this was proved in
\cite{btrational} for $p=0$ and \cite{Tayou} over algebraically closed fields of arbitrary characteristic. As pointed
out in \cite[Remark 6]{bht}, the isotrivial elliptic case in characteristic zero has also been dealt with in
\cite{hassett}, whereas the positive characteristic case where $p>3$ was dealt with in \cite{Tayou} (cf.\ Section \ref{sec: pos char}). We now
prove the theorem for $g=1$ following the argument in \cite{Tayou}.

Suppose that $X$ is elliptic. Then there is an elliptic fibration $\pi_X: X\to \PP^1$. Let $d$ be the positive integer
so that 
\[\coker(\Pic(X)\to \ZZ : D\mapsto D.X_t)\cong \ZZ/d\ZZ\]
for a general point $t\in \PP^1$. More precisely, $d$ is the
greatest common divisor of $DF_X$ for all $D\in \Pic(X)$, where $F_X$ is a fibre of $\pi_X$. We now briefly recall the argument in
\cite[Sec. 3]{Tayou}. For $p=\chara(k)$, we can write $[X]=\alpha+\alpha_p$ for the class of $X$ in the Brauer group
$\operatorname{Br}(J(X))$ of its Jacobian, where $\alpha$ is $p$-torsion free and $\alpha_p$ has order $p^s$ for some
$s\geq0$ (i.e., if $p=0$ then $[X]=\alpha$ and we take $\alpha_p=0$ in what follows). The prime-to-$p$ part of the
Brauer group is a divisible group, so by choosing any integer $n$ so that $n\equiv1\mod p^s$, there is a class
$\alpha_n$ so that $n(\alpha_n+\alpha_p)=\alpha+\alpha_p$. The corresponding K3 surface $\pi_Y:Y:=Y_n\to\PP^1$ is
elliptic and such that
\[
{\rm{coker}}\left(\Pic(Y)\to \Pic(Y_t)\right) \,=\, \ZZ/(nd)\ZZ
\]
for $t\in\PP^1$ general and so that there is a dominant rational map $f_n: Y\dashrightarrow X$ over $\PP^1$, i.e.,
preserving the elliptic fibrations. Our aim is now first to find a genus 1 multisection $C$ of $\pi_Y$ of degree
bounded in terms of $nd$ and then show that $f_n$ maps $C$ birationally onto its image.

If there is a minimal nef divisor $D$ on $Y$ such that $D F_Y = nd$ for a fibre $F_Y$ of $\pi_Y$, then we obtain the required
genus 1 curve from Proposition \ref{prop: integral curves in min nef divisor}. Otherwise, from the minimal nef
decomposition of Lemma \ref{lem: minimal nef K3} and the fact that $nd$ is the minimal multisection fibre-degree of $Y$,
there must exist a $(-2)$-curve $D$ on $Y$ such that $D F_Y = nd$. Clearly, $D$ and $F_Y$ generate a primitive
sublattice $\Lambda$ of $\Pic(Y)$ and both $D+F_Y$ and $2D+F_Y$ are big and nef as long as $nd \ge 5$. When we deform
$Y$ to a general K3 surface with Picard lattice $\Lambda$, there is an integral genus 1 curve in $|2D+F_Y|$ from Theorem
\ref{thm: appendix}.(A3) applied to $(D+F_Y)+D$ for $n$ sufficiently large.  Consider now the deformations of this
genus 1 curve over a two-dimensional family of K3 surfaces as in the proof of Theorem \ref{thm: elliptic curve
non-primitive class}, where the very general member has Picard rank one. If the degeneration of this curve to $Y$ is
integral then we are done. Otherwise, exactly as in the final step of Theorem \ref{thm: elliptic curve non-primitive class}, the locus in the family of curves
where the fibres split into reducible curves is divisorial. 
The curve must split into a sum $\Gamma+E$ on $Y$, of class $2D+F_Y$, where $\Gamma$ is an integral geometric genus one
curve and $E$ is a sum of rational curves. If $\Gamma$ is a multisection of $\pi_Y$, then we are done as we proceed to
the next step with the curve $\Gamma$ itself. If not, then as it is integral, $\Gamma$ must be a fibre. Hence $E=2D$
implying that $\Gamma,E$ span the same lattice as $D,F_Y$, so the same argument as in the proof of Theorem \ref{thm:
elliptic curve non-primitive class} gives a contradiction.

So there is an elliptic curve $C$ on $Y$ with $C Y_t =m\geq nd$ and $C$ varies in a one-parameter family, parameterised by a
curve $B$. We claim that $f_n$ maps $C_b$ birationally onto its image for $b\in B$ general. Otherwise, for $b\in B$ and
$t\in \PP^1$ general and for all pairs of points $p, q\in C_b\cap Y_t$, $p-q$ must be torsion in $\Pic(Y_t)$. Then we
have $C_b Y_t = m r_b + \tau$ in $\Pic(Y_t)$ for a point $r_b\in C_b\cap Y_t$ and some $\tau\in
\Pic(Y_t)_\text{tors}$. When $b$ varies in $B$, $C_b Y_t$ is constant in $\Pic(Y_t)$. Thus, $r_b$ is a fixed point on
$Y_t$. That is, all $C_b$'s meet a general fibre $Y_t$ at a fixed point. This is only possible if $C_b$'s are the same curve
for all $b\in B$, which contradicts our assumption on $C_b$. Therefore, $f_n$ maps $C$ birationally onto its image if
$C$ is a general member of a one-parameter family of elliptic curves on $X$. Consequently, $f_n(C)$ is a
multisection of $\pi_X$ of degree $m$ and genus $1$. By increasing $n$, we obtain an infinite sequence of elliptic
curves on $X$ of increasing degrees. This proves the theorem when $X$ is elliptic.

Suppose that $X$ is not elliptic and $|\Aut(X)| = \infty$. If there are infinitely many minimal nef divisors $D$ on $X$,
we are done by Proposition \ref{prop: integral curves in min nef divisor}. Otherwise, there are finitely many minimal
nef divisors $D_1, D_2, \ldots, D_m$ on $X$. Let $D = \sum_{i=1}^m D_i$. Obviously, $\{D_i\}$ is non-empty, and
since $X$ is not elliptic, $D$ is big and nef.

By Lemma \ref{lem: minimal nef K3}, $\Pic(X)$ is generated by $D_1,D_2,\ldots,D_m$ and finitely many $(-2)$-curves on $X$.
So there exist a constant $N$ such that $\Pic(X)$ is generated by $D_1, D_2, \ldots, D_m$ and all $(-2)$-curves $R$ on $X$
satisfying $DR \le N$. Let us assume that
\[
\Pic(X) = \Big\{
\sum_{i=1}^m a_i D_i + \sum_{j=1}^n b_j R_j: a_i,b_j\in \ZZ
\Big\}
\]
where $R_1, R_2, \ldots, R_n$ are the $(-2)$-curves on $X$ satisfying $DR_j\le N$.

Obviously, every automorphism $\sigma$ of $X$ permutes the divisors in $D_1,D_2,\ldots,D_m$. Therefore, $\sigma^* D = D$
and hence $\sigma$ also permutes the divisors in $R_1, R_2,\ldots, R_n$.  Thus we have a group homomorphism $\rho:
\Aut(X)\to \Sigma_m \times \Sigma_n$. Since $\Aut(X)$ is infinite, so is the kernel $\ker(\rho)$ of $\rho$. On the other
hand, every $\sigma\in \ker(\rho)$ preserves $D_1,D_2,\ldots,D_m,R_1,R_2,\ldots,R_n$ and hence $\Pic(X)$. So
$\ker(\rho)$ lies in the kernel of $\Aut(X)\to \Aut(\Pic(X))$, which is a contradiction as the latter is finite from
\cite[Proposition 5.3.3]{huybrechts}.
\end{proof}

If a rational curve on a K3 surface $X$ moves, then we already noted that it is uniruled and supersingular. In view of
the expected unirationality of supersingular K3 surfaces, this begs the question whether supersingular K3 surfaces
contain infinitely many rigid rational curves. Since smooth rational curves on K3 surfaces are rigid (see also
\cite{iil}), one may even ask whether supersingular K3 surfaces contain infinitely many smooth rational curves. This is
known to experts (cf.\ \cite[Theorem 6.1]{nikulinelliptic}), but we include a proof for completeness.

\begin{Proposition}
\label{prop: supersingular OK}
If $X$ is a supersingular K3 surface over an algebraically closed field $k$ with $\chara(k)=p$, then $X$ contains
infinitely many smooth rational curves $R_i$ so that $\lim_{i\to\infty}H.R_i=\infty$ for any ample class $H$. If
$p>3$, then $X$ also contains infinitely many integral curves $C_i$ of geometric genus 1 so that
$\lim_{i\to\infty}H.C_i=\infty$.
\end{Proposition}
\begin{proof}
Rudakov--Shafa\-revich \cite[p.156]{rudakovshafarevich2} have given a complete description of the N\'eron--Severi
lattices $N_{p,\sigma}$ of supersingular K3 surfaces in terms of the characteristic $p$ and the Artin-invariant
$\sigma$. Using the rank four sublattice $H_p$ (see \cite[\S 3]{shimada} for a description), we see that $N_{p,10}$
contains a $(-2)$-class. Since $N_{p,10}\subset N_{p,\sigma}$ for every $\sigma$, this implies that $N_{p,\sigma}$
contains a $(-2)$-class for every $p$ and $\sigma$.  If $L^2=-2$ for some line bundle $L$, then $X$ must contain a
smooth rational curve as a component of a section of $L$ or $-L$.

Now, the existence of one smooth rational curve implies that there are infinitely many as seen from the following. It is known
that a supersingular K3 surface has infinite automorphism group by \cite{kondoshimada}. Also, from \cite[Theorem
4.1]{kovacsrevisited} the cone of curves is spanned by $(-2)$-classes. One concludes now by the same argument as in the
end of the proof of Theorem \ref{thm: bt tayou} using \cite[Proposition 5.3.3]{huybrechts}: if there were only finitely
many smooth rational curves, then there would be infinitely many automorphisms fixing the Picard group, in particular a
polarisation which is a contradiction.

Denote by $R_i$ an infinite sequence of smooth $(-2)$-curves on $X$. To see that $H.R_i\to\infty$, we argue as suggested
to us by Tayou. If $H.R_i$ is bounded, then the projection \[L_i:=R_i - \frac{H.R_i}{H^2}H\] of $R_i$ onto the $H^\perp$
plane satisfies $C<L_i^2<0$ for some constant $C$ and all $i$. But $H^\perp$ is negative definite, so there can be only
finitely many lattice points of bounded negative self-intersection, implying that infinitely many of the $L_i$'s are
proportional. This contradicts that the $R_i$'s are distinct smooth $(-2)$-curves.

As far as genus 1 curves are concerned, from \cite[p.156]{rudakovshafarevich2} we see that every supersingular K3
surface contains a copy of a (twisted) hyperbolic plane in its Picard lattice $N_{p,\sigma}$, hence at least two
primitive isotropic vectors. From \cite[Theorem 5.1]{nikulinelliptic}, since the automorphism group is infinite, it
follows that $X$ has infinitely many primitive nef isotropic vectors. From \cite[Proposition 2.3.10]{huybrechts}, if
$p>3$ then each such class gives a smooth curve of geometric genus 1 (whereas for small characteristic we cannot rule
out that all these classes give quasi-elliptic fibrations).
\end{proof}

\begin{Corollary}
 \label{cor: rigid bogomolov}
    Conjecture \ref{conj: inf many rational curves} is equivalent to the a priori stronger conjecture that every K3
    surface over an algebraically closed field contains infinitely many rational curves that are rigid.
\end{Corollary}

\section{Regenerating curves}
\label{sec: regeneration}

Throughout the whole of this section, we consider the following situation: let $S$ be the spectrum of a DVR with
algebraically closed residue field $k$, and let
\[
   \begin{tikzcd}\calX\ar{r}& S\end{tikzcd}
\]
be a \textit{family of K3 surfaces}, i.e., a smooth and proper morphism of algebraic spaces, whose
geometric fibres are K3 surfaces. As usual, we denote by $0,\eta\in S$ the closed and generic point, respectively,
by $\calX_0,\calX_\eta$ the fibres, and by $\overline{\calX_\eta}$ the geometric generic fibre.

\begin{Definition}
    \label{def: regeneration}
    Let $C_0$ be an integral curve on the special fibre $\calX_0$. A {\em regeneration} of $C_0\subset\calX_0$ is a
    geometrically integral curve $D_\eta\subset\calX_\eta$ that has the same geometric genus as $C_0$ and such that the
    specialisation to $\calX_0$ contains $C_0$.
\end{Definition}

\begin{Remark}
Since $C_0$ and $D_\eta$ are assumed to be of the same geometric genus, it follows that the specialisation of $D_\eta$
on $\calX_0$ is equal to the union of $C_0$ and a (possibly empty) union of {\em rational} curves.
\end{Remark}

A regeneration should be thought of as an inverse of a degeneration or a specialisation. One of the main points of this
definition is that we do {\em not} require the line bundle $\OO_{\calX_0}(C_0)$ on $\calX_0$ to extend to
$\overline{\calX_\eta}$.

\subsection{The regeneration theorem}

We now come to one of the main techniques of this article. A regeneration result for rational curves was
already envisioned in \cite{bht}, a weak version of it was already established in \cite{liliedtke}
and we refer to \cite[\S 13]{huybrechts} for an overview. We do not know whether the assumption of non-uniruledness is
necessary in the following (see Remark \ref{rem:pointsissue}).

\begin{Theorem}
    \label{thm: regeneration}
    Let ${\calX}\to S$ and $C_0$ be as above. Assume that $C_0$ deforms in the expected dimension in $\calX_0$ and that
    $\calX_0$ is not uniruled. Then after possibly replacing $S$ by some finite extension,
    \begin{enumerate}
        \item there exists a regeneration $D_\eta\subset\calX_\eta$ of $C_0$, and
        \item moreover, if the line bundle $\OO_{\calX_0}(C_0)$ extends from $\calX_0$ to $\calX_\eta$, then we may
        assume that $C_0$ is the specialisation of $D_\eta$ and that $\OO_{\calX_\eta}(D_\eta)$ specialises to
        $\OO_{\calX_0}(C_0)$.
    \end{enumerate}
\end{Theorem}

\begin{Remark}
    Concerning the assumptions and conclusions, we refer to Proposition \ref{prop:defexpdim} for conditions that ensure
    that a curve deforms in the expected dimension. By this result, $C_0$ always deforms in the expected dimension if
    $\chara(k)=0$, but Example \ref{ex:supersingular} shows that it may fail if $\chara(k)>0$.
\end{Remark}

We start by proving the following, which is more or less well known to experts. In particular, it implies that curves
whose classes lift can be easily regenerated, as the second assertion of the theorem claims.
\begin{Lemma}
\label{lem: assertion two}
    If $C_0\subset\calX_0$ is an integral curve so that $\OO_{\calX_0}(C_0)$ is in the image of the injective
    specialisation map \[\begin{tikzcd}\rm{sp}: \Pic(\calX_\eta)\ar{r}&\Pic(\calX_0),\end{tikzcd}\] then there exists a
    regeneration of $C_0$ after passing to a finite extension of $S$
\end{Lemma}
\begin{proof}
If the class of $C_0$ in $\Pic(\calX_0)$ extends to $\Pic(\calX_\eta)$, then as $\Pic(\calX)\cong \Pic(\calX_\eta)$,
there exists a relative line bundle $\calL$ on the family $f:\calX\to\Spec S$ such that the restriction $\calL_0$
of $\calL$ to $\calX_0$ is isomorphic to $\OO_{\calX_0}(C_0)$. Let $\widetilde{C}_0\to C_0$ be the normalisation of
$C_0$. Composing with the embedding $C_0\to\calX_0$, we obtain a stable map $h:\widetilde{C}_0\to\calX_0$ of genus
$g$.

By Corollary \ref{cor:def}, the moduli space of stable maps $\overline{\mathcal{M}}_g(\calX,\calL)\to\Spec S$ (resp.,
$\overline{\mathcal{M}}_g(\calX_0,\calL_0)\to\Spec k$) is of dimension $g+1$ (resp., $g$) for every component through
$[h]$.  This implies that we can extend $h$ as a stable map of genus $g$ from $\calX_0$ to $\calX_\eta$.  Since the
image of $h$ is irreducible, so is every lift to $X_\eta$. In particular, there exists a geometrically integral curve
$D_\eta$ on $\calX_\eta$, whose specialisation to $\calX_0$ is equal to $C_0$. By construction,
$\OO_{\calX_\eta}(D_\eta)$ is isomorphic to the restriction of $\calL$ to the generic fibre $\calX_\eta$ and thus, the
class of $D_\eta$ in $\Pic(\calX_\eta)$ specialises to the class of $C_0$ in $\Pic(\calX_0)$.
\end{proof}

We will now need the following reduction step.

\begin{Lemma}
  \label{lem: a reduction}
   In order to prove Theorem \ref{thm: regeneration}, we may choose a primitive ample line bundle $\calH_\eta$ on $\calX_\eta$
   whose specialisation $\calH_0$ to $\calX_0$ is big and nef.
\end{Lemma}
\begin{proof}
Choose an ample line bundle $\calH_\eta$ on $\calX_\eta$ and assume that its specialisation $\calH_0$ to $\calX_0$ is
not big and nef. We distinguish two cases.

If $\calH_0.C_0<0$, then $C_0$ is necessarily a smooth rational curve on $\calX_0$, i.e., a $(-2)$-curve. After replacing
$S$ by some finite extension, Theorem \ref{thm: bogomolov mumford} gives an effective divisor $\sum_i n_iD_i \in
|\calH_\eta|$ so that every $D_i$ is a geometrically integral rational curve. Let $D_i'\subset\mathcal{X}_0$ be the specialisation
of $D_i$. Since $\calH_0.C_0<0$, there exists some $i$, say $i=1$, so that $D_1'.C_0<0$, which means that $D_1'$ is the union
of rational curves one of whose components is $C_0$. This implies that $D_1$ is a regeneration of $C_0$ and we are done.

Assume now that $\calH_0.C_0\geq0$. Let $S^{\rm h}\to S$ be the Henselisation of $S$ and let $\calX^{\rm
h}:=\calX\times_SS^{\rm h}\to S^{\rm h}$ be the base change. Then, from \cite{kollarflops} over the complex numbers and
\cite[Proposition 4.5]{liedtkematsumoto} for positive and mixed characteristic, there exists a birational map of smooth
and proper algebraic spaces over $S^{\rm h}$
\[
   \begin{tikzcd}\varphi : \calX^{\rm h} \ar[dashed]{r} & \calX^{\rm h+},\end{tikzcd}
\]
such that the specialisation of $\calH_\eta$ to the special fibre $\calX^{\rm h+}_0$ is big and nef. More precisely,
$\varphi$ is a finite sequence of flops in $(-2)$-curves, that is, there exists a union of finitely many $(-2)$-curves
$Z\subset\calX_0^{\rm h}$ such that the restriction of $\varphi$ to $\calX^{\rm h}\backslash Z$ is an isomorphism onto
its image. Moreover, the irreducible components of $Z$ are precisely the curves that intersect $\calH_0$ negatively.
As $\calH_0.C_0\geq0$, the curve $C_0^{\rm h}$ is not contained in $Z$, so by assumption, Theorem \ref{thm:
regeneration} is true for $\calX^{\rm h+}$ and thus, there exists a regeneration $D_\eta^{\rm h+}$ on $\calX^{\rm
h+}_\eta$, such that $C_0^{\rm h+}:=\overline{\varphi(C_0^{\rm h}\backslash Z)}$ is contained in the specialisation to
$\calX_0^{\rm h+}$. If we set $D_\eta^{\rm h}:=\varphi_\eta^{-1}(D_\eta^{\rm h+})$, then it is a regeneration of
$C_0^{\rm h}$.  Since it is the Henselisation (rather than a completion), the extension $S^{\rm h}\to S$ is algebraic.
Therefore, the curve $D_\eta^{\rm h}$ is already defined after a finite field extension $L/\kappa(\eta)$. If $S'$
denotes the integral closure of $S$ in $L$ and $\calX':=\calX\times_SS'$, then the curve $C_0\subset\calX_0$ admits a
regeneration on $\calX'$. This establishes Theorem \ref{thm: regeneration} in this case.
\end{proof}

\begin{specialproof}[Proof of Theorem \ref{thm: regeneration}]
From Lemma \ref{lem: a reduction} we may assume that there is a line bundle $\calH$ on $\calX$ which is ample and
primitive on $\calX_\eta$, say of degree $\calH_\eta^2=2d$, and big and nef on $\calX_0$. Also, from Lemma \ref{lem:
assertion two}, we may assume that $\OO(C_0)$ does not lift to $\Pic(\calX_\eta)$.

Like in Construction \ref{construction:defspace}, let $K=\overline{k(0)}$ and $A$ be either $K$ or $W(K)$, and denote by
$\widehat{T}(\calH_0):=\rm{Def}(\calX_0/A, \calH_0)$ and $\widehat{T}(\OO(C_0)):=\rm{Def}(\calX_0/A, \OO(C_0))$ the loci in $\rm{Def}(\calX_0/A)$
where the classes $\OO(\calH_0), \OO(C_0)$ deform respectively. As they are formal Cartier divisors, flat over $A$, and
they meet non-trivially in the closed fibre (which corresponds to $\calX_0$), their intersection
$\widehat{T}(\calH_0,\OO(C_0)):=\widehat{T}(\calH_0)\cap \widehat{T}(\OO(C_0))$ must be of codimension two in $\rm{Def}(\calX_0/A)$, and from
flatness their corresponding generic fibres also intersect in codimension two over the characteristic zero generic point
of $\rm{Def}(\calX_0/A)$. Note though that $\calX_\eta$ is not a point in $\widehat{T}(\calH_0,\OO(C_0))$ as we have assumed that the
class of $C_0$ does not extend over the whole of $S$.

As $\calH_0$ is big and nef, if $\calH_0.C_0=0$, then $C_0$ necessarily is a $(-2)$-curve from the Hodge Index Theorem.
The formal families over $\widehat{T}(\calH_0), \widehat{T}(\calH_0,\OO(C_0))$ are algebraisable (using $\calH_0$ on the
special fibre) from the discussion in Construction \ref{construction:defspace}.  Hence after passing to a finite
extension we obtain a family $\mathcal{Y}_1\to T_1$ (resp., $\mathcal{Y}\to T$ for $\widehat{T}(\calH_0)$) of K3
surfaces. Denote by $\calH, \calC$ the corresponding line bundles on $\mathcal{Y}, \mathcal{Y}_1$. A general fibre of
$\mathcal{Y}_1\to T_1$ is a characteristic zero projective K3 surface $(X,H,C)$, deformation equivalent to
$(\calX_0,\calH_0,C_0)$, where $H^2>0$, $C$ is an integral curve in $X$, and $H.C=\calH_0.C_0\geq0$. From \cite[Corollary
2.1.5]{huybrechts}, $H$ will in fact be big and nef since the contrary would force the existence of a $(-2)$-curve on $X$
which would increase the Picard rank of $X$ by one and from Proposition \ref{prop: lattice polarised stack smooth} such
surfaces lie in higher codimension.

As there are no $(-2)$-curves on $X$ other than possibly $C$, by the argument at the end of the previous paragraph, we
may even choose $n$ so that $nH-C$ is also ample. From Theorem \ref{thm: appendix} (see also the paragraph directly
after the statement of Theorem \ref{thm: nodal curves}) there is an integral rational curve $R\in|nH-C|$ with unramified
normalisation for all $n$ large enough.  Then the curve $R+C\in|nH|$ is a connected curve, even though the intersection
of $C$ and $R$ may not be transverse. We now build a stable map $h:D\to X$ constructed by gluing the normalisations
$\widehat{R},\widehat{C}$ of $R$ and $C$ at one point, and whose image is $C\cup R\in|nH|$. Note that $h$ specialises to
a stable curve $h_0$ mapping to $\calX_0$ whose image is $C_0$ and a union $R_0$ of rational curves, possibly with
multiplicities. Observe also that as $R$ is a rational curve in a K3 surface in characteristic zero, deformations of $h$
in $\overline{\mathcal{M}}_g(X,nH)$ either come from deformations of $\widehat{C}\to X$, or will necessarily smooth the
unique node of $D$. In either case, the expected dimension will be $g=g(\widehat{C})$ from Proposition
\ref{prop:defexpdim}. Hence from Corollary \ref{cor:def}, any irreducible component of
$\overline{\mathcal{M}}_g(\mathcal{Y}/T,n\calH)$ containing $h$ has dimension exactly $\dim T+g$ and surjects onto $T$.

To simplify the setup slightly, we may put (like in the proof of Theorem \ref{thm: elliptic curve non-primitive
class}) $\calX\to S$ and $(X,H,C)$ in a family $\mathcal{Y}\to U$ as follows:
\begin{itemize}
    \item $T\subset U$ the spectrum of a two-dimensional normal local ring,
    \item $\mathcal{Y}\to U$ is a family of K3 surfaces, 
    \item $\calX_\eta, \calX_0$ and $X$ are fibres of $\mathcal{Y}\to U$,
    \item the geometric generic fibre $\mathcal{Y}_{\overline\xi}$ has Picard rank 1 generated by $\calH$.
\end{itemize}
From the argument above, any irreducible component $M\subset \overline{\mathcal{M}}_g(\mathcal{Y}/U,n\calH)$ that also
contains the stable morphism $h$ above the point $X\in U$ surjects onto $U$. We would like to conclude by taking a
general point $(Y, h_Y)$ of $M$, where $Y$ is a general fibre of $\mathcal{Y}\to U$ and $h_Y:D_Y\to Y$ is a general
deformation of $h$ contained in $M$, and specialising it to a stable morphism $h_\eta: D_\eta\to \calX_\eta$. Indeed,
some irreducible component of the image of $h_\eta$ is now a candidate for a regeneration of $C_0$ as $h_\eta$
specialises to a stable map to $\calX_0$, but it is not clear that this is $h_0$ from above, nor that its
image even contains $C_0$. In fact, by taking the Stein factorisation of $M\to U$ and restricting to the component which
contains $(\calX_0, h_0)$ we may assume that the fibres of $M\to U$ are connected, so we have only proved that we can
choose $h_\eta$ whose specialisation to $\calX_0$ is deformation equivalent to $h_0$ from above, i.e., $h_\eta$ deforms
to a genus $g$ stable morphism to $\calX_0$ whose image is linearly equivalent to $C_0+R_0$, since the fibre in $M$
above $\calX_0$ could be positive dimensional (see Remark \ref{rem:pointsissue} and the analogous argument at the end of
the proof of \cite[Theorem 4.3]{liliedtke}).

Since $\calX_0$ is not uniruled, rational curves in $\calX_0$ are rigid so we can resolve this issue by fixing
general points on $C_0$ as follows. Let $g$ be the geometric genus of $C_0$ and choose $g$ general points $p_1, \ldots,
p_g\in C_0$. As the points are general and $C_0$ deforms in the expected dimension, $C_0$ is the only geometric genus
$g$ deformation of $C_0$ which passes through those points.  Since $\mathcal{Y}\to U$ is smooth, after an \'etale base
change we can spread the points $p_i$ now to sections $\sigma_1,\ldots, \sigma_g$ of $\mathcal{Y}\to U$ which also meet
the fibres $X, Y, \calX_\eta$ of $\mathcal{Y}\to U$ at $g$ general points.  Consider the moduli space with base
points (see \cite[Theorem 50]{ak}) \[\overline{\mathcal{M}}_g(\mathcal{Y}/U,n\calH, \sigma_1,\ldots,\sigma_g)\] whose
points parametrise stable genus $g$ curves mapping to fibres of $\mathcal{Y}\to U$ whose images meet all the $\sigma_i$.
Since deformations of $C_0$ deform in the expected dimension on all fibres of $\mathcal{Y}\to U$, there is an
irreducible component $M'$ of this space containing the point $h_0$, and by taking the Stein factorisation of $M'$ over
$U$ we may even assume that $M'\to U$ has connected fibres. Note now that the fibre above the point $\calX_0$ might
still be positive dimensional, e.g., if the specialisation $R_0$ of $R$ to $\calX_0$ contains multiple components, but
the images of all these stable morphisms must contain $R_0$ since $R_0$ does not deform in $\calX_0$ as it is not
uniruled. As we have fixed points, the images also contain $C_0$ by construction. In particular, some irreducible
component of the image of a stable morphism $h_\eta:D_\eta\to \calX_\eta$ in the fibre of $M'\to U$ over $\calX_\eta$
gives a regeneration of $C_0$.
\end{specialproof}

\begin{Remark}
As a consequence of the proof, we can even bound the degree of a regeneration $D_\eta$ of $C_0$. To this aim, let
$a=\calH_0^2$, $c=C^2_0$ and $b=\calH_0.C_0$, and let $(X,H,C)$ be a general deformation of $(\calX_0,\calH_0,C_0)$. Then $nH-C$
will be ample as soon as $(nH-C)^2>0$. From the Hodge Index Theorem, $b^2-ac>0$ so for \[n_0
>\frac{b+\sqrt{b^2-ac}}{a}\] $n_0H-C$ will be ample. Then Theorem \ref{thm: appendix} will produce rational
curves in $|(n_0+2)H-C|$. The regeneration will then be an irreducible component of a section of $(n_0+2)\calH_\eta$.
\end{Remark}

\begin{Remark}\label{rem:pointsissue}
At the end of the above proof, a technical argument was required so as to guarantee that it is indeed the curve $C_0$
that we are regenerating, as opposed to some deformation of $C_0$ plus a sum of rational curves. We give here an example
justifying the extra argument. Let $f:\calX_0\to\PP^1$ be a supersingular K3 surface with quasielliptic fibration whose
central fibre is the union of three $\PP^1$'s meeting at one point, and whose generic fibre is a cuspidal arithmetic
genus 1 curve. If $C_0=\PP^1$ is an irreducible component of the central fibre then it is rigid, yet it could be that
the nodal rational curve $R$ of the proof specialises in $\calX_0$ to the other two irreducible components of the
central fibre, so that the sum with $C_0$ moves in $\calX_0$. In this case, the fibre in $M$ above $\calX_0$ is
1-dimensional. Although in this case $\calH_0^2=0$ so that $\calH_0$ is not big, we could not rule out the type of
behavior exhibited in this example.
\end{Remark}

\begin{Remark}\label{rem:multiple}
In the proof we did not have to assume that multiples of the classes of $C_0$ and $\calH_0$ are independent in
the Picard group. In fact, it can happen in families of supersingular K3 surfaces over a DVR
that the specialisation morphism on Picard groups has non-trivial cokernel.
\end{Remark}

\section{Applications of the regeneration technique}
\label{sec: applications}

\subsection{To the existence of rational curves}

Although we will establish in Theorem \ref{thm: curves on complex k3} that there are infinitely many rational curves on
every algebraic K3 surface over an algebraically closed field, the proof is a somewhat technical reduction to various
cases as explained in the introduction. For posterity, we note that the following application of the regeneration
technique should hopefully lead to a cleaner proof of Conjecture \ref{conj: inf many rational curves}.

\begin{Corollary}
\label{cor: one reduction}
    Let $\calX\to S$ be a family of K3 surfaces over the spectrum of a DVR with algebraically closed residue field.
    If $\calX_0$ is not uniruled and satisfies Conjecture \ref{conj: inf many rational curves}, then so does the geometric
    generic fibre $\overline{\calX_\eta}$.
\end{Corollary}
\begin{proof}
Since any linear system only contains finitely many rigid rational curves, we may
assume that the infinitely many rational curves $R_i$ are all distinct classes in the Picard group. Since $\calX_0$ is not
uniruled, each of them can
be regenerated to $\overline{\calX_\eta}$ by Theorem \ref{thm: regeneration}, which gives the result in this case.
\end{proof}

As mentioned, the above would give a quicker proof of Conjecture \ref{conj: inf many rational curves} assuming the following
question has a positive answer (cf.\ \cite{costatschinkel, charles, elsenhansjahnel, cej}).
\begin{Question}
    Let $X$ be a K3 surface defined over $\overline{\QQ}$. Does there exist a prime $\idealp$ so that the
    reduction is smooth, satisfies $\rho(X_\idealp)\geq5$, and is not supersingular?
\end{Question}

\begin{Example}
\label{ex: CM}
    Let $X$ be a K3 surface with complex multiplication. Then $X$ is defined over $\overline{\QQ}$ and there exists a
    prime ideal $\idealp$ such that $X$ has good and supersingular reduction at $\idealp$ by \cite{Ito}. We refer to
    \cite{charles} for further discussion. Similarly, in \cite{tayouetal} it is proved that a K3 over a number field all
    of whose reductions are smooth has infinitely many reductions where the Picard rank jumps.
\end{Example}

In particular, we obtain further reduction steps, the first of which was already shown in \cite[Theorem 3]{bht}.

\begin{Corollary}
 \label{cor: reduction}
 There are the following reductions.
 \begin{enumerate}
  \item In order to prove Conjecture \ref{conj: inf many rational curves}
    for K3 surfaces in characteristic zero, it suffices to prove it for K3 surfaces that can be defined over number fields.
  \item In order to prove Conjecture \ref{conj: inf many rational curves} in general,
   it suffices to prove it for K3 surfaces that can be defined over finite fields.
 \end{enumerate}
\end{Corollary}

As another application of regeneration, we obtain a cleaner proof of the main result of \cite{liliedtke}.

\begin{Theorem}\label{thm: ll}
    Let $X$ be a K3 surface of odd Picard rank over an algebraically closed field. Then $X$ contains infinitely many
    rational curves.
\end{Theorem}
\begin{proof}
    By the Tate Conjecture, such an $X$ cannot be defined over a finite field as such a K3 has even Picard rank. By
    successive specialisations, we may assume that $X$ is defined over a global field $K$. Let $\OO_K$ be the ring of
    integers in $K$. Hence for any place $\idealp$ of good reduction, the geometric Picard rank of $X_{\idealp}$ jumps.
    From \cite[Theorem 4.1]{liliedtke} we may even assume that there are infinitely many specialisations which are not
    supersingular. For these, Theorem \ref{thm: bogomolov mumford} implies that there exists a rational curve whose class is
    not a specialisation of a class from $X$. Now let $\calX\to S\subset\Spec\OO_K$ be the smooth family all of whose
    fibres are not supersingular. For each $\idealp\in S$, Theorem \ref{thm: regeneration} gives a regeneration of this
    curve. We must now argue that these curves are distinct in $X_{\bar{K}}$, following a Hilbert scheme argument
    taken from \cite[Proof of Proposition 4.2]{liliedtke}. 

    If the degree of these rational curves with respect to some relatively ample divisor on $\calX$ is bounded by $N$,
    then the finite type Hilbert scheme $\Hom^{\leq N}(\PP^1, \calX)$ has a $\overline{\OO_K/\idealp}$-point for infinitely many
    $\idealp\in S$, hence by Chevalley's theorem also one over the generic point, meaning there is a rational curve in
    $X_{\bar{K}}$ which specialises to the aforementioned rational curves for those $\idealp\in $S. This gives the
    required contradiction since their classes were assumed to not be in the image of $\Pic(X_{\bar{K}})$ and hence the
    degrees are unbounded. We thus obtain infinitely many rational curves in $X$ which are distinct.
\end{proof}

\subsection{Generalisations}

In this section, we show that the assumptions of Theorem \ref{thm: regeneration} cannot be relaxed in various
situations. In the following examples, we show that the regeneration theorem does not always hold if the special fibre
is a uniruled K3 surface and the curve we want to regenerate is a rational curve that moves, and that it also does not hold for
families of non-algebraic K3 surfaces nor for Abelian surfaces.

\begin{Example}[Characteristic $p$]
\label{ex:supersingular}
 Let ${\calX}\to S$ be a family of K3 surfaces over the spectrum of a  DVR
 with algebraically closed residue field $k$ such that
\begin{enumerate}
    \item $\calX_0$ is unirational,
    \item the geometric generic fibre $\overline{\calX_\eta}$ is not unirational, and
    \item $k$ is uncountable.
\end{enumerate}
Since $\calX_0$ is unirational, it contains moving families of rational
curves.  In particular, since $k$ is uncountable, $\calX_0$ contains uncountably many rational curves.  On
the other hand, since the geometric generic fibre $\overline{\calX_\eta}$ is not
unirational, every linear system contains only a finite number of rational curves by a usual Hilbert scheme
argument.  Hence $\overline{\calX_\eta}$ contains only countably many rational curves.  In particular, there do exist
rational curves on $\calX_0$ that are neither specialisations nor components of specialisations of rational
curves on $\overline{\calX_\eta}$.

Such families do exist: explicit examples of a unirational K3 surface $X$
in characteristic $p\geq3$ can be found in \cite{ShiodaExample, ShiodaKummer}.
After possibly replacing $k$ by an uncountable and algebraically closed extension,
we choose a smooth projective family $\calX\to S$ of K3 surfaces,
where $S$ is a finite extension of $W(k)$, with
special fibre isomorphic to $X$, which exists by \cite{deligne}.
Clearly, the geometric generic fibre is not unirational as it is a K3 surface in characteristic zero.
\end{Example}

\begin{Example}[Non-algebraic K3 surfaces]
 Let $\Delta:=\{z\in\CC: |z|<1\}$ be the open disk
 and let $f:\calX\to\Delta$ be a complex analytic family of
 K3 surfaces.
 Assume that
 \begin{enumerate}
  \item $\calX_0:=f^{-1}(0)$ is algebraic and satisfies Conjecture \ref{conj: inf many rational curves}, and
  \item the algebraic dimension (transcendence degree of
  the field of meromorphic functions)
  of the general member of this family is zero.
 \end{enumerate}
  Then, $\calX_0$ contains infinitely many smooth curves,
 whereas the general member of $f$ contains
 only finitely many curves (see, e.g., \cite[Theorem IV.8.2]{BHPV}).
 In particular, regeneration fails for all but finitely many curves
 on $\calX_0$.

  Such families do exist:
  we will see in Theorem \ref{thm: curves on complex k3}
 that every algebraic K3 surface satisfies Conjecture \ref{conj: inf many rational curves}, but for the time
 being, we may take one of the surfaces from Theorem \ref{thm: bt tayou}
 or Theorem \ref{thm: ll}.
 Moreover, we note that a very general deformation of any
 algebraic K3 surface inside its Kuranishi space
 satisfies the second assumption.
\end{Example}

\begin{Example}(Abelian surfaces)
    Let ${\calX}\to S$ be a family of Abelian surfaces over the spectrum of a DVR with
    algebraically closed residue field $k$, such that
    \begin{enumerate}
    \item $\calX_0$ is the product of two elliptic curves $E_1\times E_2$ and
    \item the geometric generic fibre $\overline{\calX_\eta}$ is an Abelian surface
      that is not isogenous to the product of two elliptic curves.
    \end{enumerate}
    Let $h:E_1\to {\cal X}_0$ be the immersion corresponding to $E_1\times\{O\}\subset\calX_0$ and
    $O\in E_2$.
    Then, there exists no
    regeneration of $E_1$ to $\overline{\calX_\eta}$: a regeneration would be a curve $D\subset\overline{\calX_\eta}$,
    whose normalisation $\widehat{D}\to D$ is a smooth curve of genus one.
    But then, we obtain a morphism of Abelian
    varieties from the elliptic curve $\widehat{D}$ to $\overline{\calX_\eta}$, which implies that
    $\overline{\calX_\eta}$ is isogenous to the product of two elliptic curves, contrary to our assumptions.

    Such families do exist: a very general deformation of the product of two elliptic curves inside
    the moduli space $\mathcal{A}_2$
    of principally polarised Abelian surfaces has this property as a very general surface in $\mathcal{A}_2$ is the Jacobian
    of some genus two curve and the Jacobian of a genus two curve $C$ is isogenous to the product of two elliptic curves
    if and only if $C$ admits a degree two morphism onto an elliptic curve. However, such genus two curves are not very
    general.

    Note that the analogue of the first named author's result \cite{chen} concerning existence of nodal curves of admissible
    genus $g\geq2$ has recently been solved in the case of abelian surfaces in \cite{klcm, klc}, and it would be
    interesting to know if a regeneration-type theorem exists using higher genus curves instead of rational ones in this
    case.
\end{Example}

\subsection{Enriques surfaces}
Next, we show that for families of Enriques surfaces, Theorem \ref{thm: regeneration} partly holds true in genus zero,
i.e., for rational curves.

First, we want to show that part (2) of Theorem \ref{thm: regeneration} does not hold for families of Enriques surfaces.  To
state it, let us recall that an Enriques surface $X$ over an algebraically closed field is said to be {\em nodal}
(resp.,
{\em unnodal}) if $X$ contains at least one smooth rational curve (resp., does not contain smooth rational curves),
see also \cite{cdl, dk}.

\begin{Example}(Enriques surfaces)
  Let ${\calX}\to S$ be a family of Enriques surfaces, where $S$ is the spectrum of a DVR
  with algebraically closed residue field $k$ of characteristic $\neq2$, such that
\begin{enumerate}
    \item $\calX_0$ is nodal and
    \item the geometric generic fibre $\overline{\calX_\eta}$ is unnodal.
 \end{enumerate}
 If $R_0$ is a smooth rational curve on ${\calX}_0$, then the class of $R_0$ in $\Pic({\calX}_0)$ lifts to
 $\overline{{\calX}_\eta}$ since $h^2(\calX_0,\OO_{\calX_0})=0$ and by deformation theory of line bundles.
 However, $R_0$ cannot be deformed to $\overline{{\calX}_\eta}$ since it is unnodal.

Such families do exist: if $\chara(k)\neq 2$, then Enriques surfaces
form a $10$-dimensional moduli spaces over $k$ and the locus of nodal Enriques surfaces is of codimension one (see \cite{cdl, dk}).
Thus, if $Y$ is a nodal Enriques surface over $k$, then a general deformation of $Y$ over $S$ yields a family
${\calX}\to S$ as above.
\end{Example}

Second, part (1) of Theorem \ref{thm: regeneration} does hold true for families of Enriques surfaces and for rational
curves. Thus, despite the different ways rational curves behave and deform on K3 surfaces and Enriques surfaces, we
have the following analog of Theorem \ref{thm: regeneration}.
\begin{Theorem}
\label{thm: enriques}
Let ${\calX}\to S$ be a family of Enriques surfaces, where $S$ is the spectrum of a DVR with algebraically closed
residue field $k$ of characteristic $\neq2$, such that $\calX_0$ is not uniruled, and let $C_0\subset\calX_0$ be an
integral curve of geometric genus $g\leq1$. After possibly replacing $S$ by a finite extension, there exists a
geometrically integral curve $D_\eta\subset\calX_\eta$ of geometric genus $g$, such that $C_0$ is a component of the
specialisation of $D_\eta$.
\end{Theorem}
\begin{proof}
Let
\[
    \begin{tikzcd}\widetilde{{\calX}}
    := \Spec (\OO_{{\calX}}\oplus \omega_{{\calX}})
    \ar{r}{\pi}& \calX\ar{r}& S\end{tikzcd}
\]
be a canonical double cover of the family ${\calX}\to S$.
By our assumptions on the characteristic of $k$, it
follows that $\widetilde{{\calX}}\to S$ is a family of K3 surfaces.

Let $\widetilde{C}_0$ be a component of $\pi_0^{-1}(C_0)$. Since $\pi_0$ is \'etale and $g\leq1$, it is not difficult to
see that $\widetilde{C}_0$ is a curve of the same geometric genus as $C_0$ or the disjoint union of two such curves.
After possibly replacing $S$ by a finite cover, Theorem \ref{thm: regeneration} provides us with a geometrically integral
curve $\widetilde{D}_\eta\subset\widetilde{\calX_\eta}$ of genus $g$ such that one component of
$\widetilde{C}_0$ is a component of the specialisation of $\widetilde{D}_\eta$.
Then, $D_\eta:=\pi_\eta(\widetilde{D}_\eta)\subset\overline{\calX_\eta}$ is a geometrically integral curve of genus $\leq g$
such that $C_0$ is a component of its
specialisation.
Since $C_0$ has geometric genus $g$, the geometric genus of $D_\eta$ is at least $g$ and thus, equal to $g$.
\end{proof}

\section{The main theorem: Reductions}
\label{sec: main theorem 1}

In this and the next section, we are going to prove the following.

\begin{Theorem}\label{thm: curves on complex k3}
Let $X$ be a projective K3 surface over an algebraically closed field $k$ of characteristic zero, and let $g\geq0$ be an
integer. Then there exists an infinite sequence of integral curves $C_n$ of
geometric genus $g$ on $X$ such that for any ample divisor $H$ on $X$, 
\[
\lim_{n\to\infty} \deg_H C_n = \lim_{n\to\infty} H C_n = \infty.
\]
\end{Theorem}

This, of course, implies the existence of infinitely many rational curves on every projective K3 surface in
characteristic zero and settles the remaining cases from \cite{liliedtke, btrational}, and so on (see Section \ref{sec:
existence}). Note that for $g>0$ and $X$ a general K3, the above was also obtained by Nishinou \cite{nishinou}.

\begin{specialproof}[Proof of Theorem \ref{thm: curves on complex k3}]
We outline the proof and refer to the sections where each step is proved.

\begin{enumerate}
    \item In Corollary \ref{cor: gleq1 suffices}, we reduce the theorem to $g\le 1$. 
    \item In Corollary \ref{cor:odd rank} we prove the theorem for $g\le 1$ and odd Picard rank $\rho(X)$. When $g=0$, this was proved in
    \cite{liliedtke} (or Theorem \ref{thm: ll} using regeneration). Here we carry over their argument to $g=1$. This is
    done in Proposition \ref{prop: elliptic curve K3 char p reduction} by applying Theorem \ref{thm: elliptic curve
    non-primitive class}.
    \item We note that Theorem \ref{thm: bt tayou} has already dealt with elliptic K3 surfaces and K3 surfaces
     with infinite automorphism groups.
    \item In Section \ref{sec: main theorem 2}, we deal with the remaining cases where $\rho(X) = 2, 4$ and $X$ is neither elliptic nor has
    infinitely many automorphisms. This step is quite involved and has two main ingredients:
    \begin{enumerate}
        \item a method (the {\em marked point trick}) to show that the degeneration of certain integral curves on generic
        K3 surfaces to a special K3 surface remains integral,
        \item the existence of rational curves on generic K3 surfaces of rank $2$ given by Theorem \ref{thm: appendix} and the
        existence of rational curves on generic K3 surfaces with certain rank $4$ Picard lattices given
        by Theorem \ref{thm: curves on complex k3 picard rank 4}.
    \end{enumerate}
\end{enumerate}
\end{specialproof}

\begin{Remark}
    We note that we do not prove in general that there is a sequence of curves $C_n$ of geometric genus $g$ so
    that $C_n^2$ tends to infinity (which would imply $\lim_n H.C_n=\infty$ by the Hodge Index Theorem), so it could be that
    the arithmetic genus is actually bounded. Nevertheless, in some cases, e.g., $\rho\leq2$ and $g=0$, the proof does
    give this, and it would be interesting to establish whether this is true in general. 
\end{Remark}

\subsection{Curves of high genus}\label{subsec: higher genus}

We start with the easiest step that reduces the theorem to $g\leq1$. For a smooth and proper variety $Y$ over an
algebraically closed field and a divisor $D$ on $Y$, we denote by $V_{D,g}$ the \textit{Severi variety} parametrising
integral curves in the linear system $|D|$ of geometric genus $g$, which is a locally closed subset of the projective
space $|D|=\PP(\HH^0(Y,\OO_Y(D)))$. We denote by $\overline{V}_{D,g}$ the closure of $V_{D,g}$ inside $|D|$. The
following is an application of well-known results on the deformation theory of curves on K3 surfaces.

\begin{Lemma}\label{lem: genus 1 transverse}
Let $X$ be a K3 surface over an algebraically closed field of characteristic zero. Let $C,E\subset X$ be integral
curves where $E$ has geometric genus 1. Then $E$ deforms in a 1-dimensional family in $X$ and if
$v:\widehat{E}_t\to X$ is the composition with the normalisation of a general member of this family, then $v^*C$ consists
of $CE$ distinct points.
\end{Lemma}
\begin{proof}
From Proposition \ref{prop:defexpdim}, $E$ moves in a one-parameter family of genus 1 curves. Then by the
Arbarello--Cornalba Lemma used for example in Proposition \ref{prop:defexpdim} (in particular \cite[Theorem
2.5(i)]{dedieusernesi}), we see that composition with the normalisation $\nu: \widehat{E}_t\to X$ of $E_t$ is an immersion
into $X$ in the sense that it is an unramified morphism. 

We claim now that $\nu^* C$ consists of $m = CE$ distinct points. Note first that since $\nu$ is unramified, the normal
sheaf $\calN_\nu$ of $\nu$ is locally free (see Section \ref{subsec: stable maps}), and hence is $\OO_{\widehat{E}_t}$.
If $p\in X$ is a base point of the family $E_t$, then the proper transform $E'_t$ of $E_t$ on $X'$ the blow-up of
$X$ at $p$ has normalisation $\nu':\widehat{E}_t'\to X'$ which is also unramified, hence the normal sheaf $\calN_{\nu'}$
is also locally free but has degree $-1$ and hence no sections, contradicting the fact that $E'_t$ also moves in a one
dimensional family. Hence the family has no base points and so $E_t$ only meets $C$ at smooth points of $C$. Assume now that
the divisor $\nu^*C$ vanishes at a point $p$ with multiplicity $n\geq2$. As $E_t$ is general, this means that the space
of deformations of $E_t$ with tangency conditions at this point is at least one dimensional, but such deformations are
controlled by $\calN_\nu(-(n-1)p)$ (see \cite[Proposition 2.1.3]{harris}), which has negative degree. This proves the
claim.
\end{proof}

\begin{Remark}
The characteristic zero assumption in this lemma is crucial, as the following example (albeit for geometric genus 0
curves) shows:
let $E$ be a fibre of a quasi-elliptic fibration, that is, $|E|$ is a 1-dimensional
linear system, whose generic member has a cusp singularity.  
If $C$ is the ``curve of cusps'', that is, the closure of the non-smooth points of 
the general members of $|E|$, then every deformation of $E$ fails the above statement. 
\end{Remark}

To highlight what exactly goes wrong in positive characteristic, we give another proof of the above lemma. 

\begin{proof}[Second proof of Lemma \ref{lem: genus 1 transverse}]
From Theorem
\ref{theo:deffam}, $E$ deforms in a family of genus $1$ curves that cover $X$. So there exist dominant morphisms $\pi:
Y\to B$ and $f: Y\to X$, where $Y$ is a smooth projective surface, $B$ is a smooth projective curve, $f$ is separable,
and $f: Y_b\to X$ is the normalisation of $E$ for some $b\in B$. We may choose $E$ to be a general member of this
family. That is, $b$ is a general point of $B$. We claim that $f^* C$ and $Y_b$ meet transversely. 
Since $f$ is a {\em separable} generically finite map, we have $K_Y = f^* K_X + R = R$ for some effective divisor $R$ on $Y$ and adjunction
gives $R Y_b = K_Y Y_b = 0$.  Note that if there exists a multi-section $M$ of $\pi$ such that $f_* M = 0$, then
$M\subset {\mathop{\mathrm{supp}}\nolimits}(R)$, which contradicts the fact $RY_b = 0$.  Suppose that $Y_b$ and $f^* C$
meet at a point $p$ with multiplicity $\ge 2$. Then $f^* C$ contains a multi-section $P$ of $\pi$ with multiplicity $\ge
2$. Since $f$ is finite at a general point of $P$, $f$ must be ramified along $P$.  Therefore, $P\subset
{\mathop{\mathrm{supp}}\nolimits}(R)$, which again contradicts $RY_b = 0$.
\end{proof}

\begin{Lemma}\label{lem: severi K3}
Let $X$ be a K3 surface over an algebraically closed field of characteristic $0$, and let $D$ be a divisor on $X$.
\begin{enumerate}
    \item Let $C$ and $E$ be integral curves on $X$ of geometric genus $0$ and $1$, respectively, so that $C+E\in |D|$. If
    $g$ is an integer so that $2\le g \le CE$, then $C\cup E \in \overline{V}_{D,g}$.
    \item If $C\in |D|$ is an integral rational curve which has (at least) two branches at a singularity $q\in C$, then
    $C\in \overline{V}_{D,1}$.
\end{enumerate}
\end{Lemma}
\begin{proof}
From Lemma \ref{lem: genus 1 transverse} we can consider now the stable map $f: \widehat{C}\cup \widehat{E}\to X$ of
genus $g$ such that $f_* (\widehat{C}\cup \widehat{E}) = C\cup E$, $\widehat{C}$ and $\widehat{E}$ are the
normalisations of $C$ and $E$, respectively, and $\widehat{C}$ and $\widehat{E}$ meet transversely at $g$ points. Then
$f$ deforms in a family of dimension at least $g$ from Theorem \ref{theo:deffam}. Hence $C\cup E$ can be deformed to an
integral curve of geometric genus $g$, since otherwise $C$ or $E$ would deform too much. That is, $C\cup E\in
\overline{V}_{D,g}$.

For the second statement, let $\nu: \widehat{C} \to X$ be the normalisation of $C$.  Since $C$ has at least two branches
at $q$, $\nu$ factors through a curve $\overline{C}$ as in the diagram
\[
\begin{tikzcd}
\widehat{C} \ar{r} \ar[bend left=40]{rr}{\nu} & \overline{C} \ar{r}{f} & X
\end{tikzcd}
\]
where $\overline{C}$ is an integral curve birational to $C$ with one node $p$ satisfying $f(p) = q$ as the only
singularity. So $f:\overline{C}\to X$ is a stable map of arithmetic genus $1$ and it deforms in dimension at least $1$.
Consequently, $C = f(\overline{C})$ deforms to an integral curve of geometric genus $1$. Namely, $C\in
\overline{V}_{D,1}$.
\end{proof}

\begin{Corollary}\label{cor: gleq1 suffices}
To prove Theorem \ref{thm: curves on complex k3} in characteristic zero, it suffices to prove it for $g=0,1$.
\end{Corollary}
\begin{proof}
Suppose that the theorem holds for $g=0,1$. Let $E_1$ and $E_2$ be two elliptic curves on $X$ such that $\deg_H E_1 \ne
\deg_H E_2$ with respect to an ample divisor $H$. We claim that $E_1 + E_2$ is big and nef.

Since both $E_i$'s are nef, $E_1 + E_2$ is big and nef as long as $(E_1+E_2)^2 > 0$. If one of $E_i$ is big, we are done.
Otherwise, $E_1^2 = E_2^2 = 0$. Since $\deg_H E_1 \ne \deg_H E_2$, $E_1$ and $E_2$ are linearly independent in
$\Pic_\QQ(X)$. Therefore, $E_1 E_2 > 0$ and hence $E_1 + E_2$ is big and nef.

Fixing $g\ge 2$, there are only finitely many rational curves $C\subset X$
meeting an ample divisor $A$ such that $AC<2g$. In particular, since $E_1+E_2$ is big and nef and can be written as a
sum of an ample divisor and an effective one, we get that there are finitely many rational curves $C$ such that
$(E_1+E_2) C < 2g$. Therefore, there exist infinitely many rational curves $C_n$ on $X$ such that $C_n E \ge g$ for
some $E\in \{E_1, E_2\}$. By the above lemma, $C_n\cup E$ deforms to an integral curve of geometric genus $g$ for all
$n$.
\end{proof}

\begin{Remark}\label{rem:char p AC}
Lemma \ref{lem: severi K3} and Corollary \ref{cor: gleq1 suffices} hold more generally over arbitrary algebraically
closed field, assuming $C,E$ satisfy the statement of Lemma \ref{lem: genus 1 transverse} and moreover both deform in
the expected dimensions 0, 1 respectively (the final assumption can be replaced by the assumption that $X$ is not
uniruled). Hence, to reduce Theorem \ref{thm: curves on complex k3} to the case of $g=0,1$ in positive
characteristic, one need only prove that there are infinitely many pairs $(C_n, E_n)$ satisfying these conditions.
\end{Remark}

\subsection{Odd Picard rank via characteristic $p$ reduction}

Next, let us prove Theorem \ref{thm: curves on complex k3} for $g=1$ and $\rho(X)$ odd. As commented before, we
just have to carry out the same characteristic $p$ reduction argument in \cite{bht, liliedtke} for $g=1$. It comes down to
proving the statement of \cite[Proposition 4.2]{liliedtke} for genus 1 curves.

\begin{Proposition}
\label{prop: elliptic curve K3 char p reduction}
Let $(X,H)$ be a polarised K3 surface over a number field $K$, such
that $\Pic(X) = \Pic(X_{\overline{\QQ}})$ has odd rank and $X_{\overline{\QQ}}$ is not elliptic. Then, there is a finite
extension $L/K$ such that for every $N \ge 0$ there exists a set $S_N$ of places of $L$ of
density $1$ such that for all $\mathfrak{q} \in S_N$,
\begin{enumerate}
\item
the reduction $(X_L)_\mathfrak{q}$ is a smooth and non-supersingular K3 surface,
\item
the reduction $H_\mathfrak{q}$ is ample,
\item
there exists an integral geometric genus 1 curve $D_\mathfrak{q}$
on $((X_L)_\mathfrak{q})_{\overline{\mathbf{F}}_p}$ so that
\begin{enumerate}
\item
the class of $D_\mathfrak{q}$ does not lie in $\Pic(X)\otimes_\ZZ \QQ$, where we view $\Pic(X)
\subset\Pic(((X_L)_\mathfrak{q})_{\overline{\mathbf{F}}_p})$ via the specialisation homomorphism, and
\item
$D_\mathfrak{q} . H_\mathfrak{q} \ge N$.
\end{enumerate}
\end{enumerate}
\end{Proposition}

\begin{proof}
The first two statements are standard results (cf.\ \cite{BZ, NO}). By \cite{NO}, we also know that every K3 surface
over $\overline{\FF}_p$ has even Picard rank.  Let $\Lambda$ be the smallest primitive sublattice of
$\Pic(((X_L)_\idealq)_{\overline{\FF}_p})$ containing $\Pic(X)$. Since $\rho(X)$ is odd and
$\rho(((X_L)_\idealq)_{\overline{\FF}_p})$ is even,
\[
\Lambda\subsetneq\Pic(((X_L)_\idealq)_{\overline{\FF}_p}).
\]
Also, $\Lambda$ does not contain isotropic classes since $X_{\overline{\QQ}}$ is not elliptic.  For every $N$, there are
at most finitely many places $\idealq$ such that there is a $(-2)$-curve $R_\mathfrak{q}\subset
((X_L)_\mathfrak{q})_{\overline{\mathbf{F}}_p}$ satisfying $H_\mathfrak{q} . R_\mathfrak{q} < N$ and
$R_\mathfrak{q}\not\in \Lambda$. Otherwise, we could lift $R_\idealq$ to $R\subset X_{\overline{L}}$ so that then
$R\in\Pic(X_{\overline{L}})=\Pic(X)$ and hence $R_\idealq\in \Lambda$ for all but finitely many $\idealq$, which is a
contradiction.

Thus, outside of finitely many places $\idealq$, $H_\idealq$ and $R_\idealq$ satisfy \eqref{eq: elliptic curve
non-primitive class 03} with $A = H_\idealq$ and $R = R_\idealq$ for all $(-2)$-curves $R_\idealq\subset
((X_L)_\idealq)_{\overline{\FF}_p}$ and $R_\idealq\not\in \Lambda$. By Theorem \ref{thm: elliptic curve
non-primitive class}, there exists an integral genus 1 curve $D_\idealq$ on $((X_L)_\idealq)_{\overline{\FF}_p}$ whose
divisor class does not lie in $\Lambda$. If we fix an $L$-rational point $P$ on $X$, we can choose $D_\mathfrak{q}$ such
that $D_\idealq$ passes through $P_\idealq$. Again, by the lifting argument in the proof of \cite[Proposition 4.2]{liliedtke}
(see also the proof of Theorem \ref{thm: ll}), we see that $H_\idealq .
D_\idealq < N$ only for finitely many places $\idealq$.
\end{proof}

\begin{Corollary}\label{cor:odd rank}
  Let $X$ be a K3 surface over an algebraically closed field $k$ in characteristic zero such that $X$ has odd Picard rank $\rho(X)$. Then Theorem
  \ref{thm: curves on complex k3} holds true for $X$.
\end{Corollary}

\begin{proof}
Like in Theorem \ref{thm: ll}, one reduces first to the case that the ground field is $\overline{\QQ}$ by the standard
spreading out argument with Noether-Lefschetz loci as in \cite{bht} and \cite[Theorem 13.3.1]{huybrechts}. By Corollary
\ref{cor: gleq1 suffices}, it suffices to treat the case of geometric genus $g\leq1$.  Next, we may assume that $X$ is
not elliptic since these have already been treated in Theorem \ref{thm: bt tayou}.

The case $g=0$ is the main result of \cite{liliedtke} (see also Proposition \ref{thm: ll} and its proof). The case
$g=1$ follows from Proposition \ref{prop: elliptic curve K3 char p reduction} combined with the Regeneration Theorem
\ref{thm: regeneration} and the fact that genus 1 curves always deform in dimension one on a non-uniruled K3
surface from Proposition \ref{prop:defexpdim}.
\end{proof}

\section{The main theorem: The marked point trick} \label{sec: main theorem 2}

By the Hasse--Minkowski Theorem, projective K3 surfaces with Picard rank $\ge 5$ are elliptic, so from
Theorem \ref{thm: bt tayou} and Corollary \ref{cor:odd rank} the only remaining cases of Theorem \ref{thm: curves on complex k3} are K3 surfaces with
Picard rank $2$ and $4$. From \cite[p661]{nikulin} (an unpublished result of Vinberg), there are only the two rank four
lattices $\Lambda$, namely those of Theorem \ref{thm: curves on complex k3 picard rank 4}, such that a K3 surface $X$
with $\Pic(X) = \Lambda$ is not elliptic and has a finite automorphism group.  

It remains to prove Theorem \ref{thm: curves on complex k3} in the case where the Picard rank is 2 or 4 and $g=0,1$.
One of the main ingredients to do this will be the following, saying that in some situations, we can make
sure that a degeneration of a rational (resp., geometric genus 1) curve on a general K3 surface (e.g., those from Theorem
\ref{thm: appendix}) to a special K3 surface of the same Picard lattice remains integral. Actually, we have already
used a similar argument in the proof of Theorem \ref{thm: elliptic curve non-primitive class} by exploiting the
discriminant locus of a family of stable maps. A more sophisticated version of this argument is given by Propositions
\ref{prop: marked point trick g=0} and \ref{prop: marked point trick g=1}, which use a method that we call the {\em
marked point trick}.

\begin{Theorem}\label{thm: curves on complex k3 picard rank r}
    Let $X$ be a nonsupersingular projective K3 surface over an algebraically closed field $k$ with $\Pic(X) = \Lambda$
    for a lattice $\Lambda$ of rank $r\ge 2$, and let $L\in \Lambda$ be a primitive ample divisor on $X$. Let $M\subseteq
    M_{\Lambda,k}$ be an irreducible component of the moduli space of $\Lambda$-polarised K3 surfaces over $k$
    containing the point representing $X$. Suppose that for a general $Y\in M$, there is an integral rational (resp.,
    geometric genus 1) curve in $|L|$ on $Y$. Then the same holds for $X$, i.e., there is an integral rational (resp.,
    geometric genus 1) curve in $|L|$ on $X$.
\end{Theorem}

\begin{Remark}
Since integrality is an open condition, the above actually implies that if there is an integral rational (resp.,
geometric genus 1) curve in $|L|$ for a primitive ample divisor $L$ on a K3 surface of Picard lattice $\Lambda$, the
same holds for every K3 surface with Picard lattice $\Lambda$.
\end{Remark}

Note in particular that Theorem \ref{thm: curves on complex k3 picard rank r} and Theorems \ref{thm: appendix} and
\ref{thm: curves on complex k3 picard rank 4} combined imply the remaining cases of \ref{thm: curves on complex k3},
since on a K3 surface $X$ of Picard rank at least two, we may choose two $\QQ$-linearly independent primitive (i.e.,
indivisible) ample classes $L_1,L_2 \in \Pic(X)$ and then note that $nL_1+L_2$ is ample and primitive for all $n\geq0$.
In fact we get something stronger in Picard rank 2 and for the two special rank four lattices above.

\begin{Corollary}\label{cor: curves on complex k3 picard rank 2}
Let $X$ be a projective K3 surface over an algebraically closed field of characteristic $0$ and assume that $X$ has
Picard rank $\rho(X)= 2$. Then for every primitive ample divisor $L$ on $X$ satisfying one of A1-A3 in Theorem
\ref{thm: appendix}, there is an integral rational (resp., geometric genus 1) curve in $|L|$.
\end{Corollary}

The above statement can be thought of as a step, in the rank two case, in the direction of \cite[Conjecture
13.0.2]{huybrechts} stating that for any K3, there are infinitely many rational curves in multiples of some fixed
polarisation.

\subsection{The marked point trick}

We now come to the proof of Theorem \ref{thm: curves on complex k3 picard rank r}. The technical heart of it is the
following genus 0 marked point trick.

\begin{Proposition}\label{prop: marked point trick g=0}
Let
\begin{equation}\label{eq: marked point trick g=0 00}
\begin{tikzcd}
\mathscr{C}\ar{d}[left]{f} \ar{r} & V\ar{d}{\varphi} & E\ar[hook']{l}\ar{d}\\
\calX \ar{r}{\pi} & U & \{o\}\ar[hook']{l}
\end{tikzcd}
\end{equation}
be a commutative diagram of quasi-projective varieties over an algebraically closed field,
where
\begin{itemize}
    \item all morphisms are projective,
    \item $\varphi: V\to U$ is a dominant map between two smooth quasi-projective surfaces $V$ and $U$,
    \item $o\in U$ is a point of $U$, $E$ is a connected component of $\varphi^{-1}(o)$, and
    \item $f: \mathscr{C}/V\to \calX$ is a family of stable maps of genus $0$ to $\calX$ over $V$ whose general fibres $\mathscr{C}_b$ are smooth for $b\in V$.
\end{itemize}
Suppose that there is a closed subscheme $Y$ of $\mathscr{C}_E = \mathscr{C}\times_V E$
satisfying that $Y$ is flat over $E$ of pure dimension $\dim E+1$,
\begin{equation}\label{eq: marked point trick g=0 01}
f(Y_s) = f(Y) \text{ for all }s\in E,
\end{equation}
and there are a point $e\in E$ and two distinct irreducible components
$G_1$ and $G_2$ of $Y_e$ with $f_* G_i \ne 0$ for $i=1,2$.
Then there exists an integral curve $F\subset V$ such that $\varphi_* F \ne 0$, $E\cap F \ne \emptyset$,
and $\mathscr{C}_p$ is singular for all $p\in F$.
\end{Proposition}

\begin{proof}
We complete $U, V, \mathscr{C}$ and $\calX$ to projective varieties such that $U$ and $V$ remain smooth and $f:
\mathscr{C}/V\to \calX$ is still a family of stable maps to $\calX$ by the projectivity of the moduli space of stable
maps. Let us choose a general member $D$ of the complete linear series of a sufficiently ample divisor on $\calX$ such
that
\begin{itemize}
    \item $\dim D\cap f(\mathscr{C}_s) = 0$ for all $s\in V$,
    \item $f^* D$ and $\mathscr{C}_e$ meet transversely,
    \item $f(G_i) . D \ge 4$ for $i=1,2$,
and
    \item $f^* D$ and $\mathscr{C}_b$ meet transversely for every point $b\in E$ that lies on two distinct irreducible components of $E$.
\end{itemize}
Note that $E$ is either a point or a connected curve.
Hence there are at most finitely many points $b$ lying on two distinct components of $E$ and we can always choose $D$ general enough so that $f^* D$ and $\mathscr{C}_b$ meet transversely for all such $b$.

Since $D$ meets $f(\mathscr{C}_s)$ properly for all $s\in V$, every irreducible component of $f^*D$ dominates $V$.
After a generically finite base change of $V$, we may assume
\[
f^*D \,=\, P_1 + P_2 + \ldots + P_N,
\]
where $P_1,P_2,\ldots,P_N$ are $N=(f_* \mathscr{C}_s).D$ distinct rational sections of $\mathscr{C}/V$.
Since $f^* D$ and $\mathscr{C}_e$ meet transversely, $P_i\cap \mathscr{C}_e$ are $N$ distinct points on $\mathscr{C}_e$.
Also, as $f^* D$ and $\mathscr{C}_b$ meet transversely for every point $b\in E$ that lies on two distinct irreducible
components of $E$, we have that
$P_i \cap Y_s \ne \emptyset$ for all $s\in E$ if and only if $P_i\cap Y_e \ne \emptyset$. We will now exhibit four
sections which satisfy this equivalence.
By our choice of $D$, we can find four $P_i$'s, say
$P_1, P_2, P_3, P_4$, such that
\[
\begin{aligned}
P_1\cap G_1 &\ne\emptyset,\ P_2\cap G_1 \ne\emptyset,\ P_3\cap G_2 \ne\emptyset,\ P_4\cap G_2 \ne \emptyset
\text{ and }
\\
f(P_i\cap \mathscr{C}_e) &\ne f(P_j\cap \mathscr{C}_e) \text{ for all } 1\le i < j\le 4.
\end{aligned}
\]
As pointed out above, we have
$P_i\cap Y_s \ne\emptyset$
for all $1\le i\le 4$ and $s\in E$.
Since $D$ meets $f(\mathscr{C}_s)$ properly, $f(P_i\cap \mathscr{C}_s)$ is a point for every $s\in V$. Hence
\[
f(P_i\cap \mathscr{C}_s) = f(P_i\cap Y_s)
\]
for all $1\le i\le 4$ and $s\in E$. On the other hand,
by \eqref{eq: marked point trick g=0 01}, $f(Y_s) = f(Y)$ is a fixed curve for all $s\in E$. Therefore, $D\cap f(Y)$ is a set of finitely many points and hence
\[
f(P_i\cap \mathscr{C}_s) \,=\, f(P_i\cap Y_s) \,\subset\, D\cap f(Y)
\]
must be constant as $s$ moves in $E$ for $1\le i\le 4$ since $E$ is connected. Thus,
\[
f(P_i\cap \mathscr{C}_s) = q_i = f(P_i\cap \mathscr{C}_e)
\]
for all $1\le i\le 4$ and $s\in E$, where $q_i$ are four distinct points.
From Proposition \ref{prop:ss reduction}, after replacing $V$ by a generically finite base change and $E$ by a connected
component of the inverse image of $E$ here, we get a family
\begin{equation}\label{eq: marked point trick g=0 02}
\begin{tikzcd}
\widehat{\mathscr{C}}\ar{r}{\alpha} \ar{rd}[below left]{\widehat{f}} & \mathscr{C}\ar{d}[left]{f} \ar{r} & V\ar{d}{\varphi} & E\ar[hook']{l}\ar{d}\\
& \calX \ar{r}{\pi} & U & \{0\}\ar[hook']{l} 
\end{tikzcd}
\end{equation}
where $\widehat{f}: (\widehat{\mathscr{C}}, \widehat{P}_1,\widehat{P}_2,\widehat{P}_3,\widehat{P}_4)
\to \calX$
is a family of stable maps of genus $0$ with $4$ marked points over $V$, with $\widehat{P}_i$ being the inverse images of $P_i$ under $\alpha$.

We use the notation $\overline{\mathcal M}_{g,n}$, ${\mathcal M}_{g,n}$ and $\partial \overline{\mathcal M}_{g,n}$
to denote the moduli space of $n$-pointed stable curves of genus $g$, its open part parametrising smooth curves of genus
$g$ with $n$ distinct marked points and the boundary divisor $\overline{\mathcal M}_{g,n} \backslash {\mathcal
M}_{g,n}$, respectively.  There is a natural morphism (namely the \textit{stabilisation morphism}) $\sigma: V\to
\overline{\mathcal M}_{0,4}$ sending
\[
\sigma(s) = (\widehat{\mathscr{C}}_s, \widehat{\mathscr{C}}_s\cap \widehat{P}_1, \widehat{\mathscr{C}}_s\cap\widehat{P}_2,\widehat{\mathscr{C}}_s\cap\widehat{P}_3,\widehat{\mathscr{C}}_s\cap\widehat{P}_4),
\]
where it is understood that some components of $\widehat{\mathscr{C}}_s$
are contracted.

Let us consider the pullback $\sigma^*(\partial \overline{\mathcal M}_{0,4})$.
Since $\partial \overline{\mathcal M}_{0,4}$ consists of three points in
$\overline{\mathcal M}_{0,4} \cong \PP^1$, it is an ample divisor.
Therefore, $\sigma^*(\partial \overline{\mathcal M}_{0,4})$ is nef on $V$. Also we have $e\in \sigma^{-1}(\partial \overline{\mathcal M}_{0,4})$ by our choice of $P_i$. In summary, we have
\[
\sigma^*(\partial \overline{\mathcal M}_{0,4}) \text{ is nef \quad and\quad }
E\cap \sigma^{-1}(\partial \overline{\mathcal M}_{0,4}) \ne \emptyset.
\]
We claim that there is an irreducible component $F$ of $\sigma^{-1}(\partial \overline{\mathcal M}_{0,4})$ such that
\begin{equation}
\label{eq: marked point trick g=0 03}
E\cap F\ne\emptyset \text{ \quad and\quad } F\not\subset E.
\end{equation}
Otherwise, for every irreducible component $F$ of $\sigma^{-1}(\partial \overline{\mathcal M}_{0,4})$, either $F\cap E =
\emptyset$ or $F\subset E$. Since $E\cap \sigma^{-1}(\partial \overline{\mathcal M}_{0,4}) \ne \emptyset$, there is
a connected component $\Sigma$ of $\sigma^{-1}(\partial \overline{\mathcal M}_{0,4})$ such that $\Sigma\subset E$. Since
$\varphi_* E = 0$, $E$ is supported on a union of integral curves with negative definite self-intersection matrix; the
same holds for $\Sigma$. This contradicts the fact that $\sigma^*(\partial \overline{\mathcal M}_{0,4})$ is nef.

Therefore, there is an integral curve $F\subset \sigma^{-1}(\partial \overline{\mathcal M}_{0,4})$ satisfying \eqref{eq:
marked point trick g=0 03}. Clearly, $\varphi_* F \ne 0$. Otherwise, $\varphi(F) = o$ since $E\cap F\ne\emptyset$; then
this implies that $F\subset E$ since $E$ is a connected component of $\varphi^{-1}(o)$, which is a contradiction.
Therefore, $\varphi_* F \ne 0$.

For a general point $p\in F$, if $\widehat{\mathscr{C}}_p$ contains (at least) two irreducible components $\Gamma_1\ne
\Gamma_2$ such that $\widehat{f}_* \Gamma_i\ne 0$ for $i=1,2$, then the same must hold for $\mathscr{C}_p$; hence
$\mathscr{C}_p$ must be singular and we are done.

Assume for a contradiction that this is not the case. Then there is a unique component $\Gamma$ of
$\widehat{\mathscr{C}}_p$ such that $\widehat{f}_* \Gamma\ne 0$. Since $p\in \sigma^{-1}(\partial
\overline{\mathcal M}_{0,4})$, there exists a connected union $J$ of components of $\widehat{\mathscr{C}}_p$ such that
$\Gamma$ and $J$ meet at a single point and $J$ meets two out of the four sections
$\widehat{P}_1,\widehat{P}_2,\widehat{P}_3,\widehat{P}_4$, say
\[
J \cap \widehat{P}_a \ne \emptyset \text{ and } J\cap \widehat{P}_b \ne\emptyset
\]
for some $1\le a < b\le 4$. Since $\Gamma$ and $J$ meet at a single point and $\Gamma$ is the only component of
$\widehat{\mathscr{C}}_p$ not contractible under $\widehat{f}$, we see that $\widehat{f}$ contracts $J$ to a point.
Therefore,
\[
f(\mathscr{C}_p \cap P_a) = \widehat{f}(\widehat{\mathscr{C}}_p\cap \widehat{P}_a)
= \widehat{f}(\widehat{\mathscr{C}}_p\cap \widehat{P}_b)
= f(\mathscr{C}_p\cap P_b) = \widehat{f}(J)
\]
for a general point $p\in F$ and hence also for all $p\in F$. But $E\cap F \ne\emptyset$ and $f(P_i\cap \mathscr{C}_s) =
q_i$ are four distinct points for all $s\in E$, which is a contradiction.
\end{proof}

We have a similar statement for stable maps of genus $1$.

\begin{Proposition}\label{prop: marked point trick g=1}
Consider the commutative diagram \eqref{eq: marked point trick g=0 00} of quasi-projective varieties over an
algebraically closed field, where
\begin{itemize}
    \item all morphisms are projective,
    \item $\varphi: V\to U$ is a dominant map between two smooth quasi-projective surfaces $V$ and $U$,
    \item $o\in U$ is a point of $U$, $E$ is a connected component of $\varphi^{-1}(o)$, and
    \item $f: \mathscr{C}/V\to \calX$ is a flat family of stable maps of genus $1$ to $\calX$ over $V$ whose general
    fibres $\mathscr{C}_b$ are smooth for $b\in V$.
\end{itemize}
Suppose that there is a closed subscheme $Y$ of $\mathscr{C}_E = \mathscr{C}\times_V E$
such that $Y$ is flat over $E$ of pure dimension $\dim E+1$ satisfying \eqref{eq: marked point trick g=0 01}
and there are a point $e\in E$ and an irreducible component
$G$ of $Y_e$ which has geometric genus $0$ and satisfies $f_* G \ne 0$.
Then there exists an integral curve $F\subset V$ such that $\varphi_* F \ne 0$, $E\cap F \ne \emptyset$,
and $\mathscr{C}_p$ is singular for all $p\in F$.
\end{Proposition}

\begin{proof}
The proof is very similar to that of Proposition \ref{prop: marked point trick g=0}, which is why we only sketch it.
Again, let us first complete $U, V, \mathscr{C}$ and $\calX$ to projective varieties.
We choose a general member $D$ of the complete linear series of a sufficiently ample divisor on $\calX$ such that
\begin{itemize}
    \item $\dim D\cap f(\mathscr{C}_s) = 0$ for all $s\in V$,
    \item $f^* D$ and $\mathscr{C}_e$ meet transversely,
    \item $f(G) . D \ge 2$, and
    \item $f^* D$ and $\mathscr{C}_b$ meet transversely for every point $b\in E$ that lies on two distinct irreducible components of $E$.
\end{itemize}
Again, after a generically finite base change of $V$, we may assume
\[
f^*D \,=\, P_1 + P_2 + \ldots + P_N
\]
where $P_1,P_2,\ldots,P_N$ are $N=(f_* \mathscr{C}_s).D$ distinct rational sections of $\mathscr{C}/V$.
We can find two $P_i$'s, say $P_1, P_2$, such that
\[
 P_1\cap G \,\ne\,\emptyset,\quad
 P_2\cap G \,\ne\,\emptyset,\text{ \quad and\quad }
 f(P_1\cap \mathscr{C}_e) \,\ne\, f(P_2\cap \mathscr{C}_e).
\]
As before, we can prove
$
f(P_i\cap \mathscr{C}_s) = q_i = f(P_i\cap \mathscr{C}_e)
$
for all $1\le i\le 2$ and $s\in E$, where $q_1$ and $q_2$ are two distinct points.

After a generically finite base change of $V$, we can find a morphism $\alpha: \widehat{\mathscr{C}}\to \mathscr{C}$ with the diagram \eqref{eq: marked point trick g=0 02}
such that $\widehat{f}: (\widehat{\mathscr{C}}, \widehat{P}_1,\widehat{P}_2)
\to \calX$
is a family of stable maps of genus $1$ with $2$ marked points over $V$ with $\widehat{P}_i$ being the proper transforms of $P_i$ under $\alpha$.

Again, we consider the natural morphism $\sigma: V\to \overline{\mathcal M}_{1,2}$ sending
\[
\sigma(s) = (\widehat{\mathscr{C}}_s, \widehat{\mathscr{C}}_s\cap \widehat{P}_1, \widehat{\mathscr{C}}_s\cap\widehat{P}_2).
\]
We know that ${\mathcal M}_{1,2}$ is affine. So there is an ample effective divisor $A$ of $\overline{\mathcal M}_{1,2}$ supported on $\partial \overline{\mathcal M}_{1,2}$.
By considering the pullback $\sigma^*(A)$, we can again show that
there is an irreducible component $F$ of $\sigma^{-1}(\partial \overline{\mathcal M}_{1,2})$ satisfying
\eqref{eq: marked point trick g=0 03}.
The rest of the proof is identical to that of Proposition \ref{prop: marked point trick g=0}.
\end{proof}

\subsection{Proof of Theorem \ref{thm: curves on complex k3 picard rank r}}
\label{subsec: proof of mpt thm}

Equipped with the propositions of the previous subsection, we prove first the case where $\rho(X) = 2$ so as to
elucidate the method.
\medskip

\begin{specialproof}[Proof of Theorem \ref{thm: curves on complex k3 picard rank r} when $\rho(X) = 2$]
We consider $(X,L)\in M_{2d}$ as a point in (a smooth atlas of) the moduli space of K3 surfaces polarised by a primitive
ample line bundle where $L^2=2d$. As $X$ is not supersingular, from Proposition \ref{prop: lattice polarised stack
smooth}, we may take $U\subset M$ a general smooth, irreducible, affine surface passing through $(X,L)$ in this moduli
space (more correctly, in this smooth atlas). Since $M_{\Lambda,\ZZ}$ is smooth around $X$, $U$ can be chosen so that there
is precisely one irreducible curve $B\subset U$ passing through $X$ and parametrising K3 surfaces whose Picard lattices
contain $\Lambda$, i.e., $U\cap M=B$. To summarise, if $\calX\to U$ is the corresponding family of K3 surfaces, then we have
that $\Pic(\calX_s) = \ZZ \calL_s$ for $s\in U$ very general and $\calL\in \Pic(\calX)$, $\Pic(\calX_b) = \Lambda$ for
$b\in B$ very general, $\calX_0 = X$ and $\calL_0 = L$ for a point $0\in B$.

From the assumptions there exists an integral rational (resp., geometric genus 1) curve in $|\calL_b|$ on $\calX_b$ for
$b\in B$ general. ``Spreading'' these curves over $U$ by Corollary \ref{cor:def}, taking a closure containing $0$ and
finally applying stable reduction from Proposition \ref{prop:ss reduction}, we arrive at a diagram like \eqref{eq:
marked point trick g=0 00}:
\[
\begin{tikzcd}
\mathscr{C}\ar{d}[left]{f} \ar{r} & V\ar{d}{\varphi} & & E\ar[hook']{ll}\ar{d}\\
\calX \ar{r}{\pi} & U & B\ar[hook']{l} & \{0\}\ar[hook']{l}
\end{tikzcd}
\]
where
\begin{itemize}
    \item all morphisms are projective,
    \item $V$ is a smooth quasi-projective surface,
    surjective and generically finite over $U$ via $\varphi$,
    \item $E$ is a connected component of $\varphi^{-1}(0)$,
    \item $f: \mathscr{C}/V\to \calX$ is a flat family of stable maps of genus $g=0$ or 1 to $\calX$ over $V$,
    \item $\mathscr{C}_s$ is smooth for $\varphi(s)\in B$ general, and
    \item if $g=1$, we require $f(\mathscr{C}_s)$ to pass through a fixed general point $q\in X$ for all $s\in E$.
\end{itemize}
Clearly, the image $f(\mathscr{C}_s)$ is rigid for $s\in E$;
otherwise we obtain a contradiction to Proposition \ref{prop:defexpdim}.
That is,
\[
f(\mathscr{C}_s) = f(\mathscr{C}_E) \text{ for all } s\in E
\]
where $\mathscr{C}_E = \mathscr{C} \times_V E$.
We claim that $\mathscr{C}_s$ is smooth for all $s\in E$. Otherwise, suppose that $\mathscr{C}_e$ is singular for some $e\in E$.

When $g = 0$, from stability of the model, it is obvious that there exist two irreducible components $G_1\ne G_2$ of
$\mathscr{C}_e$ with $f_* G_i \ne 0$ for $i=1,2$.  When $g=1$, $\mathscr{C}_e$ obviously contains a component $G$ such
that $G$ is rational and $f_* G\ne 0$. Consequently, by Proposition \ref{prop: marked point trick g=0} or \ref{prop:
marked point trick g=1} with the same notation as above and $Y=\mathscr{C}_E$, we get that there exists an
integral curve $F\subset V$ such that $\varphi_* F\ne 0$, $E\cap F \ne \emptyset$ and $\mathscr{C}_p$ is singular for
all $p\in F$.

We claim that $\mathscr{C}_p$ is reducible for all $p\in F$. This is obvious when $g = 0$. When $g=1$,
$f(\mathscr{C}_s)$ passes through a general point $q\in X$ for all $s\in E$ so $f(\mathscr{C}_s)$ contains a genus 1
curve. Therefore, $\mathscr{C}_s$ has a component of geometric genus $1$ not contracted by $f$ for all $s\in E$. 
Since $E\cap F \ne \emptyset$, the same holds for $\mathscr{C}_p$ over a general point $p\in F$. It follows that
$\mathscr{C}_p$ is reducible for all $p\in F$.

Consequently, we may write $\mathscr{C}_p = M_p + N_p$ for $p\in F$, where $M_p$ and $N_p$ are unions of components of
$\mathscr{C}_p$ with $f_* M_p \ne 0$ and $f_* N_p \ne 0$. When $p\in E\cap F$, $f_* M_p + f_* N_p \in |L|$ on $X$.
Since $L$ is primitive, $f_* M_p$ and $f_* N_p$ must generate $\Lambda$ over $\QQ$. Therefore,
$\Pic(\calX_{\varphi(p)})$ contains $\Lambda$ as a sublattice for all $p\in F$. This implies that $B = \varphi(F)$. But
$\mathscr{C}_p$ is smooth for $\varphi(p)\in B$ general, which is a contradiction.

Therefore, $\mathscr{C}_s$ is smooth for all $s\in E$. Then $f_* \mathscr{C}_s\in |L|$ is supported on an integral curve
on $X$ for $s\in E$. Since $L$ is primitive, it must be reduced and hence an integral curve of genus $g$ in $|L|$.
\end{specialproof}

The general case is now similar, and we identify the parts that need to be altered.\medskip

\begin{specialproof}[Proof of Theorem \ref{thm: curves on complex k3 picard rank r}]
Let $B\subset M_\Lambda$ be a general smooth affine curve passing through the point $0$ representing $X$ (again, in some
suitable smooth atlas). By our hypotheses, there is an integral rational (resp., geometric genus 1) curve in $|L|$ on a
general $Y\in B$. Let $C\subset X$ be a flat limit of these curves. For genus 1 curves, we may choose $C$ to pass
through a general point on $X$.
Suppose that
\[
C = \sum_{i=1}^n \mu_i \Gamma_i
\]
where $\mu_i\in \ZZ^+$ and $\Gamma_i$ are irreducible components of $C$ for $i=1,2,\ldots,n$. We can always find a primitive sublattice
$\Sigma$ of $\Lambda$ satisfying that
\begin{itemize}
\item $\rank_\ZZ \Sigma = r - 1$,
\item $L\in \Sigma$, and
\item $D\not\in \Sigma$ for all $0 < D < C$, i.e.,
\[
\Sigma \cap \Big\{
\sum_{i=1}^n a_i \Gamma_i: 0\le a_i\le \mu_i,\ a_i\in \ZZ,\
0 < \sum_{i=1}^n a_i < \sum_{i=1}^n \mu_i
\Big\} = \emptyset.
\]
\end{itemize}
Such $\Sigma$ always exists since the primitive sublattices $\Sigma\subset \Lambda$ of rank $r-1$ and containing $L$ are
parametrised by the Grassmannian ${\mathop\text{Gr}}_\QQ(r-2,r-1)\cong \PP_\QQ^{r-2}$. Since $L$ is primitive, $L$
and $D$ are linearly independent over $\QQ$ for all $0 < D < C$. Therefore, the locus of $D\in \Sigma$ has codimension
one in $\PP_\QQ^{r-2}$ for each $0<D<C$. Since there are only finitely many such $D$'s, we can always find $\Sigma \in
\PP_\QQ^{r-2}$ that does not contain any $0<D<C$.

Let $M_\Sigma$ be the component of the moduli space of $\Sigma$-polarised K3 surfaces containing $M_\Lambda$. As before,
we choose a general smooth affine surface $U\subset M_\Sigma$ containing $B$ and let $\calX\to U$ be the corresponding
family of K3 surfaces such that $\Pic(\calX_s) = \Sigma$ for $s\in U$ very general, $\Pic(\calX_b) = \Lambda$ for $b\in
B$ very general, $\calX_0 = X$ and $\calL_0 = L$ for $\calL\in \Pic(\calX)$.

Again, by spreading integral rational (resp., geometric genus 1) curves in $|\calL|$ on $\calX_b$ for $b\in B$ over $U$
and applying stable reduction, we arrive at a diagram like (\ref{eq: marked point trick g=0 00}). The rest of the proof is
the same as above.
\end{specialproof}

\section{Positive characteristic}\label{sec: pos char}

Theorems \ref{thm: bt tayou}, \ref{thm: curves on complex k3 picard rank r}, Corollary \ref{cor: reduction} and
Proposition \ref{prop: supersingular OK} imply that in order to obtain Theorem A in the cases $g=0,1$ (and hence Conjecture \ref{conj: inf
many rational curves}) in positive characteristic, the cases that remain are the following, all over
$\overline{\FF}_p$.
\begin{enumerate}
    \item $X$ is a general point of $M_{\Lambda,\overline{\FF}_p}$ where $\Pic(X)=\Lambda$ is of rank 2.
    \item $X$ is a general point of $M_{\Lambda,\overline{\FF}_p}$ for $\Lambda$ one of the lattices in Theorem \ref{thm:
    curves on complex k3 picard rank 4}.
    \item $X$ is isotrivially elliptic, $p\leq3$ and $g=0$.
    \item $X$ is uniruled, $p\leq3$ and $g=1$.
\end{enumerate}
In this section we will prove that there is a mixed characteristic version of Theorem \ref{thm: curves on complex k3
picard rank r} which produces integral rational curves in primitive classes on any K3 surface of lattice from (1) or (2)
from the list above, by specialising them from characteristic zero.

\begin{Remark}
Note that to generalise the higher genus claims of Theorem \ref{thm: curves on complex k3} to arbitrary characteristic,
one would need a version of Lemma \ref{lem: severi K3}, which ultimately relies on extending the lemma of
Arbarello--Cornalba (cf.\ Lemma \ref{prop:defexpdim}); see Remarks \ref{rem:char p AC} and \ref{rem: char p AC}.
\end{Remark}

We begin with a version of stable reduction over an arbitrary base scheme which will be used in the proof. 

\begin{Proposition}\label{prop:ss reduction}
    Let $\calX\to U$ be a smooth proper family of varieties over an integral scheme, let $\pi:\calC\to U$ be a flat proper
    morphism whose generic fibre is a geometrically integral curve of geometric genus $g\geq0$, and let $f:\calC\to\calX$ be a
    $U$-morphism which does not contract the generic fibre. Then there is a commutative diagram
    \[
    \begin{tikzcd}
        \calC'\ar{d}[left]{\pi'}\ar[bend left]{rr}{f'}\ar{r} & \calC\ar{d}[left]{\pi}\ar{r}[below]{f} & \calX\ar{dl} \\
        V\ar{r}{\gamma} & U & 
    \end{tikzcd}
    \]
    where $V$ is an integral scheme, $\gamma$ is a proper surjective generically finite morphism, and $\pi'$ is a stable
    family of genus $g$ curves with smooth generic fibre.
\end{Proposition}
\begin{proof}
    The proof follows that of \cite[Th\'eor\`eme 1.5]{benoistss} and we thank its author for referring us to it. Let
    $\eta$ be the generic point of $U$. Let $\widehat{\calC}_{\eta}$ be the normalisation of $\calC_\eta$ which, after passing
    to a finite extension $k(\eta_1)/k(\eta)$, we may assume is a smooth projective curve of genus $g$ over $k(\eta_1)$.
    If $g\geq2$, consider the induced moduli map
    \[
    \eta_1\to \mathbf{M}:=\overline{\mathbf{M}}_g(\calX/U, \beta)
    \]
    where $\beta$ the cohomology class of the images of $f$ and $\mathbf{M}$ the Kontsevich moduli stack (see, e.g.,
    \cite[Theorem 50]{ak}, \cite[Theorem 13.3.5]{olsson} for this level of generality) which is Deligne--Mumford.
    If $g=0$, consider four general closed points $P_1,\ldots,P_4\in\widehat{\calC}_{\eta_1}(k(\eta_1))$ and take
    instead the induced morphism $\eta_1\to \mathbf{M}:=\overline{\mathbf{M}}_{0,4}(\calX/U,P_1,\ldots,P_4,\beta)$ (where $P_i$,
    with slight abuse of notation, are again the images of $P_i$ in $\calX_{\eta}$), and analogously if $g=1$, take one point
    $P_1$ and $\mathbf{M}:=\overline{\mathbf{M}}_{1,1}(\calX/U, P_1, \beta)$.

    By Chow's Lemma for Deligne--Mumford stacks, there is a finite surjective morphism $Z\to \mathbf{M}$ from a scheme
    $Z$ and a semistable family $C\to Z$ of genus $g$ (resp., marked) curves over it. Now let $\eta'$ be an irreducible component of
    $\eta_1\times_{\mathbf{M}} Z$, and take $\widetilde{U}$ the normalisation of $U$ in $\eta'$ and $V\to\widetilde{U}$
    the proper birational morphism resolving the indeterminacies of $\widetilde{U}\dashrightarrow Z$. Pulling back $C$
    to $V$ gives the required family $\calC'\to V$.
\end{proof}

\begin{Remark}
    Although $\calX$ appears auxilliary in the above statement, its existence is used in the proof.
\end{Remark}

The goal of the rest of this section is to prove the following.

\begin{Theorem}
    Let $X$ be a non-supersingular smooth projective K3 surface with Picard rank $\rho\geq2$ over an algebraically closed
    field $k$ with $\chara(k)=p>0$, and let $L$ be a primitive ample class in $\Lambda=\Pic(X)$. Let $M$ be an
    irreducible component of $M_{\Lambda,\ZZ}$ containing $X$, and assume that the general point $Y\in M$ has an
    integral rational (resp., geometric genus 1) curve in $|L|$. Then there is an integral rational (resp., geometric
    genus 1) curve in $|L|$ on $X$.
\end{Theorem}
\begin{proof}
We give a proof only in the case of geometric genus 0 and where the rank of $\Lambda$ is $\rho=2$, and refer to
Subsection \ref{subsec: proof of mpt thm} for the changes that need to be made for the higher rank case and for $g=1$.

Let $L^2=2d$, $W=\Spec W(k)$ the ring of Witt vectors of $k$ with generic and special points $\eta, p$ respectively, and
denote by $M_{2d, W}$ (a smooth atlas of) the moduli space of $2d$-polarised K3 surfaces over $W$ as in Section
\ref{subsec:mod of k3s}. Consider now a diagram
\[
\begin{tikzcd}
    \cal{X}\ar{r} & U\ar[hook]{r}\ar{dr} & M_{2d,W}\ar{d} \\
    & & W 
\end{tikzcd}
\]
where $U$ is a quasi-projective subvariety of $M_{2d, W}$ containing $X$, smooth and of relative dimension 1 over $W$
(with special and generic fibre $U_p$ and $U_\eta$ respectively) and $\calX\to U$ is the corresponding family of K3
surfaces, in such a way that:
\begin{enumerate}
    \item $U_p$ is a general curve in $M_{2d,k}$ containing the point $X$, i.e., its generic\footnote{If
    $k=\overline{\FF}_p$, then every closed point necessarily has even Picard rank.} point corresponds to a K3 surface
    with geometric Picard group equal to $\ZZ L$,

    \item $U_{\overline\eta}$ is a general curve in $M_{2d, \overline{\eta}}$ which contains a lift $X_{\overline\eta}$
    of $X$ such that $X_{\overline\eta}$ is a general K3 with $\Pic X_{\overline\eta}=\Lambda$ (whose existence is
    guaranteed by \cite{lieblichmaulik}), and as above the very general point of $U_{\overline\eta}$ corresponds to a
    characteristic zero K3 of Picard group $\ZZ L$,
    
    \item there is a DVR $B$, flat over $W$, embedded in $U$ with image containing $X_{\overline\eta}, X$, in such a way
    that $M_{\Lambda, W}\cap U = B$, i.e., $B$ is the $\Lambda$-Noether--Lefschetz locus in $U$.
\end{enumerate}

From the assumption there is an integral rational curve $R\in |L|$ on $X_{\overline\eta}$. We spread the normalisation
morphism $\widehat{R}\to R$ out to a family of genus zero curves $\mathcal{T}\to U$, whose generic member is integral.
By Nagata compactification \cite[\href{https://stacks.math.columbia.edu/tag/0F41}{Tag 0F41}]{stacks}, we may assume that
$U$ and $\mathcal{T}$ are proper over $W$. From Proposition \ref{prop:ss reduction}, there is a generically finite and
proper morphism $f:V\to U$ and a generically smooth family 
\[
\begin{tikzcd}
    \calC\ar{d}\ar{r} & \calX\ar{d} \\
    V\ar{r}{f} & U
\end{tikzcd}
\]
of stable curves of genus 0 mapping to $\calX$ with images of class $L$. By the usual flatness criterion over one
dimensional bases, $V$ is flat and proper over $W$. If $0\in U_p$ is the point corresponding to $X$, we let $E\subset V$
be an irreducible component of $f^{-1}(0)$.

If the fibre of $\calC\to V$ over some point of $E$ is smooth, then we are done as its image will be an integral curve
in $|L|$ on $X$. If not, then it can be written as a sum of at least two components like in the proof of Theorem
\ref{thm: curves on complex k3 picard rank r}. We proceed now like in the beginning of the proof of Proposition
\ref{prop: marked point trick g=0}. After base changing $V$, we may assume that we have four rational sections $P_1,\ldots,
P_4$ of $\calC\to V$, two on each of two adjacent irreducible components of $\calC\times_V E$, neither of which gets
contracted to $\calX$, exactly like in the original proof, so that these sections meet general points on the generic
fibre of $\calC$. We get an induced moduli (stabilisation) $W$-morphism
\[
\begin{tikzcd}
V\ar{rr}{\phi}\ar{dr} & & \overline{\mathcal{M}}_{0,4,W}\cong\PP^1_W\ar{dl} \\
& W &
\end{tikzcd}
\]
Note that the fibre $\calC_p\to V_p$ over $p\in W$ is a generically smooth morphism since the general K3 parametrised by
$U_p$ has Picard rank one generated by the primitive class $L$, hence the corresponding fibre of $\calC$ is an irreducible
stable genus 0 curve. Over the point $X$, there is at least one fibre of $\calC$ which is reducible, so $\phi(V_p)$ meets the
boundary of $\overline{\mathcal{M}}_{0,4,k}$, implying that the induced morphism $V_p\to \overline{\mathcal{M}}_{0,4,k}$
is surjective as $V_p$ is proper. Similarly, since the sections $P_i$ were chosen general, $\phi$ cannot contract the
generic fibre $V_{\overline\eta}$, so $\phi$ is surjective as $V$ is proper over $W$. In particular, we may pull back the relatively
ample $W$-divisor $\partial\overline{\mathcal{M}}_{0,4, W}$ to a divisor $F\subset V$, flat over $W$.

As $f$ is surjective, the image of $F$ in $U$ must again meet both the generic and the central fibres $U_{\overline\eta}, U_p$
respectively. Note though that $f(F)_{\overline\eta}$ must be the point $X_{\overline\eta}$, since its points
correspond to K3s whose corresponding curves in $\calC$ are reducible yet of primitive class $L$, hence with components rationally spanning $\Pic
X_{\overline\eta}\otimes\QQ$. In other words, $f(F)_{\overline\eta}\subset B_{\overline\eta}$ is in the
$\Lambda$-Noether--Lefschetz locus of $U_{\overline\eta}$, yet $X_{\overline\eta}$ is the only such point. This
contradicts the irreducibility of $R$ though, proving that the specialisation of $R$ to $X$ must also be irreducible.
\end{proof}

The above, along with Theorems \ref{thm: appendix} and \ref{thm: curves on complex k3 picard rank 4}, now implies the following.

\begin{Corollary}\label{cor:char p cor}
Let $X$ be a projective K3 surface over an algebraically closed field $k$ of characteristic $p\geq0$. Then $X$ contains
infinitely many rational curves, with the only possible exception if $k=\overline{\FF}_p$ with $p=2,3$ and $X$ is
isotrivially elliptic.
\end{Corollary}

The analogous statement for genus 1 curves holds, with the only possible exception uniruled K3 surfaces in
characteristic $p\leq3$. Needless to say, we do not expect avoiding these cases is necessary.

\appendix

\numberwithin{equation}{section}

\section{Rational Curves on Generic K3s with Picard Rank $2$}\label{appendix}

In this appendix, we will give an alternative proof of the existence of rational curves on generic K3 surfaces with
Picard rank $2$ \cite{generic}.

\begin{Theorem} \label{thm: appendix}
Let $\Lambda$ be a K3 lattice of rank two, let $M_{\Lambda,\CC}$ be the moduli space of
$\Lambda$-polarised complex K3 surfaces, and let $L\in \Lambda$ be a big and nef line bundle on a general K3 surface
$X\in M_{\Lambda,\CC}$. Then there exists an open and dense subset $U$ of $M_{\Lambda,\CC}$ (with respect to the Zariski
topology), depending on $L$, such that on every K3 surface $X\in U$, the complete linear series $|L|$ contains an
integral rational curve with unramified normalisation if one of the following holds:
\begin{enumerate}
\item[A1.] $\det(\Lambda)$ is even;
\item[A2.] $L = L_1 + L_2 + L_3$ for some $L_i\in \Lambda$ satisfying that
$L L_i > 0$ and $L_i^2 > 0$ for $i=1,2,3$;
\item[A3.] $L = L_1 + L_2$ for some $L_i\in \Lambda$ satisfying that
$L L_i > 0$ for $i=1,2$, $L_1^2 > 0$, $L_2^2 = -2$ and
$L_1 - L_2\not\in n\Lambda$ for all $n\in\ZZ$ and $n\ge 2$.
\end{enumerate}
\end{Theorem}

Compared with Theorem \ref{thm: nodal curves}, instead of the existence of nodal rational curves, we obtain rational
curves $C$ on $X$ whose normalisation $\nu: C^\nu \to X$ is immersive (i.e.,\ unramified). This is weaker but sufficient
for the purposes of this paper. On the other hand, condition A3 is slightly simpler than that in Theorem \ref{thm: nodal curves}.

The first step of our proof of Theorem \ref{thm: appendix} is the same as that of Theorem \ref{thm: nodal curves} in
\cite{generic}: we specialize a general K3 of Picard rank $2$ to a K3 of higher Picard rank. More precisely we recall
the following from \cite[Lemma 3.3, 3.6, 3.8]{generic}.

\begin{Proposition}\label{prop: rank two embedding}
Let $\Lambda$ be a K3 lattice of rank $2$, i.e., an even lattice of rank $2$ with signature $(1,1)$. Then there exists a
primitive embedding $\sigma: \Lambda \hookrightarrow \Sigma$, where $\Sigma$ is a lattice with intersection matrix
\begin{equation}
\label{K3RATNOTESE007}
\begin{bmatrix}
2\\
& -2\\
&& \ddots\\
&&& -2
\end{bmatrix}_{7\times 7}
\end{equation}
if $\det(\Lambda)$ is even and
\begin{equation}
\label{K3RATNOTESE019}
\begin{bmatrix}
0 & 1\\
1 & -2\\
&& -2\\
&&& \ddots\\
&&&& -2
\end{bmatrix}_{6\times 6}
\end{equation}
if $\det(\Lambda)$ is odd. In addition, for a fixed $L\in \Lambda$ with $L^2 > 0$, there exists a K3 surface $X$ such
that $\Pic(X)\cong \Sigma$ and $\sigma(L)$ is big and nef on $X$. 
\end{Proposition}

We refer the reader to \cite{generic} for the proof of the above proposition, which involves some basic lattice theory.

Our second step differs from that in \cite{generic}. Instead of further degenerating K3 surfaces of lattices
\eqref{K3RATNOTESE007} and \eqref{K3RATNOTESE019} to unions of rational surfaces, we show that there are some ``basic''
nodal rational curves including $(-2)$-curves meeting transversely on these K3 surfaces.  Then we can show that there
exist unions of these curves that can be deformed to rational curves in the desired divisor classes on a general K3
surface of Picard rank $2$. This approach will eliminate the technical difficulty of producing log rational curves on a
log K3 surface \cite[Theorem 3.10, Proposition 3.12]{generic}. On the other hand, we cannot guarantee that the resulting
rational curves are nodal. So Theorem \ref{thm: appendix} is weaker than the main theorem in \cite{generic}.
Nevertheless, it is sufficient for the purposes of this paper, as explained in Section \ref{subsect:generic}.

In what follows, we will repeatedly make use of the following set which is the set of indivisible effective
isotropic classes in $\Pic(X)$ with intersection matrix \eqref{K3RATNOTESE007}:
\begin{equation}\label{CGLAPDE011}
\begin{aligned}
\Pi &= \big\{ A - E_a: 1\le a \le 6\big\}
\cup \big\{ 3A - E_a - \sum E_j: 1\le a \le 6\big\}
\\
&\quad \cup \big\{ 2A - \sum_{j\ne a, b} E_j: 1\le a<b\le 6 \big\}.
\end{aligned}
\end{equation}

\begin{Proposition}\label{prop: (-2)-curves even lattice}
Let $X$ be a general K3 surface whose Picard lattice is generated by effective divisors $A, E_1,E_2,...,E_6$ with intersection matrix \eqref{K3RATNOTESE007}. Then
\begin{enumerate}
\item the sum $\sum C_i$ of $(-2)$-curves on $X$ has simple normal crossings;
\item for every sequence of classes $D_1,D_2,...,D_m\in\Pi$,
there are nodal rational curves $F_j\in |D_j|$ on $X$ such that
\[
F_j + F_{j+1} + \sum C_i
\]
has normal crossings for $j=1,2,...,m$ (set $D_{m+1}=0$ and $F_{m+1} = 0$);
\item for every sequence of classes $D_1,D_2,...,D_m\in \Pi$, there are nodal rational curves $F_j\in |D_j|$ on $X$ such that
\[
\begin{aligned}
&F_j = F_{j+1} \hspace{12pt}\text{ if } D_j = D_{j+1}\hspace{12pt} \text{ and }\\
&F_j + F_{j+1} + \sum C_i \hspace{12pt} \text{ has normal crossings if } D_j\ne D_{j+1}
\end{aligned}
\]
for $j=1,2,...,m$.
\end{enumerate}
\end{Proposition}
\begin{proof}
Such a K3 surface $X$ can be realised as a double cover $\varphi: X\to S$ of a del Pezzo surface $S$ of degree $3$
ramified over a general curve $R\in |-2K_S|$ \cite[Remark 3.4]{generic}. Every $(-2)$-curve on $X$ is the preimage of a
$(-1)$-curve on $S$ under $\varphi$. Thus, the union of $(-2)$-curves on $X$ has simple normal crossings if and
only if $R + I$ has simple normal crossings, where $I = \sum I_j$ is the union of $(-1)$-curves on $S$. So it suffices to show that
\begin{itemize}
\item no three among $R, I_1,I_2,..., I_{27}$ meet at one point;
\item any two curves among $R, I_1,I_2,..., I_{27}$ meet transversely.
\end{itemize}
It is easy to see that these hold for a general del Pezzo surface $S$ of degree $3$. So the sum $\sum C_i$ of $(-2)$-curves on $X$ has simple normal crossings. This proves (1).

Now let us degenerate a K3 surface of Picard lattice \eqref{K3RATNOTESE007} to a K3 surface with Picard lattice
\begin{equation}
\label{CGLAPDE000}
\begin{bmatrix}
2\\
& -2\\
&& \ddots\\
&&& -2
\end{bmatrix}_{9\times 9}
\end{equation}
To set this up, we let $X/B$ be a family of K3 surfaces over a DVR whose central fibre $X_0$ is a K3 surface with Picard
lattice $\Pic(X_0)$ given by \eqref{CGLAPDE000} and generated by effective divisors $A, E_1, E_2,..., E_6, E_7, E_8$,
where $A, E_1,E_2,...,E_6$ extend over $B$ and generate $\Pic(X_\eta)$.

The central fibre $X_0$ is a double cover $\varphi: X_0\to S_0$ of a del Pezzo surface $S_0$ of degree $1$ ramified over
a smooth curve in $|-2K_{S_0}|$. By the same argument as before, we see that the sum of $(-2)$-curves on $X_0$ has simple normal crossings
for a general choice of $X_0$.

We can write each $D_j$ as
\[
\begin{aligned}
D_j &= E_7 + (D_j - E_7) = \Gamma_{j1} + \Gamma_{j2}
\hspace{12pt} \text{if $j$ is odd,}\\ 
D_j &= E_8 + (D_j - E_8) = \Gamma_{j1} + \Gamma_{j2}
\hspace{12pt} \text{if $j$ is even}.
\end{aligned}
\]
Since $\Gamma_{j1}$ and $\Gamma_{j2}$ are $(-2)$-curves on $X_0$ meeting transversely, $\Gamma_{j1} \cup \Gamma_{j2}$
deforms to a nodal rational curve $F_j\in |D_j|$ on $X_\eta$.  By our choice of $\Gamma_{j1}$ and $\Gamma_{j2}$, it is
clear that
\[
\Gamma_{j1} + \Gamma_{j2} + \Gamma_{j+1,1} + \Gamma_{j+1,2} + \sum C_i
\]
has normal crossings for $j=1,2,...,m-1$. So the same holds on $X_\eta$ too. This proves (2).

The last statement (3) follows easily from (2) by applying it to
$D_1', D_2', ..., D_k'$ given by
\[
\begin{aligned}
D_1' &= D_1 = D_2 = ... = D_{i_1}\\
\ne D_2' &= D_{i_1 + 1} = D_{i_1+2} = ... = D_{i_2} \ne ...\\
\ne D_k' &= D_{i_{k-1} + 1} = D_{i_{k-1} + 2} = ... = D_m.
\end{aligned}
\]
\end{proof}

\begin{Lemma}\label{lem: nefeven}
Let $X$ be a general K3 surface whose Picard lattice is generated by effective divisors $A, E_1, E_2,...,E_6$ with intersection matrix \eqref{K3RATNOTESE007}. Then for every nef divisor $D$ on $X$, there exists a set $M$ of $5$ disjoint $(-2)$-curves on $X$ such that for every $I\in M$, 
there exist $d_1, d_2, ..., d_m\in \ZZ_{\ge 0}$ and $D_1,D_2,...,D_m\in \Pi$ satisfying $D_1 I \ge 2$ and
\begin{equation}
\label{K3RATNOTESE900}
D = d_1 (I + D_1) + d_2 D_2 + d_3 D_3 + ... + d_m D_m
\end{equation}
in $\Pic(X)$, where $\Pi$ is defined by \eqref{CGLAPDE011}.
\end{Lemma}

\begin{proof}
Let $\Aut(\Sigma)$ be the lattice automorphism group of $\Sigma = \Pic(X)$, and let $\Aut(\Sigma)^+$ be the subgroup of
$\Aut(\Sigma)$ preserving the nef cone of $\Pic(X)$. We know that $\Aut(\Sigma)^+$ is generated by the permutations on
$\{E_i\}$ and the Cremona transformations of $\Pic(S)$ for the double cover $\varphi: X\to S$ over a del Pezzo surface
$S$ of degree $3$ \cite[Remark 3.4]{generic}. So for any two sets $T_1$ and $T_2$ of $k$ disjoint $(-2)$-curves, there
exists $\sigma\in \Aut(\Sigma)^+$ such that $\sigma(T_1) = T_2$.

A divisor $D$ on $X$ is nef if and only if $D I \geq 0$ for all $(-2)$-curves $I\subset X$. For $D$ nef, let $T$ be the set
of $(-2)$-curves $I$ such that $D I = 0$. If there exist $I_1, I_2\in T$ such that $I_1 I_2 > 0$, then $(I_1 + I_2)^2 =
D (I_1+I_2) = 0$. And since $D$ is nef, we necessarily have
\[
D = d(I_1 + I_2)
\]
for some $d \in \ZZ_{\ge 0}$, where $I_1 + I_2\in \Pi$. We are done.

Otherwise, $I_1 I_2 = 0$ for all $I_1 \ne I_2\in T$. Then there exists $\sigma\in \Aut(\Sigma)^+$ such that $\sigma(T) = \{E_{k+1}, E_{k+2}, ..., E_6\}$. If $k = 0$, then we let
\[
M = \big\{E_1,E_2,...,E_5\big\}
\]
and for every $I = E_i\in M$, 
\[
\sigma(D) = d A = d (E_i + (A - E_i))
\]
for some $d\in \ZZ_{\ge 0}$ and we are done.

Otherwise, 
\[
\sigma(D) = dA - m_1 E_1 - m_2 E_2 - ... - m_k E_k
\]
for some positive integers $d, m_1,m_2,..., m_k$. Without loss of generality, let us assume that
\[
m_1 \le m_2 \le ... \le m_k.
\]
Note that $D I > 0$ for all $(-2)$-curves $I\not\in T$. So
\[
\begin{aligned}
d &> m_k\\
d &> m_{k-1} + m_k\hspace{12pt} \text{if } k \ge 2\\
2d &>m_{k-4} + m_{k-3} + m_{k-2} + m_{k-1} + m_k \hspace{12pt} \text{if } k\ge 5.
\end{aligned}
\]
It is easy to see that $\sigma(D) - (A - E_k)$ is nef if $k \le 5$. Then the lemma follows from induction on $-K_X D$.

Suppose that $k=6$. Then $\sigma(D) - (2A - E_3-E_4-E_5-E_6)$ is nef unless $d=3$ and $m_i = 1$ for $i=1,2,...,6$.

If $\sigma(D) - (2A - E_3-E_4-E_5-E_6)$ is nef, then it follows from induction. Otherwise,
\[
\sigma(D) = 3A - \sum_{i=1}^6 E_i.
\]
We let
\[
M = \big\{A - E_a - E_6: a = 1,2,...,5\big\}.
\]
For each $I = A - E_a - E_6\in M$,
\[
\sigma(D) = (A - E_a - E_6) + (2A - \sum_{i\ne a,6} E_i)
\]
and we are done.
\end{proof}

Now we are ready to prove Theorem \ref{thm: appendix} for $\det(\Lambda)$ even, except for the following two cases:
\begin{equation}
\label{K3RATNOTESE901}
\Lambda = \ZZ A \oplus \ZZ F,\
\begin{bmatrix}
A^2 & AF\\
AF & F^2
\end{bmatrix}
= \begin{bmatrix}
2 & 2\\
2 & 0
\end{bmatrix}
\text{ and } L = A + d F \text{ for } d>0,
\end{equation}
\begin{equation}
\label{K3RATNOTESE902}
\Lambda = \ZZ A \oplus \ZZ F,\
\begin{bmatrix}
A^2 & AF\\
AF & F^2
\end{bmatrix}
= \begin{bmatrix}
0 & 2\\
2 & 0
\end{bmatrix}
\text{ and } L = A + d F \text{ for } d>0.
\end{equation}
We will deal with the above two cases separately later.

\begin{proof}[Proof of Theorem \ref{thm: appendix} for $\det(\Lambda)$ even except for cases \eqref{K3RATNOTESE901} and \eqref{K3RATNOTESE902}]
Let $X$ be a general K3 surface whose Picard lattice $\Sigma$ is generated by effective divisors $A,E_1,E_2,...,E_6$
with intersection matrix \eqref{K3RATNOTESE007}. By Proposition \ref{prop: rank two embedding}, for every $L\in \Lambda$
satisfying $L^2 > 0$, we can embed $\Lambda$ as a primitive sublattice of $\Sigma$ such that $L$ is big and nef on $X$.
We will construct a rational stable map $f: C\to X$ such that $f_* C \in |L|$ and $f$ can be deformed to an immersive
map $f': \PP^1 \to X'$ to a K3 surface $X'$ with $\Pic(X') = \Lambda$.

By Lemma \ref{lem: nefeven}, there exists a set $M$ of $5$ disjoint $(-2)$ curves on $X$ such that for every $I\in M$, there
exist nonnegative integers $d_1,d_2,...,d_m$ and $D_1, D_2,...,D_m\in \Pi$ such that
\[
L = d_1(I+D_1) + d_2 D_2 + ... + d_m D_m.
\]
Since $\Lambda$ has rank $2$ and signature $(1,1)$, $\Lambda \cap M$ contains at most one $(-2)$-curve in $M$. So we may choose $I$ such that $I\not\in \Lambda$.
Let us also assume that $d_2,...,d_m > 0$ and $D_2, D_3,...,D_m$ are distinct.

We will argue in the following four cases:
\begin{equation}\label{CGLAPDE001}
\text{Either } I + D_1\not\in \Lambda \text{ and } d_1 > 0
\text{ or }
D_i\not\in \Lambda \text{ for some } 2\le i \le m.
\end{equation}
If this fails, we will end up in one of the next three cases.
\begin{equation}\label{CGLAPDE002}
L = d_1(I + D_1) + d_2 D_2,\ d_1 > 0
\text{ and } I+D_1, D_2\in \Lambda,
\end{equation}
\begin{equation}\label{CGLAPDE003}
L = d_2 D_2 + d_3 D_3 \text{ and } D_2, D_3\in \Lambda,
\end{equation}
\begin{equation}\label{CGLAPDE004}
L = d_1 (I + D_1),\ d_1 > 0 \text{ and } I + D_1\in \Lambda.
\end{equation}
Note that we cannot have $L = d_i D_i$ since $L$ is big.

Suppose that \eqref{CGLAPDE001} holds. Then we can find 
$W\in \Lambda^\perp$ such that $W(I + D_1) \ne 0$ if $I+D_1\not\in \Lambda$ and $d_1 > 0$ or $W D_i \ne 0$ if $D_i\not\in \Lambda$, where $\Lambda^\perp$ is the orthogonal complement of $\Lambda$ in $\Sigma$.
Let us write
\[
\begin{aligned}
L &= d_1(I+D_1) + d_2 D_2 + ... + d_m D_m\\
&= B_1 + B_2 + ... + B_d
\end{aligned}
\]
where $d = d_1 + d_2 + ... + d_m$,
each $B_i$ is one of $I+D_1, D_2, D_3, ..., D_m$, and we arrange $B_i$ in the order of
\begin{equation}\label{CGLAPDE006}
W B_1 \ge WB_2 \ge ... \ge WB_d.
\end{equation}
Note that $W L = 0$. By our choice of $W$, $W B_i \ne 0$ for at least one $i$. So we conclude that
\begin{equation}\label{CGLAPDE005}
W(B_1 + B_2 + ... + B_i) > 0 \Rightarrow
B_1 + B_2 + ... + B_i \not\in \Lambda
\end{equation}
for all $i=1,2,...,d-1$.

Let $G_1,G_2,...,G_d$ be the classes given by
\[
G_i = \begin{cases}
D_j & \text{if } B_i = D_j\\
D_1 & \text{if } B_i = I + D_1
\end{cases}
\]
for $i=1,2,...,d$.

By Proposition \ref{prop: (-2)-curves even lattice}, we can find nodal rational curves $F_i\in |G_i|$ such that
\[
F_i = F_{i+1} \text{ if } G_i = G_{i+1} \text{ and }
F_i + F_{i+1} + \sum I_j \text{ has normal crossings if } G_i\ne G_{i+1}
\]
for $i=1,2,...,d$ (set $F_{d+1} = G_{d+1} = 0$), where $\sum I_j$ is the sum of all $(-2)$-curves on $X$.

Note that $F_i F_{i+1} \ge 2$ if $F_i\ne F_{i+1}$. If $F_{i-1}\ne F_i$ and $F_i\ne F_{i+1}$, we can find two distinct points $q_1$ and $q_2$ on $F_i$ such that $q_1\in F_{i-1}\cap F_i$ and $q_2\in F_i\cap F_{i+1}$.

Let $f: C\to X$ be the rational stable map given as follows:
\begin{itemize}
\item $C = C_1 \cup C_2 \cup ... \cup C_d$ and $f_* C_i \in |B_i|$
for $i=1,2,...,d$.
\item If $B_i \in \{ D_j\}$, then $f: C_i\cong \PP^1 \to X$ is the normalisation of $F_i$.
\item If $B_i = I + D_1$, then $C_i = C_{i,1} \cup C_{i,2}$, where
\[
\begin{aligned}
C_{i,1}, C_{i,2} &\cong \PP^1,\ f_* C_{i,1} = F_i,\ f_* C_{i,2} = I,\\
C_{i,1}\cap C_{i,2} &\ne \emptyset \text{ and } C_{i,2}\cap C_k = \emptyset
\text{ for all } k\ne i.
\end{aligned}
\]
\item $C_i$ and $C_{i+1}$ meet at one point $p_i$ for
$i=1,2,...,d-1$.
\item If $F_i = F_{i+1}$, then $f(p_i) = q_i$ is the node of $F_i$
and $f: C_i\cup C_{i+1}\to F_i$ is a local isomorphism at $p_i$, i.e., $f$ maps $C_i$ and $C_{i+1}$ locally at $p_i$ to the two branches of $F_i$ at $q_i$.  
\end{itemize}
Roughly speaking, $C$ is a chain of smooth rational curves
$C_{1,1}\cup C_{2,1}\cup ... \cup C_{d,1}$ with $C_{i,2}$ attached to $C_{i,1}$, where we let $C_{i,1} = C_i$ and $C_{i,2} = \emptyset$ if $f_* C_i = F_i$. The dual graph of such $f$ is illustrated in Figure \ref{CGLFIG000}.

\begin{figure}[!htb]
\begin{tikzpicture}[scale=1.5]
\node[below] at (0,0) {\small $C_{1,1}$};
\node[below] at (1,0) {\small $C_{2,1}$};
\node[below] at (2,0) {\small $C_{i,1}$};
\node[below] at (3,0) {\small $C_{d-1,1}$};
\node[below] at (4,0) {\small $C_{d,1}$};
\node[above] at (2,1) {\small $C_{i,2}$};

\draw[thick] (0,0) -- (1,0);
\draw[thick, dashed] (1,0) -- (2,0);
\draw[thick, dashed] (2,0) -- (3,0);
\draw[thick] (3,0) -- (4,0);
\draw[thick] (2,1) -- (2,0);

\draw[fill, gray!50] (0,0) circle(0.05);
\draw[fill, gray!50] (1,0) circle(0.05);
\draw[fill, gray!50] (2,0) circle(0.05);
\draw[fill, gray!50] (3,0) circle(0.05);
\draw[fill, gray!50] (4,0) circle(0.05);
\draw[fill, gray!50] (2,1) circle(0.05);

\node[below] at (2,-0.5) {\small $f_* C_i = F_i+I$};
\end{tikzpicture}
\caption{Dual graph of $f:C\to X$ in \eqref{CGLAPDE001}}\label{CGLFIG000}
\end{figure}

Clearly, $f_* C \in |L|$ and
$f$ is an isomorphism onto its image locally at every point of $C$.
So $f$ is rigid and it can be deformed to a rational stable map
$f': C' \to X'$ over a K3 surface $X'$ with $\Pic(X') = \Lambda$.

If $C'$ is reducible, then $C$ ``splits'' as
\[
C = \Gamma_1 + \Gamma_2
\]
for some connected curves $\Gamma_1, \Gamma_2\ne \emptyset$ such that
$f_* \Gamma_1, f_*\Gamma_2 \in \Lambda$. There are two possibilities
for $\Gamma_1$ and $\Gamma_2$:
\begin{itemize}
\item One of $\Gamma_1$ and $\Gamma_2$ is $C_{i,2}$ for some $i$. That is, $C$ splits as
\[
C = (C_1 + ... + C_{i-1} + C_{i,1} + C_{i+1} + ... + C_d) + C_{i,2}.
\]
But $f_* C_{i,2} = I\not\in \Lambda$.
\item One of $\Gamma_1$ and $\Gamma_2$ is $C_1 + C_2 + ... + C_i$. That is, $C$ splits as
\[
C = (C_1 + C_2 + ... + C_i) + (C_{i+1} + C_{i+2} + ... + C_d).
\]
for some $i=1,2,...,d-1$. But
\[
f_* (C_1 + C_2 + ... + C_i) = B_1 + B_2 + ... + B_i\not\in \Lambda
\]
by \eqref{CGLAPDE005}.
\end{itemize}
In conclusion, $C'$ is irreducible and hence $C'\cong \PP^1$. 
In addition, since $f$ is a local isomorphism, $f': \PP^1\to X'$ is unramified.
This proves the theorem in the case of \eqref{CGLAPDE001}. For the other cases, we just have to apply some tweaks of the above construction.

Suppose that \eqref{CGLAPDE002} holds. If $d_1 > 1$, we write
\[
\begin{aligned}
L &= d_1(I + D_1) + d_2 D_2\\
&= D_1 + (2I + D_1) + 
\underbrace{(I+D_1) + ... + (I+D_1)}_{d_1-2}\\
&\quad  + \underbrace{D_2 + ... + D_2}_{d_2}\\
&= B_1 + B_2 + ... + B_d
\end{aligned}
\]
where $d=d_1+d_2$, and each $B_i$ is one of $D_1, I+D_1, 2I+D_1, D_2$. Since $I + D_1\in \Lambda$ and $I\not\in \Lambda$, $D_1\not\in \Lambda$.
So we can find $W\in \Lambda^\perp$ such that $W D_1 \ne 0$. We arrange $B_i$ in the order of \eqref{CGLAPDE006} and we still have \eqref{CGLAPDE005}.

Let us fix two nodal rational curves $F_1\in |D_1|$ and $F_2\in |D_2|$ such that
\[
F_1 + F_2 + \sum I_j
\]
has normal crossings for $(-2)$-curves $I_j\subset X$. We construct a rational stable map
$f: C\to X$ similarly as before:
\begin{itemize}
\item $C = C_1 \cup C_2 \cup ... \cup C_d$ and $f_* C_i \in |B_i|$
for $i=1,2,...,d$.
\item If $B_i = D_j$, then $f: C_i\cong \PP^1 \to X$ is the normalisation $F_j$.
\item If $B_i = I + D_1$, then $C_i = C_{i,1} \cup C_{i,2}$, where
\[
\begin{aligned}
C_{i,1}, C_{i,2}&\cong \PP^1,\ f_* C_{i,1} = F_1,\ f_* C_{i,2} = I,\\
C_{i,1} \cap C_{i,2} &\ne \emptyset\text{ and }
C_{i,2}\cap C_k = \emptyset \text{ for all } k\ne i.
\end{aligned}
\]
\item If $B_i = 2I + D_1$, then $C_i = C_{i,1} \cup C_{i,2}\cup C_{i,3}$, where
\[
\begin{aligned}
C_{i,1}, C_{i,2}, C_{i,3}&\cong \PP^1,\ f_* C_{i,1} = F_1,\ f_* C_{i,2} = f_* C_{i,3} = I,\\ 
C_{i,1}\cap C_{i,2} &\ne \emptyset,\ C_{i,1}\cap C_{i,3} \ne \emptyset,
\text{ and}\\
(C_{i,2}\cup C_{i,3})\cap C_k &= \emptyset \text{ for all } k\ne i.
\end{aligned}
\]
Note that this is possible since $ID_1 \ge 2$.
\item $C_i$ and $C_{i+1}$ meet at one point $p_i$ for $i=1,2,...,d-1$.
\item If $F_j \subset f(C_i) \cap f(C_{i+1})$ for some $j=1,2$, then $f(p_i) = q_j$ is the node of $F_j$
and $f: C_i\cup C_{i+1}\to F_j$ is a local isomorphism at $p_i$, i.e., $f$ maps $C_i$ and $C_{i+1}$ locally at $p_i$ to the two branches of $F_j$ at $q_j$.  
\end{itemize}
By the same argument as before, we see that $f: C\to X$ deforms to a rational stable map $f': \PP^1\to X'$ over a K3 surface $X'$ with $\Pic(X') = \Lambda$. The dual graph of such $f$ is illustrated in Figure \ref{CGLFIG001}.

\begin{figure}[!htb]
\begin{tikzpicture}[scale=1.5]
\node[below] at (0,0) {\small $C_{1,1}$};
\node[below] at (1,0) {\small $C_{2,1}$};
\node[below] at (2,0) {\small $C_{i,1}$};
\node[below] at (3,0) {\small $C_{j,1}$};
\node[below] at (4,0) {\small $C_{d-1,1}$};
\node[below] at (5,0) {\small $C_{d,1}$};
\node[above] at (2,1) {\small $C_{i,2}$};
\node[above] at (2.5,1) {\small $C_{j,2}$};
\node[above] at (3.5,1) {\small $C_{j,3}$};

\draw[thick] (0,0) -- (1,0);
\draw[thick, dashed] (1,0) -- (2,0);
\draw[thick, dashed] (2,0) -- (3,0);
\draw[thick, dashed] (3,0) -- (4,0);
\draw[thick] (4,0) -- (5,0);
\draw[thick] (2,1) -- (2,0);
\draw[thick] (2.5,1) -- (3,0) -- (3.5,1);

\draw[fill, gray!50] (0,0) circle(0.05);
\draw[fill, gray!50] (1,0) circle(0.05);
\draw[fill, gray!50] (2,0) circle(0.05);
\draw[fill, gray!50] (3,0) circle(0.05);
\draw[fill, gray!50] (4,0) circle(0.05);
\draw[fill, gray!50] (5,0) circle(0.05);
\draw[fill, gray!50] (2,1) circle(0.05);
\draw[fill, gray!50] (2.5,1) circle(0.05);
\draw[fill, gray!50] (3.5,1) circle(0.05);

\node[below] at (2,-0.5) {\small $f_* C_i = F_1+I$, $f_* C_j = F_1 + 2I$};
\end{tikzpicture}
\caption{Dual graph of $f:C\to X$ in \eqref{CGLAPDE002} and $d_1 > 1$}\label{CGLFIG001}
\end{figure}

Suppose that $d_1 = 1$. If $ID_2 > 0$, we write
\[
\begin{aligned}
L &= (I+D_1) + d_2 D_2 = D_1 + (I+D_2) + 
\underbrace{D_2 + ... + D_2}_{d_2 - 1}\\
&= B_1 + B_2 + ... + B_d
\end{aligned}
\]
where $d = 1 + d_2$ and each $B_i$ is one of $D_1, D_2, I+D_2$. Since $D_1\not\in \Lambda$, we can find $W\in \Lambda^\perp$ such that
$WD_1 \ne 0$. Again we arrange $B_i$ in the order of \eqref{CGLAPDE006} and we still have \eqref{CGLAPDE005}.
The rational stable map $f:C\to X$ can be constructed similarly as before. We leave the details to the reader.

Suppose that $d_1 = 1$ and $ID_2 = 0$. If $D_1 D_2 > 2$, we can find two $(-2)$-curves $I_1$ and $I_2$ such that
\[
I_1,I_2\not\in \Lambda,\ I_1 D_2 > 0,\ I_2 D_2 > 0, \text{ and }
D_1 = I_1 + I_2.
\]
Without loss of generality, suppose that $II_1 > 0$. We write
\[
\begin{aligned}
L &= (I+D_1) + d_2 D_2 = (I + I_1) + (I_2+D_2) + 
\underbrace{D_2 + ... + D_2}_{d_2 - 1}\\
&= B_1 + B_2 + ... + B_d
\end{aligned}
\]
where $d = 1 + d_2$ and each $B_i$ is one of $I+I_1, D_2, I_2+D_2$. Since $D_2\in \Lambda$ and $I_2\not\in \Lambda$, $I_2+D_2\not\in \Lambda$. So we can find $W\in \Lambda^\perp$ such that
$W(I_2 + D_2) \ne 0$. Again, we arrange $B_i$ in the order of \eqref{CGLAPDE006} and we still have \eqref{CGLAPDE005}.
The construction of the rational stable map $f:C\to X$ is left to the reader.

Finally for case \eqref{CGLAPDE002}, we are left with $d_1 = 1$, $ID_2 = 0$, and $D_1 D_2 = 2$. It is easy to see that $ID_1 = 2$. So $\Lambda$ is generated by $I+D_1$ and $D_2$ with intersection matrix
\[
\begin{bmatrix}
2 & 2\\
2 & 0
\end{bmatrix}
\]
and $L = (I+D_1) + d_2 D_2$. This is the case \eqref{K3RATNOTESE901}
which we will treat later.

Suppose that \eqref{CGLAPDE003} holds. Without loss of generality, let us assume that $d_2 \le d_3$. We fix two nodal rational curves $F_2\in |D_2|$ and $F_3\in |D_3|$ such that $F_2 + F_3 + \sum I_j$ has normal crossings for $(-2)$-curves $I_j\subset X$.

If $D_2D_3 > 2$, then as before, we can find two $(-2)$-curves $I_1$ and $I_2$ such that
\[
I_1,I_2\not\in \Lambda,\ I_1 D_3 > 0,\ I_2 D_3 > 0, \text{ and }
D_2 = I_1 + I_2.
\]
If $d_3 = 1$, then we simply let
$f: C\to X$ be the rational stable map such that
\[
\begin{aligned}
C &= C_1 \cup C_2\cup  C_3,\ C_1,C_2,C_3\cong \PP^1\\
f_* C_1 &= I_1,\ f_* C_2 = F_3,\ f_* C_3 = I_2\\
C_1\cap C_2 &\ne \emptyset, \text{ and }
C_2\cap C_3\ne \emptyset.
\end{aligned}
\]
If $d_3 > 1$, then we write
\[
\begin{aligned}
L &= d_2 D_2 + d_3 D_3 = (I_1 + D_3) + (I_2 + D_3)\\
&\quad + 
\underbrace{D_2 + ... + D_2}_{d_2 - 1} + \underbrace{D_3 + ... + D_3}_{d_3 - 2}\\
&= B_1 + B_2 + ... + B_d
\end{aligned}
\]
where $d = d_1+d_2 - 1$ and each $B_i$ is one of $D_2, D_3, I_1 +D_3, I_2+D_3$. Since $D_3\in \Lambda$ and $I_1\not\in \Lambda$, $I_1 + D_3\not\in \Lambda$. So we can find $W\in \Lambda^\perp$ such that
\[
W(I_1 + D_3) \ne 0.
\]
Again, we arrange $B_i$ in the order of \eqref{CGLAPDE006} and we still have \eqref{CGLAPDE005}. The construction of the rational stable map $f: C\to X$ is left to the reader.

Suppose that $D_2D_3 = 2$ and $d_2 > 1$. We can find a $(-2)$-curve $I_1$ such that $I_1\not\in \Lambda$ and $I_1 D_2 = I_1 D_3 = 0$. Let
$I_2 = D_2 - I_1$.  
We write
\[
\begin{aligned}
L &= d_2 D_2 + d_3 D_3 = (I_2 + D_3) + (2I_1 + I_2 + D_3)\\
&\quad + 
\underbrace{D_2 + ... + D_2}_{d_2 - 2} + \underbrace{D_3 + ... + D_3}_{d_3 - 2}\\
&= B_1 + B_2 + ... + B_d
\end{aligned}
\]
where $d = d_1+d_2 - 2$ and each $B_i$ is one of $D_2, D_3, I_2+D_3, 2I_1 + I_2 + D_3$. Since $D_2, D_3\in \Lambda$ and $I_1\not\in \Lambda$, $I_2 + D_3\not\in \Lambda$. So we can find $W\in \Lambda^\perp$ such that
$W(I_2 + D_3) \ne 0$. Again, we arrange $B_i$ in the order of \eqref{CGLAPDE006} and we still have \eqref{CGLAPDE005}.
We can construct a rational stable map $f: C \to X$ with
\[
C = C_1\cup C_2\cup ...\cup C_d \text{ and } f_* C_i\in |B_i|
\]
as before. The only complication here is that when $B_i = 2I_1 + I_2 + D_3$, we let
\[
\begin{aligned}
C_i &= C_{i,1}\cup C_{i,2}\cup C_{i,3}\cup C_{i,4},\\ 
C_{i,1}, C_{i,2}, C_{i,3}, C_{i,4}&\cong \PP^1,\\
f_* C_{i,1} &= F_3,\ f_* C_{i,2} = I_2,\ f_* C_{i,3} = f_* C_{i,4} = I_1,\\
C_{i,1}\cap C_{i,2} &\ne\emptyset,\
C_{i,2}\cap C_{i,3} \ne \emptyset,\ 
C_{i,2}\cap C_{i,4} \ne\emptyset,\text{ and}\\
(C_{i,2}\cup C_{i,3}\cup C_{i,4}) \cap C_k &= \emptyset \text{ for all } k\ne i.
\end{aligned}
\]

For case \eqref{CGLAPDE003}, we are left with $d_2 = 1$ and $D_2D_3 = 2$. So $\Lambda$ is generated by $D_2$ and $D_3$ with intersection matrix
\[
\begin{bmatrix}
0 & 2\\
2 & 0
\end{bmatrix}
\]
and $L = D_2 + d_3 D_3$. This is the case \eqref{K3RATNOTESE902}
which we will treat later.

For the last case \eqref{CGLAPDE004}, we fix a nodal rational curve $F_1\in |D_1|$ such that $F_1 + \sum I_j$ has normal crossings for $(-2)$-curves $I_j\subset X$. If $d_1 = 1$, then we simply let
$f: C\to X$ be the rational stable map such that
\[
\begin{aligned}
C &= C_1 \cup C_2,\ C_1,C_2\cong \PP^1,\\
f_* C_1 &= F_1,\ f_* C_2 = I, \text{ and } C_1\cap C_2 \ne \emptyset.
\end{aligned}
\]
If $d_1 > 1$, then we write
\[
\begin{aligned}
L &= d_1 (I+D_1) = D_1 + (2I+ D_1) + 
\underbrace{(I+ D_1) + ... + (I + D_1)}_{d_1 - 2}\\
&= B_1 + B_2 + ... + B_d
\end{aligned}
\]
where $d = d_1$ and each $B_i$ is one of $D_1, I+D_1, 2I +D_1$. Since $I+D_1\in \Lambda$ and $I\not\in \Lambda$, $D_1\not\in \Lambda$. So we can find $W\in \Lambda^\perp$ such that
$W D_1 \ne 0$. Again, we arrange $B_i$ in the order of \eqref{CGLAPDE006} and we still have \eqref{CGLAPDE005}.
We once again leave the construction of the rational stable map $f:C\to X$ to the reader.
\end{proof}

Let us take care of the cases \eqref{K3RATNOTESE901}
and \eqref{K3RATNOTESE902}.

It is easy to check that general K3 surfaces with Picard lattices given in \eqref{K3RATNOTESE901} and \eqref{K3RATNOTESE902} can
be specialised to K3 surfaces with Picard lattices
\begin{equation}\label{CGLAPDE007}
\begin{bmatrix}
-2 & 3 & 1\\
3 & -2 & 1\\
1 & 1 & 0
\end{bmatrix}
\end{equation}
and
\begin{equation}\label{CGLAPDE008}
\begin{bmatrix}
-2 & 2 & 1\\
2 & -2 & 1\\
1 & 1 & 0
\end{bmatrix}
\end{equation}
respectively.

\begin{Proposition}\label{prop: even special cases}
Let $X$ be a general K3 surface with Picard lattice generated by effective divisors $A_1, A_2, F$ with intersection
matrix \eqref{CGLAPDE007} or \eqref{CGLAPDE008}. Then there exists a nodal rational curve $\Gamma\in |F|$ such that $A_1
+ A_2 + \Gamma$ has normal crossings.
\end{Proposition}
\begin{proof}
We will show in what follows that such a K3 surface $X$ can be realised as a special double cover of 
a rational ruled surface $S = \FF_a = \PP(\OO_{\PP^1}\oplus
\OO_{\PP^1}(a))$, where $a = 1$ if $\Pic(X)$ is given by \eqref{CGLAPDE007}
and $a = 0$ if $\Pic(X)$ is given by \eqref{CGLAPDE008}.

Suppose that $\Pic(S)$ is generated by effective divisors
$B$ and $G$ satisfying
\[
BG = 1,\ G^2 = 0,\text{ and } B^2 = a.
\]
We fix two smooth curves $C\in |B|$ and $H\in |G|$ and $a+3$ distinct points
$p_1,p_2,...,p_{a+2}\in C\backslash H$ and $p_{a+3}\in H\backslash C$. Then there exists a smooth curve $R\in |-2K_S|$ such that
$R$ is tangent to $C$ and $H$ at $p_1,p_2,...,p_{a+2}$ and $p_{a+3}$, respectively, and meets $H$ transversely at two other points outside of $p_{a+3}$.

Let $\varphi: X\to S$ be the double cover of $S$ ramified over $R$. Then
\[
\varphi^* C = A_1 + A_2\hspace{12pt}\text{and}\hspace{12pt}
\varphi^* H = \Gamma,
\]
where $A_1$ and $A_2$ are two $(-2)$-curves on $X$ with
\[
A_1 A_2 = a+2 \hspace{12pt}\text{and}\hspace{12pt}
A_1 F = A_2 F = 1 \text{ for }
F = \varphi^* G,
\]
and $\Gamma$ is a nodal rational curve in $|F|$. Obviously, $A_1$ and $A_2$ meet transversely at $a+2$ points $\varphi^{-1}(p_i)$ for $i=1,2,...,a+2$. So $A_1 + A_2 + \Gamma$ has normal crossing. For a general choice of $R$, $X$ is a K3 surface with Picard lattice \eqref{CGLAPDE007} or \eqref{CGLAPDE008}.
\end{proof}

\begin{proof}[Proof of Theorem \ref{thm: appendix} in the cases \eqref{K3RATNOTESE901} and \eqref{K3RATNOTESE902}]
Let $X$ be a general K3 surface with Picard lattice $\Sigma$ generated by effective divisors $A_1, A_2, F$ with intersection matrix \eqref{CGLAPDE007} or \eqref{CGLAPDE008}. We can embed $\Lambda$ into $\Sigma$ such that $\Lambda$ is generated by $A = A_1 + A_2$ and $F$ and
\[
L = A + dF = (A_1 + A_2) + dF.
\]
By Proposition \ref{prop: even special cases}, there exists a nodal rational curve $\Gamma\in |F|$ such that
\[
A_1 + A_2 + \Gamma
\]
has normal crossings.

Let $f: C\to X$ be the rational stable map given as follows:
\begin{itemize}
\item $C = C_0\cup C_1\cup ...\cup C_d \cup C_{d+1}$ and $C_i\cong \PP^1$ for $i=0,1,...,d$.
\item $C_i$ and $C_{i+1}$ meet at a point $p_i$ for $i=0,1,...,d$.
\item $f_* C_0 = A_1$, $f_* C_{d+1} = A_2$, and
$f_* C_i = \Gamma$ for $i=1,2,...,d$.
\item For $i=1,2,...,d-1$, $f(p_i) = q$ is the node of $\Gamma$
and $f: C_i\cup C_{i+1}\to \Gamma$ is a local isomorphism at $p_i$, i.e., $f$ maps $C_i$ and $C_{i+1}$ locally at $p_i$ to the two branches of $\Gamma$ at $q$.
\end{itemize}
It is easy to see that $f: C\to X$ deforms to a rational stable map
$f': \PP^1 \to X'$ to a K3 surface $X'$ with Picard lattice $\Pic(X') = \Lambda$. Since $f$ is immersive, the same holds for $f'$.
\end{proof}

This finishes the proof of Theorem \ref{thm: appendix} for $\det(\Lambda)$ even.

\begin{Proposition}\label{prop: (-2)-curves odd lattice}
Let $X$ be a general K3 surface whose Picard lattice is generated by effective divisors $A, F,E_1,E_2,E_3,E_4$ with intersection matrix \eqref{K3RATNOTESE019}. Then
there exist nodal rational curves
\begin{equation}\label{CGLAPDE010}
\Gamma_1 \in |A| \hspace{12pt}\text{and}\hspace{12pt} \Gamma_2\in |4A + 2F - E_1 - E_2 - E_3 - E_4|
\end{equation}
such that
\[
\Gamma_1 + \Gamma_2 + \sum C_i
\]
has normal crossings, where $\sum C_i$ is the sum of all $(-2)$-curves on $X$.
\end{Proposition}

\begin{proof}
It is easy to see that the $(-2)$-curves on $X$ are
\cite[Remark 3.7]{generic}
\[
F,\ E_i, \text{ and } A - E_i \text{ for } i = 1,2,3,4.
\]
To show that $\Gamma_1 + \Gamma_2 + \sum C_i$ has normal crossings, it suffices to show that
each of $\sum E_i + \sum (A - E_i)$, $\Gamma_2 + \sum E_i$ and $\Gamma_1 + \Gamma_2$
has normal crossings.

Instead of working with $X$, we specialize it to a K3 surface $Y$ with Picard lattice
\[
\begin{bmatrix}
0 & 1 & 1\\
1 & -2 & 0\\
1 & 0 & -2\\
&&& -2\\
&&&& \ddots\\
&&&&& -2
\end{bmatrix}_{7\times 7}
\]
Let $A, F, F', E_1,E_2,E_3,E_4$ be the effective generators of $\Pic(Y)$ with the above intersection matrix. We may regard $\Pic(X)$ as the sublattice of $\Pic(Y)$ generated by $A, F, E_1,E_2,E_3,E_4$. So it suffices to prove the statement on $Y$.

Such $Y$ can be constructed as follows.
Let $S$ be the blowup of the rational ruled surface $\FF_2$ at $4$ general points, and let $\varphi: Y\to S$ be a double cover of $S$ ramified over a general curve $R\in |-2K_S|$.

Suppose that $\Pic(S)$ is generated by effective divisors
$B, G, I_1, I_2, I_3, I_4$, where $B$ and $G$ are the pullback of the divisors on ${\mathbb F}_2$ satisfying $B^2 = 0$, $BG=1$, and $G^2 = -2$, and $I_j$ are the exceptional divisors of the blowup $S\to \FF_2$. Then
\[
\varphi^* B = A,\ \varphi^* G = F + F', \text{ and } \varphi^* I_j = E_j
\text{ for } j=1,2,3,4.
\]
Since $G + \sum I_j + \sum (B - I_j)$ has simple normal crossings, the same holds
for $\sum C_i$ on $Y$.

Let us consider the rational curves in $|8 B + 2G - 2\sum I_j|$. By deforming a union of two curves in $|4B + G - \sum I_j|$, we see that there is a nodal rational curve
\[
D \in |8B + 2G - 2\sum I_j| = |-2K_S - 2G|
\]
since $-K_S - G$ is very ample; in addition, a general such $D$ meets $I_j$ transversely for $j=1,2,3,4$.

We claim that there is a smooth curve $R\in |-2K_S|$ such that $R$ is tangent to $D$ at eight distinct points. Let $M$ be a general curve in $|-K_S|$ meeting $D$ transversely at eight distinct points $p_1,p_2,...,p_8$. It is easy to see that a general member $R$ of the pencil in $|-2K_S|$ spanned by $2M$ and $D + 2G$ is a smooth curve tangent to $D$ at $p_1,p_2,...,p_8$. If 
$\varphi: Y\to S$ is ramified over $R$, then
\[
\varphi^* D = \Gamma_2 + \Gamma_2'
\]
where $\Gamma_2$ and $\Gamma_2'$ are two nodal rational curves in
\[
\Gamma_2\in |4A + 2F - \sum E_j| \hspace{12pt}\text{and}
\hspace{12pt} \Gamma_2'\in |4A + 2F' - \sum E_j|,
\]
respectively. This produces a nodal rational curve $\Gamma_2\in |4A + 2F - \sum E_j|$. Since $D + \sum I_j$ has normal crossings, the same holds for $\Gamma_2 + \sum E_j$. It remains to find a nodal rational curve $\Gamma_1\in |A|$ such that $\Gamma_1 + \Gamma_2$ has normal crossings.

The above shows that for a general curve $R\in |-2K_S|$, there is a nodal rational curve $D\in |-2K_S - 2G|$ tangent to $R$ at eight distinct points. For a general curve $R\in |-2K_S|$, there are exactly $24$ smooth curves $J_1, J_2, ..., J_{24}$ in $|B|$ such that each $J_i$ meet $R$ at three distinct points, i.e., $R$ is tangent to $J_i$ at one of the three points in $R\cap J_i$. This follows from the standard argument by showing the incidence correspondence
\[
\Big\{(R, J): R\in |-2K_S| \text{ and } J\in |B|
\text{ are smooth and tangent to each other}
\Big\}
\]
is irreducible of dimension $h^0(-2K_S) - 1$. On the other hand, by examining the projection $D\subset \FF_2\to \PP^1$, we see that there are at most three curves in $|B|$ that fail to meet $D$ transversely. Therefore, there is $J\in \{J_1,J_2,....,J_{24}\}$ such that $J$ meets $D$ transversely and meets $R$ at three distinct points. Then
\[
\Gamma_1 = \varphi^* J
\]
is a nodal rational curve meeting $\Gamma_2$ transversely.
\end{proof}

\begin{proof}[Proof of Theorem \ref{thm: appendix} for $\det(\Lambda)$ odd]
Let $X$ be a general K3 surface whose Picard lattice $\Sigma$ 
is generated by effective divisors $A,F,E_1,E_2,E_3,E_4$ with intersection matrix 
\eqref{K3RATNOTESE019}. By Proposition \ref{prop: rank two embedding}, for every $L\in \Lambda$ satisfying $L^2 > 0$, we can embed
$\Lambda$ as a primitive sublattice of $\Sigma$ such that $L$ is big and nef on $X$.
We will construct a rational stable map $f: C\to X$ such that
$f_* C \in |L|$ and $f$ can be deformed to an immersive map $f': \PP^1 \to X'$ 
to a K3 surface $X'$ with $\Pic(X') = \Lambda$.

We claim that
\begin{equation}\label{CGLAPDE009}
\Lambda \cap (\ZZ E_1 \oplus \ZZ E_2 \oplus \ZZ E_3\oplus \ZZ E_4) = 0.
\end{equation}
Otherwise, $G = m_1 E_1 + m_2 E_2 + m_3 E_3 + m_4 E_4\in \Lambda$ for some coprime integers $m_1,m_2,m_3,m_4$. Then we can find $H\in \Lambda$ such that $G$ and $H$ generate the lattice $\Lambda$. Since $GD$ is even for all $D\in \Sigma$, $\det(\Lambda)$ must be even. This is a contradiction.

For every $D\in \Sigma$, there is at most one class in $\{D - E_i: i=1,2,3,4\}$ belonging to $\Lambda$. Therefore, for any three classes $D_1, D_2, D_3\in \Sigma$, there exists $E_i$ such that $D_j - E_i\not\in \Lambda$ for $j=1,2,3$.
Without loss of generality, let us assume that
\begin{equation}\label{CGLAPDE013}
A - E_1,\ A + F - E_1,\ 2A + F - E_1 \not\in \Lambda.
\end{equation}

By Proposition \ref{prop: (-2)-curves odd lattice},
we can find two nodal rational curves in \eqref{CGLAPDE010} such that
\[
\Gamma_1 + \Gamma_2 + F + \sum_{i=1}^4 E_i + \sum_{i=1}^4 (A-E_i)
\]
has normal crossings, where we use $A-E_i$ to denote the $(-2)$-curve in $|A-E_i|$.

Here is a sketch of our argument. As before, we will construct $f: C\to X$ by writing
\[
L = B_1 + B_2 + ... + B_d
\]
such that at least one $B_i\not\in \Lambda$ and each $B_i$ is rationally equivalent to a union of rational curves. More precisely, we will choose each $B_i$ to be one of
\begin{equation}\label{CGLAPDE014}
\begin{aligned}
& A,\ 2A + F - E_1,\ 4A + 2F - E_T \text{ for } T\subset \{1,2,3,4\} \text{ and}\\
& 4A + 2F + E_i - E_T \text{ for } i\not\in T \subset \{1,2,3,4\}
\end{aligned}
\end{equation}
where we use the notation
\[
E_T = \sum_{j\in T} E_j
\]
for a subset $T$ of $\{1,2,3,4\}$.

If $B_i\not\in \Lambda$ for some $i$, then we can find $W\in \Lambda^\perp$ such that $W B_i \ne 0$. We arrange $B_i$ in the order of \eqref{CGLAPDE006} and we have \eqref{CGLAPDE005}. Then we construct the rational stable map $f: C\to X$ as follows:
\begin{itemize}
\item $C = C_1 \cup C_2 \cup ... \cup C_d$ and $f_* C_i \in |B_i|$
for $i=1,2,...,d$.
\item If $B_i = A$, then $C_i\cong \PP^1$ and $f_* C_i = \Gamma_1$.
\item If $B_i = 2A + F - E_1$, then
\[
\begin{aligned}
C_i &= C_{i,1} \cup C_{i,2}\cup C_{i,3},\\
C_{i,1}, C_{i,2}, C_{i,3} &\cong \PP^1,\\
f_* C_{i,1} &= \Gamma_1,\ f_* C_{i,2} = F,\ f_* C_{i,3} = A - E_1,\\
C_{i,1} \cap C_{i,2} &\ne \emptyset,\ C_{i,2}\cap C_{i,3} \ne \emptyset\text{ and}\\
(C_{i,2}\cup C_{i,3})\cap C_k &= \emptyset \text{ for all } k\ne i.
\end{aligned}
\]
\item If $B_i = 4A + 2F - E_T$ for some $T\subset \{1,2,3,4\}$, then
\[
\begin{aligned}
C_i &= C_{i,1} \cup C_{i,2}\cup ... \cup C_{i,a} \text{ for } a = 5 - |T|,\\
C_{i,j}&\cong \PP^1,\ f_* C_{i,j} \ne 0 \text{ for } j=1,2,...,a,\\
f_* C_{i,1} &= \Gamma_2,\ \sum_{j=2}^a f_* C_{i,j} = \sum_{k\not\in T} E_k,\\
C_{i,1}\cap C_{i,j} &\ne \emptyset \text{ for } j=2,...,a,
\text{ and}\\
\bigcup_{j=2}^a C_{i,j}\cap C_k &= \emptyset \text{ for all } k\ne i.
\end{aligned}
\]
\item If $B_i = 4A + 2F + E_b - E_T$ for some $b\not\in T\subset \{1,2,3,4\}$, then
\[
\begin{aligned}
C_i &= C_{i,1} \cup C_{i,2}\cup ... \cup C_{i,a} \text{ for } a = 6 - |T|,\\
C_{i,j}&\cong \PP^1,\ f_* C_{i,j} \ne 0 \text{ for } j=1,2,...,a,\\
f_* C_{i,1} &= \Gamma_2,\ \sum_{j=2}^a f_* C_{i,j} = E_b + \sum_{k\not\in T} E_k,\\
C_{i,1}\cap C_{i,j} &\ne \emptyset \text{ for } j=2,...,a,
\text{ and}\\
\bigcup_{j=2}^a C_{i,j}\cap C_k &= \emptyset \text{ for all } k\ne i.
\end{aligned}
\]
\item $C_i$ and $C_{i+1}$ meet at one point $p_i$ for $i=1,2,...,d-1$.
\item If $\Gamma_j \subset f(C_i) \cap f(C_{i+1})$ for some $j=1,2$, then $f(p_i) = q_j$ is the node of $\Gamma_j$
and $f: C_i\cup C_{i+1}\to \Gamma_j$ is a local isomorphism at $p_i$, i.e., $f$ maps $C_i$ and $C_{i+1}$ locally at $p_i$ to the two branches of $\Gamma_j$ at $q_j$.  
\end{itemize}
By the same argument as before, we see that $f: C\to X$ deforms to a rational stable map $f': \PP^1\to X'$ over a K3 surface $X'$ with $\Pic(X') = \Lambda$. The dual graph of such $f$ is illustrated in Figure \ref{CGLFIG002}.

\begin{figure}[!htb]
\begin{tikzpicture}[scale=1.5]
\node[below] at (0,0) {\small $C_{1,1}$};
\node[below] at (1,0) {\small $C_{2,1}$};
\node[below] at (2,0) {\small $C_{i,1}$};
\node[below] at (4,0) {\small $C_{j,1}$};
\node[below] at (5,0) {\small $C_{d-1,1}$};
\node[below] at (6,0) {\small $C_{d,1}$};
\node[right] at (2,1) {\small $C_{i,2}$};
\node[right] at (2,2) {\small $C_{i,3}$};
\node[above] at (3.5,1) {\small $C_{j,2}$};
\node[above] at (4.5,1) {\small $C_{j,a}$};

\draw[thick] (0,0) -- (1,0);
\draw[thick, dashed] (1,0) -- (2,0);
\draw[thick, dashed] (2,0) -- (4,0);
\draw[thick, dashed] (4,0) -- (5,0);
\draw[thick] (5,0) -- (6,0);
\draw[thick] (2,2) -- (2,1) -- (2,0);
\draw[thick] (3.5,1) -- (4,0) -- (4.5,1);
\draw[thick] (4,1) -- (4,0);
\draw[thick, dotted] (3.5,1) -- (4.5,1);

\draw[fill, gray!50] (0,0) circle(0.05);
\draw[fill, gray!50] (1,0) circle(0.05);
\draw[fill, gray!50] (2,0) circle(0.05);
\draw[fill, gray!50] (4,0) circle(0.05);
\draw[fill, gray!50] (5,0) circle(0.05);
\draw[fill, gray!50] (6,0) circle(0.05);
\draw[fill, gray!50] (2,1) circle(0.05);
\draw[fill, gray!50] (2,2) circle(0.05);
\draw[fill, gray!50] (3.5,1) circle(0.05);
\draw[fill, gray!50] (4,1) circle(0.05);
\draw[fill, gray!50] (4.5,1) circle(0.05);

\node[below] at (3,-0.5) {\small $f_* C_i = \Gamma_1+F + (A-E_1)$, 
$f_* C_j = \Gamma_2 + \sum E_k$};
\end{tikzpicture}
\caption{Dual graph of $f:C\to X$ for $\det(\Lambda)$ odd}\label{CGLFIG002}
\end{figure}

We claim that it is possible to carry out the above construction when
\begin{equation}\label{CGLAPDE012}
\text{either}\hspace{12pt} AL > 2 \hspace{12pt}\text{or}\hspace{12pt}
\sum_{i=1}^4 LE_i > 0.
\end{equation} 
Let
\[
L = n A + c F - \sum_{i=1}^4 m_i E_i.
\]
We know that $L$ is nef if and only if
\[
n \ge 2c \ge 4m_i\ge 0
\]
for $i=1,2,3,4$ \cite[Remark 3.7]{generic}. Then \eqref{CGLAPDE012} is equivalent to saying that either $c > 2$ or one of $m_i$ is positive. Note that if it is the latter, $c \ge 2m_i \ge 2$.

Suppose that \eqref{CGLAPDE012} holds.
Let us write
\[
L = m(4A + 2F) - \sum_{i=1}^4 m_i E_i + (n-4m) A + (c-2m) F \hspace{12pt}\text{for }
m = \left\lfloor \frac{c}2\right\rfloor.
\]
Since $m \ge m_j$ for all $j$, we can write $L = \sum B_i$ such that each $B_i$ is one of
\[
A,\ 2A + F,\ 4 A + 2F - E_T \text{ for } T\subset \{1,2,3,4\}
\]
and there is at most one $B_i = 2A + F$ (when $c$ is odd).

We need to apply some tweaks to the choice of $B_i$. If $B_i = 2A + F$ for some $i$, this only happens when $c\ge 3$. So there exists $B_j = 4A + 2F - E_T$ for some $T$. In this case, we change $B_i$ and $B_j$ to
\[
B_i = 2A + F - E_1 \hspace{12pt}\text{and}\hspace{12pt} B_j = 4A + 2F + E_1 - E_T.
\]
So we have
\[
L = B_1 + B_2 + ... + B_d
\]
where each $B_i$ is one of the classes among \eqref{CGLAPDE014} and at most one $B_i$ is $2A + F - E_1$ (when $c$ is odd).

We need to apply one more tweak. If there exist $i\ne j$ such that
\[
B_i = B_j = 4A + 2F - E_T
\]
for some $T\subsetneq \{1,2,3,4\}$, we choose some $b\in \{1,2,3,4\}\backslash T$ and change $B_i$ and $B_j$ to
\[
B_i = 4A + 2F + E_b - E_T \hspace{12pt}\text{and}\hspace{12pt} 
B_j = 4A + 2F - E_b - E_T.
\]
In conclusion, we can write
\[
L = B_1 + B_2 + ... + B_d
\]
for $B_i$ with the following properties:
\begin{itemize}
\item Each $B_i$ is one of the classes among \eqref{CGLAPDE014}.
\item At most one $B_i$ is $2A + F - E_1$ (when $c$ is odd).
\item There do not exist $i\ne j$ such that
\[
B_i = B_j = 4A + 2F - E_T
\]
for some $T\subsetneq \{1,2,3,4\}$.
\end{itemize}

If $d = 1$, then the rational stable map $f:C\to X$ constructed above 
obviously deforms to a rational stable map $f': \PP^1\to X'$ over a K3 surface $X'$ with $\Pic(X') = \Lambda$.

If $d \ge 2$, we claim that there is at least one $B_i\not\in \Lambda$. 
Otherwise, suppose that $B_1, B_2, ..., B_d\in \Lambda$. 

Since $2A + F - E_1\not\in \Lambda$ by \eqref{CGLAPDE013}, $B_i\ne 2A + F - E_1$
for all $i$. It is also obvious that there cannot be two classes $B_i\ne B_j$ in
\[
\Big\{ 4A + 2F - E_T\Big\}\cup \Big\{ 4A + 2F + E_b - E_T\Big\}.
\]
Otherwise, $B_i - B_j \in \Lambda$ and $B_i - B_j\in \sum \ZZ E_k$, contradicting
\eqref{CGLAPDE009}. And since $B_1,B_2,...,B_d$ contain at least two distinct classes,
we have either
\[
\begin{aligned}
B_1,B_2,...,B_d &\in \big\{ A,\ 4A + 2F + E_b - E_T\big\}\subset \Lambda \text{ for some } b\not\in T\subset \{1,2,3,4\}\\
\text{or }
B_1,B_2,...,B_d &\in \big\{ A,\ 4A + 2F - E_T\big\}\subset \Lambda \text{ for some } T\subset \{1,2,3,4\}.
\end{aligned}
\]
If it is the former, then $A$ and $4A + 2F + E_b - E_T$ generate the lattice $\Lambda$ and $\det(\Lambda)$ is even. If it is latter, then the same argument shows that $\det(\Lambda)$ is even if $T\ne\emptyset$. Therefore, $T=\emptyset$. Hence $B_i = A$ or $4A + 2F$ for all $i$.
By our choice of $B_i$, there is only one $4A + 2F$ among $B_i$. Thus, $c = 2$. But $m_j = 0$ for all $j$, contradicting \eqref{CGLAPDE012}.

In conclusion, there is at least one $B_i\not\in \Lambda$. This proves the existence of a rational stable map $f: C\to X$ which can be deformed to a stable map $f':\PP^1\to X'$, when \eqref{CGLAPDE012} holds.

We can prove another simple case when
\begin{equation}\label{CGLAPDE015}
F \not\in \Lambda.
\end{equation}
Suppose that $c\le 2$, $m_1=m_2=m_3=m_4=0$ and $F\not\in \Lambda$. We write
\[
\begin{aligned}
L &= n A + cF = \underbrace{(A+F)  + ... + (A+F)}_c + \underbrace{A + ... + A}_{n-c}\\
&= B_1 + B_2 + ... + B_d
\end{aligned}
\]
where $d = n$ and each $B_i$ is one of $A$ and $A+F$. Then there is at least one $B_i\not\in \Lambda$. We can construct a rational stable map $f:C \to X$ accordingly.

This proves the existence of an integral rational curve in $|L|$ with unramified normalisation on a general K3 surface $X'$ with Picard lattice $\Lambda$, if \eqref{CGLAPDE012} or \eqref{CGLAPDE015} hold.
To finish the proof of the theorem, it remains to show that the conditions A2 and A3 imply one of the above.

When A2 holds, $L = L_1 + L_2 + L_3$ for some $L_i\in \Lambda$ such that $LL_i > 0$ and $L_i^2 > 0$. It is easy to see that $L_i$ are effective and $A L_i > 0$ for $i=1,2,3$. Therefore, $c = A L \ge 3$. We are done.

When A3 holds,
$L = L_1 + L_2$ for some $L_i\in \Lambda$ satisfying that
$L L_i > 0$, $L_1^2 > 0$, $L_2^2 = -2$, and
$L_1 - L_2\not\in n\Lambda$ for all integers $n\ge 2$. Clearly, $L_i$ are effective, $AL_1 > 0$ and $AL_2 \ge 0$.

If $A L \ge 3$ or $\sum L E_i >0$ or $F\not\in \Lambda$, then we are done. Otherwise, $c = AL \le 2$, $m_1 = m_2 = m_3 = m_4 = 0$, and $F\in \Lambda$. Then
\[
L_1 = (n - l) A + (c - b) F - \sum_{i=1}^4 a_i E_i \hspace{12pt}\text{and}\hspace{12pt} L_2 = lA + bF + \sum_{i=1}^4 a_i E_i
\]
for some integers $a_i$, $0\le b < c$ and $0\le l<n$.

Since $F\in \Lambda$, we necessarily have $a_i = 0$ for all $i$, $b=1$, $c=2$, and $l=0$. That is,
\[
L_1 = n A + F\hspace{12pt}\text{and}\hspace{12pt} L_2 = F.
\]
Then $L_1 - L_2\in n\Lambda$, which is a contradiction.
\end{proof}

\bibliographystyle{alpha}

\begin{thebibliography}{BHPV04}

\bibitem[AH11]{ah}
Dan Abramovich and Brendan Hassett.
\newblock Stable varieties with a twist.
\newblock In {\em Classification of algebraic varieties}, EMS Ser. Congr. Rep.,
  pages 1--38. Eur. Math. Soc., Z\"urich, 2011.

\bibitem[AV02]{abramovichvistoli}
Dan Abramovich and Angelo Vistoli.
\newblock {Compactifying the space of stable maps}.
\newblock {\em Journal of the American Mathematical Society}, 15(1):27--75,
  2002.

\bibitem[Ach20]{achter}
Jeffrey~D. Achter.
\newblock Arithmetic occult period maps.
\newblock {\em Algebr. Geom.}, 7(5):581--606, 2020.

\bibitem[AK03]{ak}
Carolina Araujo and J\'anos Koll\'ar.
\newblock Rational curves on varieties.
\newblock In {\em Higher dimensional varieties and rational points ({B}udapest,
  2001)}, volume~12 of {\em Bolyai Soc. Math. Stud.}, pages 13--68. Springer,
  Berlin, 2003.

\bibitem[AC81]{arbcorn}
Enrico Arbarello and Maurizio Cornalba.
\newblock Footnotes to a paper of {B}eniamino {S}egre: ``{O}n the modules of
  polygonal curves and on a complement to the {R}iemann existence theorem''
  ({I}talian) [{M}ath. {A}nn. {\bf 100} (1928), 537--551; {J}buch {\bf 54},
  685].
\newblock {\em Math. Ann.}, 256(3):341--362, 1981.
\newblock The number of $g^{1}_{d}$'s on a general $d$-gonal curve, and the
  unirationality of the {H}urwitz spaces of $4$-gonal and $5$-gonal curves.

\bibitem[ACG11]{acgh2}
Enrico Arbarello, Maurizio Cornalba, and Pillip~A. Griffiths.
\newblock {\em Geometry of algebraic curves. {V}olume {II}}, volume 268 of {\em
  Grundlehren der Mathematischen Wissenschaften [Fundamental Principles of
  Mathematical Sciences]}.
\newblock Springer, Heidelberg, 2011.
\newblock With a contribution by Joseph Daniel Harris.

\bibitem[Art74]{artinsupersingular}
Michael Artin.
\newblock Supersingular {K}3 surfaces.
\newblock {\em Ann. Sci. \'{E}cole Norm. Sup. (4)}, 7:543--567 (1975), 1974.

\bibitem[BHPV04]{BHPV}
Wolf~P. Barth, Klaus Hulek, Chris A.~M. Peters, and Antonius Van~de Ven.
\newblock {\em Compact complex surfaces}, volume~4 of {\em Ergebnisse der
  Mathematik und ihrer Grenzgebiete. 3. Folge. A Series of Modern Surveys in
  Mathematics [Results in Mathematics and Related Areas. 3rd Series. A Series
  of Modern Surveys in Mathematics]}.
\newblock Springer-Verlag, Berlin, second edition, 2004.

\bibitem[Bea04]{beauville}
Arnaud Beauville.
\newblock Fano threefolds and {K}3 surfaces.
\newblock In {\em The {F}ano {C}onference}, pages 175--184. Univ. Torino,
  Turin, 2004.

\bibitem[Ben15]{benoist}
Olivier Benoist.
\newblock Construction de courbes sur les surfaces {K}3 (d'apr\`es
  {B}ogomolov-{H}assett-{T}schinkel, {C}harles, {L}i-{L}iedtke, {M}adapusi
  {P}era, {M}aulik{$\ldots$}).
\newblock {\em Ast\'erisque}, (367-368):Exp. No. 1081, viii, 219--253, 2015.

\bibitem[Ben20]{benoistss}
Olivier Benoist.
\newblock R\'{e}duction stable en dimension sup\'{e}rieure [d'apr\`es
  {K}oll\'{a}r, {H}acon-{X}u, ...].
\newblock {\em Ast\'{e}risque}, (422, S\'{e}minaire Bourbaki. Vol. 2018/2019.
  Expos\'{e}s 1151--1165):291--326, 2020.

\bibitem[BHT11]{bht}
Fedor Bogomolov, Brendan Hassett, and Yuri Tschinkel.
\newblock {Constructing rational curves on {K}3 surfaces}.
\newblock {\em Duke Mathematical Journal}, 157(3):535--550, April 2011.

\bibitem[BT00]{btdensity}
Fedor Bogomolov and Yuri Tschinkel.
\newblock {Density of rational points on elliptic K3 surfaces}.
\newblock {\em Asian Journal of Mathematics}, 4(2):351--368, 2000.

\bibitem[BT05]{btrational}
Fedor Bogomolov and Yuri Tschinkel.
\newblock Rational curves and points on {K}3 surfaces.
\newblock {\em Amer. J. Math.}, 127(4):825--835, 2005.

\bibitem[BZ09]{BZ}
Fedor Bogomolov and Yuri Zarhin.
\newblock Ordinary reduction of {K}3 surfaces.
\newblock {\em Cent. Eur. J. Math.}, 7(2):206--213, 2009.

\bibitem[Cha14]{charles}
Francois Charles.
\newblock {On the Picard number of {K}3 surfaces over number fields}.
\newblock {\em Algebra {\&} Number Theory}, 8(1):1--17, 2014.

\bibitem[Che99]{chen}
Xi~Chen.
\newblock Rational curves on {K}3 surfaces.
\newblock {\em J. Algebraic Geom.}, 8(2):245--278, 1999.

\bibitem[CGL22]{generic}
Xi~Chen, Frank Gounelas, and Christian Liedtke.
\newblock Rational curves on lattice-polarised {K}3 surfaces.
\newblock {\em Algebr. Geom.}, 9(4):443--475, 2022.

\bibitem[CL13]{C-L}
Xi~Chen and James~D. Lewis.
\newblock Density of rational curves on {K}3 surfaces.
\newblock {\em Math. Ann.}, 356(1):331--354, 2013.


\bibitem[CDL]{cdl}
Fran\c{c}ois Cossec, Igor Dolgachev, and Christian Liedtke.
\newblock {\em {E}nriques Surfaces, Volume 1}.
\newblock book in preparation.

\bibitem[CT14]{costatschinkel}
Edgar Costa and Yuri Tschinkel.
\newblock Variation of {N}\'{e}ron-{S}everi ranks of reductions of {K}3
  surfaces.
\newblock {\em Exp. Math.}, 23(4):475--481, 2014.

\bibitem[CEJ20]{cej}
Edgar Costa, Andreas-Stephan Elsenhans, and J\"{o}rg Jahnel.
\newblock On the distribution of the {P}icard ranks of the reductions of a
  {$K3$} surface.
\newblock {\em Res. Number Theory}, 6(3):[Paper No. 27, 25 pp.], 2020.

\bibitem[DS17]{dedieusernesi}
Thomas Dedieu and Edoardo Sernesi.
\newblock Equigeneric and equisingular families of curves on surfaces.
\newblock {\em Publ. Mat.}, 61(1):175--212, 2017.

\bibitem[Del81]{deligne}
Pierre Deligne.
\newblock Rel\`evement des surfaces {K}3 en caract\'eristique nulle.
\newblock In {\em Algebraic surfaces ({O}rsay, 1976--78)}, volume 868 of {\em
  Lecture Notes in Math.}, pages 58--79. Springer, Berlin-New York, 1981.
\newblock Prepared for publication by Luc Illusie.

\bibitem[DK]{dk}
Igor Dolgachev and Shigeyuki Kond\={o}.
\newblock {\em Enriques surfaces, Volume 2}.
\newblock book in preparation.

\bibitem[Dol96]{dolgachev}
Igor Dolgachev.
\newblock Mirror symmetry for lattice polarized {K}3 surfaces.
\newblock {\em J. Math. Sci.}, 81(3):2599--2630, 1996.
\newblock Algebraic geometry, 4.

\bibitem[EJ11]{elsenhansjahnel}
Andreas~Stephan Elsenhans and J{\"{o}}rg Jahnel.
\newblock {The {P}icard group of a {K}3 surface and its reduction modulo $\rho$}.
\newblock {\em Algebra and Number Theory}, 5(8):1027--1040, 2011.

\bibitem[FGI{\etalchar{+}}05]{fga}
Barbara Fantechi, Lothar G\"ottsche, Luc Illusie, Steven~L. Kleiman, Nitin
  Nitsure, and Angelo Vistoli.
\newblock {\em Fundamental algebraic geometry}, volume 123 of {\em Mathematical
  Surveys and Monographs}.
\newblock American Mathematical Society, Providence, RI, 2005.
\newblock Grothendieck's FGA explained.

\bibitem[FP97]{fp}
William Fulton and Rahul Pandharipande.
\newblock Notes on stable maps and quantum cohomology.
\newblock In {\em Algebraic geometry---{S}anta {C}ruz 1995}, volume~62 of {\em
  Proc. Sympos. Pure Math.}, pages 45--96. Amer. Math. Soc., Providence, RI,
  1997.

\bibitem[vdGK00]{katsuravandergeer}
Gerard van~der Geer and Toshiyuki Katsura.
\newblock {On a stratification of the moduli of {K}3 surfaces}.
\newblock {\em Journal of the European Mathematical Society}, 2(3):259--290,
  aug 2000.

\bibitem[GHS03]{ghs}
Tom Graber, Joseph Harris, and Jason Starr.
\newblock {Families of rationally connected varieties}.
\newblock {\em Journal of the American Mathematical Society}, 16(1):57--67,
  2003.

\bibitem[Har86]{harris}
Joe Harris.
\newblock On the {S}everi problem.
\newblock {\em Invent. Math.}, 84(3):445--461, 1986.

\bibitem[HM98]{harrismorrison}
Joe Harris and Ian Morrison.
\newblock {\em Moduli of curves}, volume 187 of {\em Graduate Texts in
  Mathematics}.
\newblock Springer-Verlag, New York, 1998.

\bibitem[Has03]{hassett}
Brendan Hassett.
\newblock Potential density of rational points on algebraic varieties.
\newblock In {\em Higher dimensional varieties and rational points ({B}udapest,
  2001)}, volume~12 of {\em Bolyai Soc. Math. Stud.}, pages 223--282. Springer,
  Berlin, 2003.

\bibitem[Huy16]{huybrechts}
Daniel Huybrechts.
\newblock {\em Lectures on {K}3 surfaces}, volume 158 of {\em Cambridge Studies
  in Advanced Mathematics}.
\newblock Cambridge University Press, Cambridge, 2016.

\bibitem[HK13]{hk}
Daniel Huybrechts and Michael Kemeny.
\newblock {Stable maps and {C}how groups}.
\newblock {\em Documenta Mathematica}, 18:507--517, 2013.

\bibitem[Ito20]{Ito}
Kazuhiro Ito.
\newblock On the supersingular reduction of {$K3$} surfaces with complex
  multiplication.
\newblock {\em Int. Math. Res. Not. IMRN}, (20):7306--7346, 2020.

\bibitem[IIL20]{iil}
Kazuhiro Ito, Tetsushi Ito, and Christian Liedtke.
\newblock Deformations of rational curves in positive characteristic.
\newblock {\em J. Reine Angew. Math.}, 769:55--86, 2020.

\bibitem[Kem15]{kemenythesis}
Michael Kemeny.
\newblock {Stable maps and singular curves on {K}3 surfaces}.
\newblock {\em arXiv e-prints}, page arXiv:1507.00230, July 2015.

\bibitem[KLC22]{klc}
Andreas~Leopold Knutsen and Margherita Lelli-Chiesa.
\newblock Genus two curves on abelian surfaces.
\newblock {\em Ann. Sci. \'{E}c. Norm. Sup\'{e}r. (4)}, 55(4):905--918, 2022.

\bibitem[KLCM19]{klcm}
Andreas~Leopold Knutsen, Margherita Lelli-Chiesa, and Giovanni Mongardi.
\newblock Severi varieties and {B}rill-{N}oether theory of curves on abelian
  surfaces.
\newblock {\em J. Reine Angew. Math.}, 749:161--200, 2019.

\bibitem[Kol89]{kollarflops}
J\'{a}nos Koll\'{a}r.
\newblock Flops.
\newblock {\em Nagoya Math. J.}, 113:15--36, 1989.

\bibitem[Kol96]{kollar}
J\'anos Koll\'ar.
\newblock {\em Rational curves on algebraic varieties}, volume~32 of {\em
  Ergebnisse der Mathematik und ihrer Grenzgebiete. 3. Folge. A Series of
  Modern Surveys in Mathematics [Results in Mathematics and Related Areas. 3rd
  Series. A Series of Modern Surveys in Mathematics]}.
\newblock Springer-Verlag, Berlin, 1996.

\bibitem[KS14]{kondoshimada}
Shigeyuki Kondo and Ichiro Shimada.
\newblock On a certain duality of {N}\'{e}ron-{S}everi lattices of
  supersingular {K}3 surfaces.
\newblock {\em Algebr. Geom.}, 1(3):311--333, 2014.

\bibitem[Kon95]{kontsevich}
Maxim Kontsevich.
\newblock Enumeration of rational curves via torus actions.
\newblock In {\em The moduli space of curves ({T}exel {I}sland, 1994)}, volume
  129 of {\em Progr. Math.}, pages 335--368. Birkh\"{a}user Boston, Boston, MA,
  1995.

\bibitem[Kov13]{kovacsrevisited}
S{\'{a}}ndor~J. Kov{\'{a}}cs.
\newblock {The cone of curves of {K}3 surfaces revisited}.
\newblock {\em Birational Geometry, Rational Curves, and Arithmetic},
  1:163--169, 2013.

\bibitem[LL11]{liliedtke}
Jun Li and Christian Liedtke.
\newblock {Rational curves on {K}3 surfaces}.
\newblock {\em Inventiones mathematicae}, 188(3):713--727, October 2011.

\bibitem[LM18a]{lieblichmaulik}
Max Lieblich and Davesh Maulik.
\newblock A note on the cone conjecture for {K}3 surfaces in positive
  characteristic.
\newblock {\em Math. Res. Lett.}, 25(6):1879--1891, 2018.

\bibitem[LM18]{liedtkematsumoto}
Christian Liedtke and Yuya Matsumoto.
\newblock Good reduction of {K}3 surfaces.
\newblock {\em Compos. Math.}, 154(1):1--35, 2018.

\bibitem[MM83]{morimukai}
Shigefumi Mori and Shigeru Mukai.
\newblock The uniruledness of the moduli space of curves of genus {$11$}.
\newblock In {\em Algebraic geometry ({T}okyo/{K}yoto, 1982)}, volume 1016 of
  {\em Lecture Notes in Math.}, pages 334--353. Springer, Berlin, 1983.

\bibitem[Nik87]{nikulin}
Vyacheslav V. Nikulin.
\newblock Discrete reflection groups in {L}obachevsky spaces and algebraic
  surfaces.
\newblock In {\em Proceedings of the {I}nternational {C}ongress of
  {M}athematicians, {V}ol. 1, 2 ({B}erkeley, {C}alif., 1986)}, pages 654--671.
  Amer. Math. Soc., Providence, RI, 1987.

\bibitem[Nik14]{nikulinelliptic}
Viacheslav V. Nikulin.
\newblock Elliptic fibrations on {K}3 surfaces.
\newblock {\em Proc. Edinb. Math. Soc. (2)}, 57(1):253--267, 2014.

\bibitem[{Nis}19]{nishinou}
Takeo {Nishinou}.
\newblock {Obstructions to deforming maps from curves to surfaces}.
\newblock {\em arXiv e-prints}, page arXiv:1901.11239, January 2019.

\bibitem[NO85]{NO}
Niels Nygaard and Arthur Ogus.
\newblock Tate's conjecture for {K}3 surfaces of finite height.
\newblock {\em Ann. of Math. (2)}, 122(3):461--507, 1985.

\bibitem[Ogu79]{ogus}
Arthur Ogus.
\newblock Supersingular {K}3 crystals.
\newblock In {\em Journ\'{e}es de {G}\'{e}om\'{e}trie {A}lg\'{e}brique de
  {R}ennes ({R}ennes, 1978), {V}ol. {II}}, volume~64 of {\em Ast\'{e}risque},
  pages 3--86. Soc. Math. France, Paris, 1979.

\bibitem[Ols16]{olsson}
Martin Olsson.
\newblock {\em Algebraic spaces and stacks}, volume~62 of {\em American
  Mathematical Society Colloquium Publications}.
\newblock American Mathematical Society, Providence, RI, 2016.

\bibitem[Riz06]{rizov}
Jordan Rizov.
\newblock Moduli stacks of polarized {K}3 surfaces in mixed characteristic.
\newblock {\em Serdica Math. J.}, 32(2-3):131--178, 2006.

\bibitem[R{\v{S}}05]{rudakovshafarevich2}
Alexei N. Rudakov and  Igor R. {\v{S}}afarevi{\v{c}}.
\newblock {Supersingular K3 Surfaces Over Fields of Characteristic 2}.
\newblock {\em Mathematics of the USSR-Izvestiya}, 13(1):147--165, 2005.

\bibitem[Ser06]{sernesi}
Edoardo Sernesi.
\newblock {\em Deformations of algebraic schemes}, volume 334 of {\em
  Grundlehren der Mathematischen Wissenschaften [Fundamental Principles of
  Mathematical Sciences]}.
\newblock Springer-Verlag, Berlin, 2006.

\bibitem[SSTT22]{tayouetal}
Ananth~N. Shankar, Arul Shankar, Yunqing Tang, and Salim Tayou.
\newblock Exceptional jumps of {P}icard ranks of reductions of {K}3 surfaces
  over number fields.
\newblock {\em Forum Math. Pi}, 10:Paper No. e21, 2022.

\bibitem[Shi04]{shimada}
Ichiro Shimada.
\newblock {Supersingular {K}3 surfaces in odd characteristic and sextic double
  planes}.
\newblock {\em Mathematische Annalen}, 328(3):451--468, 2004.

\bibitem[Shi74]{ShiodaExample}
Tetsuji Shioda.
\newblock An example of unirational surfaces in characteristic {$p$}.
\newblock {\em Math. Ann.}, 211:233--236, 1974.

\bibitem[Shi77]{ShiodaKummer}
Tetsuji Shioda.
\newblock Some results on unirationality of algebraic surfaces.
\newblock {\em Math. Ann.}, 230(2):153--168, 1977.

\bibitem[{Sta}19]{stacks}
The {Stacks Project Authors}.
\newblock \textit{Stacks Project}.
\newblock \url{https://stacks.math.columbia.edu}, 2019.

\bibitem[Tay20]{Tayou}
Salim Tayou.
\newblock Rational curves on elliptic {K}3 surfaces.
\newblock {\em Math. Res. Lett.}, 27(4):1237--1248, 2020.

\bibitem[Tot17]{totaro}
Burt Totaro.
\newblock Recent progress on the {T}ate conjecture.
\newblock {\em Bull. Amer. Math. Soc. (N.S.)}, 54(4):575--590, 2017.

\end{thebibliography}
\newcommand{\etalchar}[1]{$^{#1}$}

\end{document}